\documentclass [leqno, spanish, 10pt]{amsart}

\usepackage[utf8]{inputenc}

\usepackage{amssymb, amsmath, amsfonts, amsthm, amsbsy, amscd, amsxtra, stmaryrd}
\usepackage[bbgreekl]{mathbbol} 
\usepackage[mathscr]{euscript}
\usepackage{upgreek}
\usepackage{mathtools}	

\usepackage{url}
\usepackage{hyperref}
\hypersetup{
	colorlinks=false, linktocpage=true, hidelinks
}
\usepackage{bookmark}

\usepackage{thmtools}
\newtheorem{theorem}{Theorem}[section]

\newtheorem{corollary}[theorem]{Corollary}
\newtheorem{lemma}[theorem]{Lemma}

\theoremstyle{definition}

\declaretheoremstyle[
  headfont=\normalfont\bfseries,
  sharenumber = theorem,
  bodyfont=\normalfont,
  qed={$\boxtimes$},
]{examplestyle2}

\declaretheorem[
  style=examplestyle2,
  title=Example,
  refname={example,examples},
  Refname={Example,Examples}
]{example}

\declaretheorem[
  style=examplestyle2,
  title=Remark,
  refname={remark, remarks},
  Refname={Remark, Remarks}
]{remark}

\numberwithin{equation}{section}

\newcommand{\qp}{Q}
\newcommand{\jrd}{\times}
\newcommand{\ctimes}{\bar{\jrd}}

\newcommand{\mdual}{\mathcal{T}}
\newcommand{\pair}{\varrho}
\newcommand{\prepair}{\pi}

\newcommand{\X}{\mathbb{X}}
\newcommand{\sC}{\mathscr{C}}
\newcommand{\cliff}{Cl}

\newcommand{\jstar}{\star}
\newcommand{\jmult}{\circledcirc}
\newcommand{\sprod}{\odot}

\newcommand{\tprod}{\diamond}
\newcommand{\gprod}{\square}

\newcommand{\kwedge}{\owedge}
\newcommand{\dwedge}{\wedge}
\newcommand{\join}{\wedge}

\newcommand{\dbw}{D}

\newcommand{\bw}{\Phi}

\newcommand{\stwoproj}{\mathscr{M}}
\newcommand{\stwoprojdual}{\mathscr{S}}
\newcommand{\oc}{\mathscr{OC}}

\newcommand{\ext}{\mathchoice{{\textstyle\bigwedge}}%
    {{\bigwedge}}%
    {{\textstyle\dwedge}}%
    {{\scriptstyle\dwedge}}}
\newcommand{\tf}{\operatorname{\mathsf{tf}}}
\newcommand{\sR}{\mathscr{R}}
\newcommand{\stu}{\mathbb{U}}

\newcommand{\stwd}{\stw^{\ast}}

\newcommand{\opmult}{\ast}
\newcommand{\opmults}{\ast_{S}}
\newcommand{\opmulta}{\ast_{A}}
\newcommand{\rictr}{\operatorname{\rho}}
\newcommand{\rictrz}{\rictr_{\circ}}
\newcommand{\scal}{\operatorname{s}}

\newcommand{\sB}{\mathscr{B}}

\newcommand{\sX}{\mathscr{X}}
\newcommand{\sY}{\mathscr{Y}}
\newcommand{\sS}{\mathscr{S}}
\newcommand{\sK}{\mathscr{K}}
\newcommand{\sE}{\mathscr{E}}
\newcommand{\sH}{\mathscr{H}}

\newcommand{\sZ}{\mathscr{Z}}
\newcommand{\mcurv}{\mathscr{MC}}
\newcommand{\oproj}{\mathscr{P}}
\newcommand{\mcurvsym}{\overline{\mathscr{MC}}}
\newcommand{\mcurvkahler}{\mathscr{MC}_{\mathscr{K}}}
\newcommand{\mcurvweyl}{\mathscr{MC}_{\mathscr{W}}}
\newcommand{\projric}{\mathscr{P}_{\mathscr{R}}}
\newcommand{\projscal}{\mathscr{P}_{\mathscr{S}}}

\newcommand{\mcurvric}{\mathscr{MC}_{\mathscr{R}}}
\newcommand{\mcurvscal}{\mathscr{MC}_{\mathscr{S}}}

\newcommand{\mcurvkahlerweyl}{\mathscr{MC}_{\mathscr{K},\mathscr{W}}}
\newcommand{\mcurvweylpm}{\mathscr{MC}^{\pm}_{\mathscr{W}}}
\newcommand{\mcurvweylp}{\mathscr{MC}^{+}_{\mathscr{W}}}
\newcommand{\mcurvweylm}{\mathscr{MC}^{-}_{\mathscr{W}}}
\newcommand{\Id}{\operatorname{Id}}

\newcommand{\stw}{\mathbb{W}}

\newcommand{\std}{\ste^{\ast}}

\newcommand{\sdf}{\ext^{2}_{+}}

\newcommand{\asdf}{\ext^{2}_{-}}
\newcommand{\pmdf}{\ext^{2}_{\pm}}
\newcommand{\mpdf}{\ext^{2}_{\mp}}

\newcommand{\idem}{\mathbb{Idem}}

\newcommand{\sherm}{\operatorname{Sym}}
\newcommand{\sa}{\operatorname{Sym}}

\newcommand{\dalg}{\mathbb{D}}

\newcommand{\id}{\operatorname{Id}}

\newcommand{\ealg}{\mathbb{E}}

\newcommand{\ideal}{\mathbb{I}}
\newcommand{\jdeal}{\mathbb{J}}
\newcommand{\fie}{\mathbb{k}}

\newcommand{\imt}{\iota}

\newcommand{\so}{\mathfrak{so}}

\newcommand{\alg}{\mathbb{A}}

\newcommand{\balg}{\mathbb{B}}

\newcommand{\om}{\omega}
\newcommand{\mlt}{\circ}

\newcommand{\ka}{\kappa}

\newcommand{\dum}{\,\cdot\,\,}

\newcommand{\lap}{\Delta}
\renewcommand{\j}{\mathsf{i}}
\newcommand{\la}{\lambda}
\newcommand{\ep}{\epsilon}
\newcommand{\reat}{\mathbb{R}^{\times}}

\newcommand{\eno}{\operatorname{End}}
\newcommand{\si}{\sigma}

\newcommand{\integer}{\mathbb{Z}}

\def\op#1{\widehat{#1}}

\newcommand{\lb}{\langle}
\newcommand{\ra}{\rangle}

\newcommand{\ste}{\mathbb{V}}

\newcommand{\A}{\mathcal{A}}
\newcommand{\al}{\alpha}
\newcommand{\be}{\beta}
\newcommand{\ga}{\gamma}
\newcommand{\spn}{\text{Span}\,}

\newcommand{\proj}{\mathbb{P}}
\newcommand{\sym}{\mathsf{Sym}\,}
\newcommand{\asym}{\mathsf{Skew}\,}

\DeclareMathOperator{\Aut}{Aut}

\newcommand{\ad}{\operatorname{ad}}

\newcommand{\tensor}{\otimes}

\newcommand{\rea}{\mathbb R}

\newcommand{\tr}{\operatorname{\mathsf{tr}}}

\let\oldtocsection=\tocsection
\let\oldtocsubsection=\tocsubsection
\renewcommand{\tocsection}[2]{\hspace{0em}\oldtocsection{#1}{#2}}
\renewcommand{\tocsubsection}[2]{\hspace{1em}\oldtocsubsection{#1}{#2}} 

\begin{document}
\title[Algebra of curvature tensors]{The commutative nonassociative algebra of metric curvature tensors}
\author{Daniel J.~F. Fox} 
\address{Departamento de Matemática Aplicada a la Ingeniería Industrial\\ Escuela Técnica Superior de Ingeniería y Diseño Industrial\\ Universidad Politécnica de Madrid\\Ronda de Valencia 3\\ 28012 Madrid España}
\email{daniel.fox@upm.es}

\begin{abstract}
The space of tensors of metric curvature type on a Euclidean vector space carries a two-parameter family of orthogonally invariant commutative nonassociative multiplications invariant with respect to the symmetric bilinear form determined by the metric. For a particular choice of parameters these algebras recover the polarization of the quadratic map on metric curvature tensors that arises in the work of Hamilton on the Ricci flow. Here these algebras are studied as interesting examples of metrized commutative algebras and in low dimensions they are described concretely in terms of nonstandard commutative multiplications on self-adjoint endomorphisms.

The algebra of curvature tensors on a three-dimensional Euclidean vector space is shown isomorphic to an orthogonally invariant deformation of the standard Jordan product on three by three symmetric matrices. This algebra is characterized up to isomorphism in terms of purely algebraic properties of its idempotents and the spectra of their multiplication operators.

On a vector space of dimension at least four, the subspace of Weyl (Ricci-flat) curvature tensors is a subalgebra for which the multiplication endomorphisms are trace-free and the Killing type trace-form is a multiple of the nondegenerate invariant metric. This subalgebra is simple when the Euclidean vector space has dimension greater than four. In the presence of a compatible complex structure, the analogous result is obtained for the subalgebra of Kähler Weyl curvature tensors. It is shown that the anti-self-dual Weyl tensors on a four-dimensional vector space form a simple five-dimensional ideal isometrically isomorphic to the trace-free part of the Jordan product on trace-free $3 \times 3$ symmetric matrices.
\end{abstract}

\maketitle
\setcounter{tocdepth}{1} 
\begin{footnotesize}
\tableofcontents
\end{footnotesize}

\section{Introduction}
Let $\ste$ be an $n$-dimensional real vector space equipped with a metric $h_{ij}$. Let $\mcurv(\std)$ 
\begin{align}
\mcurv(\std) = \{\sY_{ijkl} \in \tensor^{4}\std: \sY_{[ij]kl} = \sY_{ijkl} = \sY_{ij[kl]}, \sY_{[ijk]l} = 0\}
\end{align}
be the $n^{2}(n^{2} - 1)/12$-dimensional vector space of \emph{metric curvature tensors}.
Any $\sY_{ijkl} \in \mcurv(\std)$ satisfies $\sY_{klij} = \sY_{ijkl}$ and $\sY_{i(jk)l}$ is symmetric in $i$ and $l$. 
The metric curvature tensors of \emph{n type} $\mcurvweyl(\std)$ comprise the kernel of the \emph{Ricci trace} $\rictr:\mcurv(\std) \to S^{2}(\std)$ defined by $\rictr(\sY)_{ij} = \sY_{pij}\,^{p}$. 
Note that $\rictr(\sY)_{ij}$ is symmetric because $2\rictr(\sY)_{[ij]} = -\sY_{ijkl}h^{kl} = 0$. The trace $\scal(\sY) = \tr \rictr(\sY) = h^{ij}\rictr(\sY)_{ij}$ is the \emph{scalar curvature} of $\sY$. (Here, and when convenient, there are used the abstract index conventions \cite[chapter $2$]{Penrose-Rindler}.) These definitions are consistent with the conventions in which the curvature tensor of the round metric $g_{ij}$ on the sphere has the form $-2g_{k[i}g_{j]l}$ (see Remark \ref{signremark} for detailed discussion of signs).

By \cite[Theorem $7.1$]{Hamilton} the curvature tensor $\sR_{ijkl}$ of a family of metrics $g(t)_{ij}$ solving the Ricci flow $\tfrac{d}{dt}g(t)_{ij} = -2\rictr(\sR(t))_{ij}$ evolves according to
\begin{align}\label{ricciflowcurv}
\tfrac{d}{dt}\sR_{ijkl} = \lap\sR_{ijkl} + 2(\sR\opmult \sR)_{ijkl} + 2\sR_{p[i}\sR_{j]}\,^{p}\,_{kl} + 2\sR_{p[k}\sR_{l]}\,^{p}\,_{ij},
\end{align}
where $\sR\opmult \sR$ is some quadratic form on $\mcurv(\std)$. The polarization of the quadratic form appearing in \eqref{ricciflowcurv} can be viewed as a commutative multiplication $\opmult$ on $\mcurv(\std)$. Here, $(\mcurv(\std), \opmult)$ is studied as an interesting example in the general context of commutative nonassociative algebras that exhibits some special structural properties. Although they did not use explicitly this algebraic perspective, it was R. Hamilton  \cite{Hamilton, Hamilton-four, Hamilton-formation} and G. Huisken \cite{Huisken} who first emphasized the importance of $\opmult$ and discovered its basic properties.

The class of commutative not necessarily associative algebras with no additional structure is too general to admit a good theory. In many interesting examples the commutative algebra $(\alg, \mlt)$ satisfies the further condition that it is \emph{metrized} meaning it is equipped with a nondegenerate bilinear form $h$ that is \emph{invariant} in the sense that the cubic form $h(x\mlt y, z)$ is completely symmetric in $x, y, z \in \alg$ (in this case $h$ is also often called a \emph{Frobenius form}). For various perspectives on metrized commutative algebras see \cite{Bordemann, Fox-simplicial, Griess-Monster, Griess-gnavoai, Hall-transpositionalgebras, Hall-Rehren-Shpectorov, Ivanov, Nadirashvili-Tkachev-Vladuts, Tkachev-correction}. 

The definition of $\opmult$ and its basic properties are described in Section \ref{curvmultsection}, and are based on Theorem \ref{opalgebratheorem}, which yields two different new constructions of $\opmult$.

By Lemma \ref{curvoplemma}, for $k \geq 2$ there is an $O(n)$-equivariant linear map $\sX \in \mcurv(\std) \to \op{\sX} \in \eno(\tensor^{k}\std)$ such that $\op{\sX}$ is self-adjoint and preserves the type (by symmetries) of tensors. If $\op{\sX}$ preserves the $O(n)$-submodule $\stw \subset \tensor^{k}\std$, it restriction to $\stw$ is written $\op{\sX}_{\stw}$. If $\sX \to \op{\sX}_{\stw}$ is injective, the pullback of the projection onto the image of $\op{\dum}_{\stw}$ of the Jordan product $\op{\sX}_{\stw}\jmult \op{\sY}_{\stw}$ yields a commutative multiplication on $\mcurv(\std)$ on which $O(n)$ acts by automorphisms. 

For example, $\op{\sX}$ preserves $\ext^{2}\std$ and $S^{2}\std$ and the induced maps $\op{\dum}_{\ext^{2}\std}$ and $\op{\dum}_{S^{2}\std}$ are injective by Corollary \ref{traceextcorollary}. By Lemma \ref{stalgebralemma}, the linear combinations of the pullbacks of the projections onto their images of the Jordan products of endomorphisms, $\op{\sX}_{\ext^{2}\std}\jmult\op{\sY}_{\ext^{2}\std}$ and $\op{\sX}_{S^{2}\std}\jmult\op{\sY}_{S^{2}\std}$, yield a two-parameter family $s\opmulta + t\opmults$ of commutative multiplications on $\mcurv(\std)$ that are metrized by the metric $\lb \sX, \sY \ra = \sX_{ijkl}\sY^{ijkl}$ on $\mcurv(\std)$ and on which $O(n)$ acts isometrically by algebra automorphisms. Moreover, a specific linear combination recovers $\opmult$ as follows. As $\sX\in \mcurv(\std)$ determines an endomorphism $\op{\sX}_{\mcurv(\std)}$ of $\mcurv(\std)$, it makes sense to define a multiplication on $\mcurv(\std)$ by $\sX \opmult \sY = \op{\sX}_{\mcurv(\std)}(\sY)$. It turns out that the multiplication $\opmult$ so defined is commutative, for Theorem \ref{opalgebratheorem} shows that $\opmult = \tfrac{3}{2}(\opmulta + \opmults)$, and that it recovers the multiplication of \eqref{ricciflowcurv}, for it shows that $\opmult$ has the explicit form \eqref{curvmultdefined} found by Hamilton. 
Because $\sX \opmult \sY = \op{\sX}(\sY)$, an immediate consequence of the self-adjointness of $\op{\sX}_{\mcurv(\std)}$ with respect to $\lb \dum, \dum \ra$ is that $(\mcurv(\std), \opmult)$ is metrized by $\lb \dum, \dum \ra$, a fact due to Huisken.

\begin{remark}\label{pullbackremark}
For any metrized commutative algebra $(\alg, \mlt, h)$ and any $t \in \reat$, $t\Id_{\alg} \in \eno(\alg)$ is an isometric algebra isomorphism from $(\alg, \mlt_{t}, t^{2}h)$ to $(\alg, \mlt, h)$ where $x\mlt_{t} y = t x \mlt y$. For this reason, the family $s\opmulta + t\opmults$ should be regarded as associated with $[s:t] \in \proj^{1}(\rea)$, and it is any one of the multiplications corresponding with $[1:1] \in \proj^{1}(\rea)$ that arises in the Ricci flow, the choice of which amounting to a normalization that is inconsequential from a purely algebraic perspective. However, considerations related to geometric applications motivate a particular choice. Concretely, the choice of $\opmult = \opmult_{1}$ over $\opmult_{-1}$ is made by requiring that a \emph{positive} multiple of the curvature tensor of the round sphere be idempotent. See Remark \ref{signremark} for further discussion.
\end{remark}

\begin{remark}
Some authors \cite{Bohm-Wilking, Richard-curvaturecones, Richard-Seshadri} define $\opmult$ directly in terms of curvature operators on $\ext^{2}\std$. Here $\opmult$ is defined on curvature tensors, and the two definitions involve curvature operators on $S^{2}\std$ and $\mcurv(\std)$ itself. Although there seems to be no good notion of representation of a commutative nonassociative algebra (at least not without embedding it in a vertex operator algebra), it is convenient to think of curvature operators on different tensor modules such as $S^{2}\std$ and $\ext^{2}\std$ as different representations of $(\mcurv(\std), \opmult)$, Theorem \ref{opalgebratheorem} showing that the multiplication itself is determined by $\mcurv(\std)$ somehow viewed as a module over itself.
\end{remark}

The irreducible submodules of $\mcurv(\std)$ under the action of certain groups of orthogonal transformations are subalgebras.
Lemma \ref{opweylsubalgebralemma} shows that the space $\mcurvweyl(\std)$ of metric curvature tensors of Weyl type is a subalgebra of $(\mcurv(\std), \opmult)$.
If $(\ste, h)$ carries an almost complex structure compatible with $h$ it makes sense to speak of the submodule of Kähler curvature tensors (see Section \ref{kahlersection} for the definition), $\mcurvkahler(\std)$ and its submodule of Kähler Weyl curvature tensor $\mcurvkahlerweyl(\std) = \mcurvkahler(\std) \cap \mcurvweyl(\std)$, and Lemma \ref{kahlersubalgebralemma} shows that $\mcurvkahler(\std)$ and $\mcurvkahlerweyl(\std)$ are subalgebras of $\mcurv(\std)$.
The one-dimensional submodule of $\mcurv(\std)$ generated by the metric is also a subalgebra (isomorphic to the real field), but the irreducible submodule generated by the Kulkarni-Nomizu products of the metric with trace-free symmetric two-tensors (the submodule comprising curvature tensors of pure trace-free Ricci type) is not a subalgebra. The fusion rules (in the sense of \cite{Hall-Rehren-Shpectorov}) describing the interactions of the irreducible summands of $\mcurv(\std)$ are given in Table \ref{fusionrule}. They follow from Theorem \ref{fusiontheorem}, which gives more information than do the fusion rules alone because it asserts the equalities of products of subspaces, rather than simply containment relations. The proofs of these relations are based on detailed calculations of products in $(\mcurv(\std), \opmult)$, given in Section \ref{productssection}, that, while technical, should be useful in further study of $\opmult$. The ingredients of the proof of Theorem \ref{fusiontheorem} also yield a conceptually simple proof of the Böhm-Wilking theorem (see Section \ref{fusionsection}) used in the construction of curvature cones. The fusion rules for the unitary irreducible subspaces of the subalgebra $\mcurvkahler(\std)$ and the corresponding analogue of the Böhm-Wilking theorem are described in the companion paper \cite{Fox-curvtensorkahler}.

A commutative algebra $(\alg, \mlt)$ is \emph{exact} (called \emph{harmonic} in \cite{Nadirashvili-Tkachev-Vladuts}) if its multiplication endomorphisms $L_{\mlt}:\alg \to \eno(\alg)$ defined by $L_{\mlt}(x) = x\mlt y$ satisfy $\tr L_{\mlt}(x) = 0$ for all $x \in \alg$. Note that an exact algebra is nonunital.
A commutative algebra $(\alg, \mlt)$ is \emph{Killing metrized}, if the Killing type trace form $\tau_{\mlt}(x, y) = \tr L_{\mlt}(x)L_{\mlt}(y)$ is nondegenerate and invariant. The multiplication of a Killing metrized commutative algebra is necessarily faithful, meaning that $L_{\mlt}$ is injective. 

Important structural features of the subalgebra $(\mcurvweyl(\std), \opmult)$ shown in Theorem \ref{einsteintheorem}, are that it is exact and Killing metrized, and is simple when $\dim \std > 4$. 

\begin{theorem}\label{einsteintheorem}
Let $(\ste, h)$ be a Euclidean vector space of dimension at least $4$. The algebra $(\mcurvweyl(\std), \opmult)$ is exact and Killing metrized. Moreover:
\begin{enumerate}
\item The Killing form $\tau_{\opmult}(\sX, \sY) = \tr L_{\opmult}(\sX)L_{\opmult}(\sY)$ is a nonzero multiple of $\lb\dum, \dum \ra$. 
\item If $\dim \ste > 4$, then $(\mcurvweyl(\std), \opmult)$ is simple.
\end{enumerate}
\end{theorem}

\begin{proof}[Proof of Theorem \ref{einsteintheorem}]
Let $(\ste, h)$ be a Euclidean vector space with $\dim \ste  = n \geq 4$. The group $O(n) = O(\ste, h)$ acts on $\mcurvweyl(\std)$ isometrically and irreducibly. By Theorem \ref{opalgebratheorem}, $(\mcurvweyl(\std), \opmult)$ is metrized by the pairing $\lb \dum, \dum \ra$ and $O(n)$ acts on $(\mcurvweyl(\std), \opmult)$ by algebra automorphisms.
By Lemma \ref{weylidempotentslemma} there is a nontrivial idempotent $\sE \in (\mcurvweyl(\std), \opmult)$.
Because $(\mcurvweyl(\std), \opmult)$ contains a nontrivial idempotent, its multiplication is nontrivial. Theorem \ref{preeinsteintheorem} implies $\tr L_{\opmult}(\sX) = 0$, and $\tau_{\opmult}(\sX, \sY) = \tr L_{\opmult}(\sX)L_{\opmult}(\sY)$ equals $\ka \lb \dum, \dum \ra$ for some nonzero $\ka$ which must be positive because both $\tau_{\opmult}$ and $\lb\dum, \dum \ra$ are positive definite.
If $\dim \ste > 4$, the action by automorphisms of the connected simple Lie group $SO(n)$ on $\mcurvweyl(\std)$ is irreducible, so Theorem \ref{simpletheorem} implies $(\mcurvweyl(\std), \opmult)$ is simple. 
\end{proof}

When $\dim \ste = 4$, a choice of orientation determines an orthogonal decomposition $\mcurvweyl(\std) = \mcurvweylp(\std) \oplus \mcurvweylm(\std)$ where $\mcurvweylpm(\std)$ are the subspaces of self-dual and anti-self-dual curvature tensors. Theorem \ref{mcurvweyltheorem}, discussed in more detail later in the introduction, shows that these are mutually isomorphic subalgebras that are simple, exact, and Killing metrized with Killing form equal to $\tfrac{21}{16}\lb\dum, \dum \ra$. 
Alternatively this is a consequence of Theorem \ref{einsteinkahlertheorem}, which is the analogue of Theorem \ref{einsteintheorem} for the algebra of Kähler-Weyl tensors. It shows that if $(\ste, h, J)$ is a $2n$-dimensional Kähler vector space, then $(\mcurvkahlerweyl(\std), \opmult)$ is a simple, exact, Killing metrized algebra with Killing form a positive multiple of $\lb\dum, \dum \ra$. When $\dim \ste = 4$, a choice of compatible almost complex structure $J$ determines an orientation of $\ste$ and Lemma \ref{kahlerweylasdlemma} shows $\mcurvweylm(\std) = \mcurvkahlerweyl(\std)$

\begin{remark}
When $\dim \ste > 4$, Theorem \ref{einsteintheorem} does not give the value of the positive constant $\ka$ such that $\tau_{\opmult} = \ka \lb \dum, \dum \ra$. To calculate $\ka$ it would suffice to calculate the eigenvalues on $\mcurvweyl(\std)$ of the operator $\op{\sE}$ associated with a nonzero idempotent, as this suffices to calculate its $\tau_{\opmult}$-norm. When $\dim \ste = 4$, the explicit calculations used to prove Theorem \ref{mcurvweyltheorem} make it possible to calculate $\ka = 21/16$ for $(\mcurvweyl^{\pm}(\std), \opmult)$ (and so also for $(\mcurvkahlerweyl(\std), \opmult)$).
\end{remark}

A basic problem is to describe $(\mcurv(\std), \opmult)$, or its subalgebras more explicitly, in terms of known algebras. As mentioned already, when $\dim \ste = 2$, the one-dimensional algebra $(\mcurv(\std), \opmult)$ is isometrically isomorphic to the field of real numbers with its Euclidean inner product. When $\dim \ste$ is $3$ or $4$ explicit results are obtained relating $\opmult$ to the usual Jordan product of symmetric endomorphisms.

V.~L. Popov's \cite{Popov} discusses invariants of algebras constructed from traces of products of powers of their multiplication operators, addressing questions such as when does a module for a group $G$ admit a nontrivial $G$-invariant multiplication that is simple or have automorphism group equal to $G$. Specific instances of this last question are addressed in \cite{Dixmier-algebres}, for $G = SL(2)$, and \cite{Elashvili}, for certain exceptional Lie groups. 
In this context, metrizability by some particular trace-form, for example Killing metrizability, appears as a structurally important condition. Its importance has been explicitly indicated in work of A. Ryba, for example \cite{Ryba}, constructing commutative nonassociative algebras on which certain finite simple groups act by automorphisms (see in particular \cite[Lemma $9.1$]{Ryba} and see also \cite{Ivanov}), and in the work of V.~G. Tkachev and collaborators dedicated to a general program, detailed in \cite{Nadirashvili-Tkachev-Vladuts}, of constructing homogeneous solutions to certain geometrically motivated linear and fully nonlinear elliptic partial differential equations, for example those describing minimal cones, by studying the algebras associated with completely symmetric cubic forms.
An interesting class of examples of exact Killing metrized commutative nonassociative algebras, relevant here also for the statement of Theorem \ref{mcurvweyltheorem} below, are the deunitalizations of the finite-dimensional simple real Euclidean Jordan algebras. 

The vector space $\sherm(\stw, g)$ of $g$-self-adjoint endomorphisms of the $n$-dimensional Euclidean vector space $(\stw, g)$ equipped with the multiplication $\jmult$ that is the symmetric part of the ordinary composition of endomorphisms is an $n(n+1)/2$-dimensional simple real Euclidean Jordan algebra with unit. Its \emph{deunitalization} is the $(n+2)(n-1)/2$-dimensional commutative, nonassociative, nonunital algebra obtained by retraction along the unit. Precisely, this is the algebra $\sherm_{0}(\stw, g) = \{A \in \sherm(\stw, g): \tr A = 0\}$ of trace-free symmetric endomorphisms of $(\stw, g)$ equipped with the multiplication
\begin{align}\label{deunitalization}
A \jrd B = A \jmult B + B \circ A - \tfrac{1}{2n}\tr(A \circ B + B \circ A)\Id_{\stw},
\end{align}
and the invariant metric $G(A, B) = \tfrac{1}{n}\tr(A \jmult B) = \tfrac{1}{n}\tr(A\circ B)$. 
When $\dim \stw = 3$, $G(A \jrd A, A) = \tfrac{1}{3}\tr(A^{3}) = \det A$. (These claims follow from standard formulas as in \cite{Faraut-Koranyi}, and are demonstrated more or less explicitly in \cite{Fox-simplicial} and \cite[section $10$]{Tkachev-universality}.)

Section \ref{3dsection} treats the case $\dim \ste = 3$. In this case the $6$-dimensional algebra $(\mcurv(\std), \opmult)$ is linearly isomorphic to $S^{2}\std$. The map sending $\al \in S^{2}\std$ to $\al^{\sharp} \in \sherm(\ste, h)$ defined by $h(\al^{\sharp}(x), y) = \al(x, y)$ for $x, y \in \std$ is a linear isomorphism. Transported from $S^{2}\std$ to $\sherm(\ste, h)$ via $\sharp$, the product $\opmult$ can be expressed in terms of familiar operations on symmetric endomorphisms. Lemma \ref{3dlemma} and Theorem \ref{3dcharacterizationtheorem} describe the product on $\sherm(\ste, h)$ corresponding to $\opmult$ as
\begin{align}
A\tprod B = A\jmult B - \tfrac{1}{4}\left(\tr(A)B + \tr(B)A + \tr(AB)I- (\tr A)(\tr B)I\right).
\end{align}
In particular, this product is nonunital and it is not the Jordan product $\jmult$. Identify $S^{2}_{0}\std$ with $\sherm_{0}(\ste, h)$ and equip it with the trace-free Jordan product $\jrd$ defined in \eqref{deunitalization}. More precisely, Lemma \ref{3dlemma} shows that $\sherm_{0}(\ste, h) \oplus \rea$ equipped with the multiplication
\begin{align}
(A, r)\star (B, s) = \left(A \jrd B + \tfrac{1}{4}(rB + sA), rs + \tfrac{1}{12}\lb A, B \ra\right).
\end{align}
is isomorphic to $(\mcurv(\std), \opmult)$ via the linear map $(\al^{\sharp}, r) \in \sherm_{0}(\ste, h) \oplus \rea \to \al + rh \in S^{2}\std$. 

What is more interesting is Theorem \ref{3dcharacterizationtheorem} that characterizes $\tprod$ in terms in intrinsic algebraic terms (for Euclidean $h$). The situation can be summarized informally as that $(\mcurv(\std), \opmult)$ is the most symmetric $O(3)$-invariant metrized commutative algebra structure on $S^{2}\std$, in that the number of orbits of its idempotents is the smallest possible, two, and the spectra of their multiplication endomorphisms have the maximal redundancy. Up to isomorphism there is a one-parameter family of $O(3)$-invariant commutative algebra structures on $S^{2}\std$ each metrized by an $O(3)$-invariant inner product and each of which contains a rank one idempotent and contains no square-zero element. The additional condition that there be only two orbits of idempotents, one generated by a multiple of $h$, the other by a rank one idempotent, distinguishes two such algebras. One of them is Killing metrized and the other is $(S^{2}\std, \tprod)$. Alternatively, they are distinguished by the multiplicity of the eigenvalue $1/2$ of the multiplication endomorphism of a rank one idempotent, which is always at least $2$, as a consequence of $O(3)$-invariance, but is $3$ uniquely for $(S^{2}\std, \tprod)$. The proof yields as corollaries that $\tprod$ is simple and its automorphism group is exactly the image of $O(3)$ in its induced action on $S^{2}\std$. Corollary \ref{3dpropertiescorollary} summarizes precisely all that is proved.

When $\dim\ste = 4$, Lemma \ref{selfdualcurvmultlemma} shows that the $5$-dimensional subspaces $\mcurvweyl^{\pm}(\std)$ are orthogonal ideals of $(\mcurvweyl(\std), \opmult)$. Theorem \ref{mcurvweyltheorem} shows that each of these subalgebras is isometrically isomorphic to the deunitalization of the six-dimensional rank $3$ simple real Euclidean Jordan algebra of symmetric endomorphisms of a $3$-dimensional vector space. 
The linear maps assigning to $\sX \in \mcurvweyl(\std)$ endomorphisms $\op{\sX}_{\pmdf\std} \in \eno(\pmdf\std)$ of the spaces $\pmdf \std$ of self-dual and anti-self-dual $2$-forms induce $SO(4)$-module isomorphisms $\mcurvweyl^{\pm}(\std) \simeq \sherm_{0}(\pmdf \std, h)$ \cite[Section $1.127$]{Besse}. The content of Theorem \ref{mcurvweyltheorem} is that a suitable multiple of $\op{\sX}_{\pmdf\std}$ is an algebra isomorphism.

\begin{theorem}\label{mcurvweyltheorem}
Let $(\ste, h)$ be a $4$-dimensional oriented Euclidean vector space.
Consider the deunitalization $(\sherm_{0}(\pmdf \std, h), \jrd)$ of the $6$-dimensional rank $3$ simple real Euclidean Jordan algebra $(\sherm(\pmdf \std, h), \jmult)$ of symmetric endomorphisms of the $3$-dimensional space $\pmdf \std$, equipped with the product $\jrd$ equal to the traceless part of the usual Jordan product $\jmult$ of endomorphisms and the metric $G(A, B) = \tfrac{1}{3}\tr A \circ B$. 
\begin{enumerate}
\item The map $\Psi:(\mcurvweyl^{\pm}(\std), \opmult, \lb \dum, \dum \ra)\to (\sherm_{0}(\pmdf \std, h), \jrd, \tfrac{4}{3}G)$ defined by $\Psi(\sX) = 3\op{\sX}$ is an $SO(4)$-equivariant
isometric algebra isomorphism.
\item The Killing form $\tau_{\opmult}(\sX, \sY) = \tr L_{\opmult}(\sX)L_{\opmult}(\sY)$ on $(\mcurvweyl^{\pm}(\std), \opmult)$ satisfies $\tau_{\opmult} = \tfrac{21}{16}\lb \dum, \dum \ra$, where $\lb \dum, \dum \ra$ is the metric on $\mcurvweyl(\std)$ given by complete contraction with $h_{ij}$.
\item\label{simplenosz} $(\mcurvweyl^{\pm}(\std), \opmult, \lb \dum, \dum \ra)$ is simple and contains no nontrivial square-zero elements.
\end{enumerate}
\end{theorem}

Theorem \ref{mcurvweyltheorem} is proved twice, at the end of Section \ref{4dsection} and again in Section \ref{idempotentsection} (Section \ref{computationsection} sketches still another proof). The isomorphism is described both conceptually and explicitly. 
The explicit isomorphism is based on the construction of a convenient basis of $\mcurvweyl^{+}(\std)$ and the calculation of the multiplication table for its elements. See Lemma \ref{4dsubalgebralemma}.
The conceptual proof is based on a calculation relating the endomorphisms of $\ext^{2}\std$ given by $\op{\sX \opmult \sY}_{\ext^{2}\std}$ and $\op{\sX}_{\ext^{2}\std}\circ \op{\sY}_{\ext^{2}\std}$, where $\circ$ denotes composition of endomorphisms, that shows
\begin{align}\label{cmultpreiso0}
\begin{split}
\tfrac{1}{3}\op{\sX \opmult \sY}_{\ext^{2}\std}& = \op{\sX}_{\ext^{2}\std}\jmult \op{\sY}_{\ext^{2}\std}- \tfrac{1}{6}\tr(\op{\star \sX}_{\ext^{2}\std} \circ \op{ \sY}_{\ext^{2}\std}) \star - \tfrac{1}{6}\tr( \op{\sX}_{\ext^{2}\std} \circ \op{\sY}_{\ext^{2}\std}) \Id_{\ext^{2}\std},
\end{split}
\end{align}
in which $\circ$ is composition of endomorphisms and $\star$ denotes both the Hodge star operator on $\ext^{2}\std$ and the involution it induces on $\mcurvweyl(\std)$; see Lemma \ref{selfdualcurvmultlemma} for details. 

For $\ste$ of dimension greater than $4$, it would be interesting to obtain a formula like \eqref{cmultpreiso0} for the difference $\op{\sX \opmult \sY}_{\mcurvweyl(\std)}$ in terms of the Jordan product $\op{\sX}_{\mcurvweyl(\std)}\jmult \op{\sY}_{\mcurvweyl(\std)}$ and expressions like those on the right side of \eqref{cmultpreiso0}. 

Part of claim \eqref{simplenosz} of Theorem \ref{mcurvweyltheorem} depends strongly on the assumption of Euclidean signature. 
It shows that in Euclidean signature $(\mcurvweyl(\std), \opmult)$ contains no square-zero element if $\dim \ste = 4$, while Lemma \ref{indefiniteszlemma} shows that if $\dim \ste \geq 4$ and $h$ has indefinite signature, then $(\mcurvweyl(\std), \opmult)$ is spanned by square-zero elements (see also Example \ref{maxdegenexample}). It would be interesting to know if $(\mcurvweyl(\std), \opmult)$ contains a nonzero square-zero element when $h$ is Euclidean and $\dim \ste > 4$. Theorem \ref{faithfulmultiplicationtheorem} shows the weaker result that the multiplication $\opmult$ is faithful if $h$ is Euclidean and $\dim \ste \geq 4$; equivalently, $(\mcurv(\std), \opmult)$ contains no zero divisors. 

The simplicity of the algebras $(\mcurvweyl^{\pm}(\std), \opmult)$ and $(\mcurvkahlerweyl(\std), \opmult)$ could perhaps appear unremarkable in light of a result of Popov showing that, over an algebraically closed field $\fie$ of characteristic zero, a generic algebra is simple. Precisely, \cite[Theorem $4$]{Popov} shows that the set of structure tensors of simple algebras over $\fie$ is open and dense. However, the first Theorem $3$  of \cite{Popov}\footnote{Due to a typographical error, its Theorem $2$ is mislabeled as Theorem $3$.} shows that a generic (in the same sense) algebra has trivial automorphism group, whereas $(\mcurvweyl^{\pm}(\std), \opmult, \lb \dum, \dum \ra)$,  $(\mcurvweyl(\std), \opmult)$, and $(\mcurvkahlerweyl(\std), \opmult)$ have large automorphism groups that contain respectively the Lie groups $SO(4)$, $O(n)$, and $U(n)$, and so are atypical from this point of view. 
Nonetheless, \cite[Theorem $5$]{Popov} shows that if the automorphism group of a finite-dimensional algebra with nontrivial multiplication over $\fie$ contains a connected algebraic subgroup that acts irreducibly on the algebra, then the algebra is simple. Although the algebras considered here are defined over $\rea$, Popov's argument can be used essentially as written to show that $(\mcurvweyl(\std), \opmult)$ and $(\mcurvkahlerweyl(\std), \opmult)$ are simple when $\dim \std > 4$. A precise statement of a more general result is given here as Theorem \ref{simpletheorem} and Theorem \ref{einsteintheorem} records its application to $(\mcurvweyl(\std), \opmult)$.

Among metrized commutative algebras those that have large automorphism groups are somewhat exceptional. That a Lie group $G$ acts on a metrized commutative algbera by automorphisms has the consequence that the orbit of an idempotent is a $G$ homogeneous space.
It would be interesting to describe completely the $G$-orbits of idempotents in subalgebras of $(\mcurv(\std), \opmult)$. For $\dim \std = 3$, Corollary \ref{3dpropertiescorollary} gives such a description, while for $\dim \std = 4$, such a description can be deduced from Theorem \ref{mcurvweyltheorem} and the computations used to prove Theorem \ref{fusiontheorem}. In this direction, Lemma \ref{idempotentparameterizationlemma} shows that when $\dim \std = 2n \geq 4$, certain of the idempotents produced by Lemma \ref{weylidempotentslemma} constitute an orbit of $O(2n)$ acting in $\mcurvweyl(\std)$ identified with the space $SO(2n)/U(n)$ of orthogonal complex structures on $\ste$ inducing a given orientation on $\ste$.

Since all claims in the paper are pure linear algebra, they extend straightforwardly to sections of tensor bundles over smooth manifolds.
Although no application to Ricci flow is immediately available, it is reasonable to hope that the results obtained here will be useful for studying curvature conditions on manifolds. 
A different, Lie theoretic, point of view on the structure of the multiplication $\opmult$ has been used profitably in \cite{Bohm-Wilking, Wilking}. For background on the definition of $\opmult$ as in \eqref{curvmultdefined}, its properties, and its role in the study of the Ricci flow, see also \cite{Andrews-Hopper}. 
Some features of the algebra $(\mcurv(\std), \opmult)$ are used implicitly in the study of the Ricci flow \cite{Bohm-Wilking, Brendle-book, Brendle-surgeryisotropic, Hamilton, Hamilton-four, Huisken, Richard-curvaturecones, Richard-lowerbounds, Richard-Seshadri, Wilking}. The algebraic perspective makes some of the manipulations used in such studies appear more natural, and focuses attention on certain structural features, namely the invariance and nondegeneracy of the Killing type trace form and the identification of idempotent elements and the spectra of their left multiplication operators, that are not self-evidently relevant from the geometric perspective. 

It would be interesting to extend results obtained here, for example Theorem \ref{mcurvweyltheorem}, to pseudo-Euclidean real vector spaces and to vector spaces over general base fields. 

\section{Notation and conventions}
All vector spaces considered here are finite-dimensional over $\rea$.
The abstract index conventions in the sense of Penrose \cite[chapter $2$]{Penrose-Rindler} are used when convenient. Given a vector space $\ste$,  $\al_{i_{1}\dots i_{l}}^{j_{1}\dots j_{k}}$ indicates an element of $\tensor^{k}\ste \tensor \tensor^{l}\std$. The indices are labels indicating tensor valencies and symmetries, and do not refer to any choice of reference frame. Enclosure of indices in square brackets or parentheses indicates complete antisymmetrization or complete symmetrization over the enclosed indices; indices delimited by vertical bars are omitted from such (anti)symmetrizations. For example $2a_{ijk} = a_{[i|j|k} + a_{(i|j|k)}$ is the decomposition of $a_{ijk}$ into its parts antisymmetric and symmetric in the first and last index. The symmetric product $\al \sprod \be \in S^{k+l}\std$ of symmetric tensors $\al \in S^{k}\std$ and $\be \in S^{l}\std$ is defined by complete symmetrization, $(\al \sprod \be)_{i_{1}\dots i_{k+l}} = \al_{(i_{1}\dots i_{k}}\be_{i_{k+1}\dots i_{k+l})}$, whereas the wedge product $\al \dwedge \be$ of antisymmetric tensors $\al \in \ext^{k}\std$ and $\be \in \ext^{l}\std$ is defined as a multiple of the complete antisymmetrization of their tensor product, by $(\al \dwedge \be)_{i_{1}\dots i_{k+l}} = \binom{k+l}{k}\al_{[i_{1}\dots i_{k}}\be_{i_{k+1}\dots i_{k+l}]}$.

Indices are raised and lowered, respecting horizontal position, using a nondegenerate symmetric bilinear form $h_{ij}$ (called a \emph{metric}) and the inverse symmetric bivector $h^{ij}$ satisfying $h^{ip}h_{pj} = \delta_{j}\,^{i}$. The pair $(\ste, h)$ is called a \emph{metric vector space}. The metric $h$ is \emph{Euclidean} if it is positive definite and in this case $(\ste, h)$ is called a \emph{Euclidean vector space}. 
Throughout the paper the norms used on tensor modules are those given by complete contraction with the metric (and not those induced from the standard $O(h)$-representation). A subspace $\mathbb{M} \subset \tensor^{k}\std \tensor \tensor^{l}\ste$ is a metric vector space with the metric $\lb \dum, \dum \ra$ defined via complete contraction with $h_{ij}$ and $h^{ij}$ by $\lb \al, \be \ra = \al_{i_{1}\dots i_{k}}^{j_{1}\dots j_{l}}\be_{a_{1}\dots a_{k}}^{b_{1}\dots b_{l}}h^{i_{1}a_{1}}\dots h^{i_{k}a_{k}}h_{j_{1}b_{1}}\dots h_{j_{l}b_{l}}$. 

Let $(\ste, h)$ be an $n$-dimensional metric vector space. When Euclidean $h$ is fixed, the abstract orthogonal group $O(n)$ is identified with the orthogonal group $O(h)$ of linear automorphisms of $\ste$ preserving $h$. The action of $GL(\ste)$ on $\ste$ given by $(g\cdot x)^{i} = x^{p}g_{p}\,^{i}$ induces the cogredient action on $\std$ given by $(g\cdot \mu)_{i} = (g^{-1})_{i}\,^{p}\mu_{p}$ and these actions extend in the usual way to $\tensor^{k}\ste \tensor \tensor^{l}\std$. By definition, $g_{i}\,^{j} \in GL(\ste)$ is in $O(h)$ if and only if $g_{i}\,^{p}g_{jp} = h_{ij}$, or, what is the same $(g^{-1})_{i}\,^{j} = g^{j}\,_{i}$. This implies the action of $O(h)$ commutes with taking traces. For example, for $\sX \in \mcurv(\std)$ and $g \in O(h)$, $\rictr(g\cdot \sX) = g \cdot \rictr(\sX)$, so $\mcurvweyl(\std)$ is an $O(h)$-submodule of $\mcurv(\std)$.

The space $\eno(\ste)$ of linear endomorphisms of $\ste$ is regarded as an algebra with multiplication $\circ$ given by composition. 
The \emph{adjoint involution} $\si_{h}:(\eno(\ste), \circ)\to (\eno(\ste), \circ)$ of the metric $h$ on $\ste$ is the real linear antiautomorphism defined by $h(\si_{h}(\phi)x, y) = h(x, \phi(y))$ for all $x, y \in \ste$ and $\phi \in \eno(\ste)$. For a metrized vector space $(\ste, h)$, the subspace $\sherm(\ste, h) = \sa(\eno(\ste), \si_{h}) = \{\phi \in \eno(\ste): \si_{h}(\phi) = \phi\}$ of $h$-self-adjoint endormorphisms is a Jordan algebra with the product $\phi \jmult \psi = \tfrac{1}{2}(\phi \circ \psi + \psi \circ \phi)$ for $\phi, \psi \in \sa(\ste, h)$. 
(When $h$ is clear from context there is written simply $\sherm(\ste)$ for brevity.)
There holds $\al \circ \be = \tfrac{1}{2}[\al, \be] + \al \jmult \be$, where $\al \jmult \be = \tfrac{1}{2}(\al \circ \be + \be \circ \al)$.

Via metric duality, $\al_{ij} \in \tensor^{2}(\std)$ is identified with the endomorphism $x^{j} \to x^{i}\al_{i}\,^{j}$ of $\ste$, and the composition of the endomorphisms of $\ste$ determined by raising the second indices of $\al_{ij} , \be_{ij} \in \tensor^{2}\std$ is given by $(\al \circ \be)_{i}\,^{j} = \al_{p}\,^{j}\be_{i}\,^{p}$. These conventions are such that, for $x, y, z, w \in \std$,
\begin{align}\label{xycirczw}
&(x \tensor y)\circ (z \tensor w) = \lb w, x \ra z \tensor y.
\end{align}
The pullback to $\tensor^{2}\std$ of the Lie bracket of endomorphisms yields the Lie bracket $[\dum, \dum]:\tensor^{2}(\std) \times \tensor^{2}(\std) \to \tensor^{2}(\std)$ given by $[\al, \be] = \al \circ \be - \be \circ \al$. Similarly, the multiplication induced via metric duality on $\tensor^{2}\std$ by the usual Jordan product of endomorphisms is denoted by $\jmult$.
The subspace $\ext^{2}\std$ is a subalgebra of $(\tensor^{2}\std, [\dum, \dum])$ and $(\ext^{2}\std, [\dum, \dum])$ is isomorphic to the Lie algebra $\mathfrak{so}(\ste, h)$. 

\section{General results about commutative algebras}
As explained in the introduction, the argument proving \cite[Theorem $5$]{Popov} adapts almost without change to the present setting to prove Theorem \ref{simpletheorem}.
For the reader's convenience, it is reproduced here with modifications appropriate to the current setting. 

\begin{theorem}\label{simpletheorem}
Let $(\alg, \mlt, h)$ be a nontrivial finite-dimensional commutative algebra. If a connected simple real Lie group $G$ acts on $(\alg, \mlt)$ irreducibly by automorphisms, then $(\alg, \mlt)$ is simple.
\end{theorem}

\begin{proof}
Since the action is irreducible, a nontrivial $G$-invariant ideal equals $\alg$. Assume $(\alg, \mlt)$ is not simple and let $\ideal \subset \alg$ be a minimal proper ideal. Then $g\cdot \ideal$ is again a minimal proper ideal for all $g \in G$. Since the sum of the ideals $g\cdot\ideal$ as $g$ ranges over $G$ is a nontrivial $G$-invariant ideal, it equals $\alg$. For $g, \bar{g} \in G$, the ideals $g\cdot\ideal$ and $\bar{g}\cdot\ideal$ either are equal or have intersection $\{0\}$, and it follows that there are $g_{1}, \dots, g_{r} \in G$ such that $\alg = \oplus_{i = 1}^{r}g_{i}\cdot \ideal$ is a direct sum of vector spaces. Since the product of distinct minimal ideals is the zero ideal, this sum is in fact a direct sum of algebras. If $\jdeal \subset \alg$ is a minimal proper ideal not equal to $g_{i}\cdot I$ for any $1 \leq i \leq r$, then its product with each $g_{i}\cdot \ideal$ is the zero ideal, so its product with $\alg$ is the zero ideal. Applying the preceding argument with $\jdeal$ in place of $\ideal$ shows that the multiplication on $\alg$ is trivial, contrary to hypothesis. The preceding shows that any minimal proper ideal of $(\alg, \mlt)$ has the form $g_{i}\cdot \ideal$ for some $1 \leq i \leq r$, and so $G$ permutes the set $\{g_{1}\cdot \ideal, \dots, g_{r}\cdot\ideal\}$. Since $G$ is connected with simple Lie algebra, any proper normal subgroup of $G$ must be discrete, so this permutation action must be trivial. Consequently, each $g_{i}\cdot \ideal$ is a $G$-invariant linear subspace of $\alg$, contradicting the $G$-irreducibility of $\alg$.
\end{proof}

\begin{theorem}\label{preeinsteintheorem}
A nontrivial finite-dimensional Euclidean metrized commutative algebra $(\alg, \mlt, h)$ on which a real Lie group $G$ acts irreducibly by isometric automorphisms is exact and Killing metrized with Killing form $\tau_{\mlt}(x, y) = \tr L_{\mlt}(x)L_{\mlt}(y)$ equal to a nonzero multiple of the metric $h$. 
\end{theorem}

\begin{proof}
Since $G$ acts on $(\alg, \mlt)$ by automorphisms, the linear form $x \to \tr L_{\mlt}(x)$ is $G$-invariant and so, because $\alg$ is $G$-irreducible, $\ker \tr L_{\mlt} = \alg$.
Likewise, $\tau_{\mlt}$ is $G$-invariant, and because both $h$ and $\tau_{\mlt}$ are $G$-invariant, the endomorphism $A \in \eno(\alg)$ defined by $\tau_{\mlt}(x, y) = h(Ax, y)$ for $x, y, \in\alg$ is $G$-invariant as well. Because both $h$ and $\tau_{\mlt}$ are symmetric, $A$ is $h$-self-adjoint and so is semisimple with real eigenvalues. Since $\alg$ is $G$-irreducible and $A$ is $h$-self-adjoint, by the Schur Lemma, $A = \ka \id_{\alg}$ for some $\ka \in \rea$, so $\tau_{\mlt} =\ka h$. Because $\ka \dim(\alg) = \tr A = |\mu|^{2}_{h}$ where $\mu(x, y, z) = h(x\mlt y, z)$, were $\ka$ zero, then $\mu$ would be identically zero. However, because $(\alg, \mlt)$ is assumed nontrivial there exist $x, y \in\alg$ such that $x \mlt y \neq 0$, so, by the nondegeneracy of $h$ there is $z \in \alg$ such that $h(x\mlt y, z) \neq 0$. 

The following alternative argument shows $\ka \neq 0$. By \cite[Lemma $2.1$]{Tkachev-laplace}, a nontrivial finite-dimensional Euclidean metrized commutative real algebra $(\alg, \mlt)$ contains a nonzero idempotent, for, if $y$ is an extremum of the restriction to the $h$-unit sphere of the cubic polynomial $h(x\mlt x, x)$, then $e =h(y\mlt y, y)^{-1}y$ is a nonzero idempotent. By the invariance of $h$, $L_{\mlt}(e)$ is a self-adjoint endomorphism of $\alg$ that preserves the orthogonal complement of the span of $e$. Hence it is diagonalizable with (possibly repeated) real eigenvalues $1, \la_{1}, \dots, \la_{n-1}$, and $\tau_{\mlt}(e, e) = \tr L_{\mlt}(e)^{2} = 1 + \sum_{i = 1}^{n-1}\la_{i}^{2} \geq 1$. Since $\lb e, e\ra \neq 0$, it follows that $\ka = \tau_{\mlt}(e, e)h(e, e)^{-1} \neq 0$.
\end{proof}

\section{Action of metric curvature tensors on tensors}
Let $(\ste, h)$ be a metric vector space. Note that $\om_{ijkl} \in \mcurvsym(\std) = \{\om_{ijkl} \in \tensor^{4}\std: \om_{(ij)kl} = \om_{ijkl} = \om_{(kl)ij}, \om_{(ijk)l} = 0\}$ satisfies $\om_{klij} = \om_{ijkl}$ and $\om_{k[ij]l}$ is antisymmetric in $k$ and $l$. A straightforward calculation proves that the linear map $\mdual:\mcurvsym(\std) \to \mcurv(\std)$ defined by $\mdual(\mu)_{ijkl} = \tfrac{2}{\sqrt{3}}\mu_{k[ij]l}$ is an isometric isomorphism with inverse given by $\mdual^{-1}(\nu)_{ijkl} = -\tfrac{2}{\sqrt{3}}\nu_{k(ij)l}$.

It is convenient to abuse notation by identifying $S^{2}(\ext^{2}\std)$ and $S^{2}(S^{2}\std)$ with the subspaces $\{\Psi_{ijkl} \in \tensor^{4}\std: \Psi_{klij} = \Psi_{ijkl} = \Psi_{[ij]kl} = \Psi_{ij[kl]}\}$ and $\{\Phi_{ijkl} \in \tensor^{4}\std: \Phi_{klij} = \Phi_{ijkl} = \Phi_{(ij)kl} = \Phi_{ij(kl)}\}$. As $\Psi_{ijkl} \in \ext^{4}\std$ satisfies $\Psi_{klij} = \Psi_{ijkl}$, $\ext^{4}\std$ can be viewed as a subspace of $S^{2}(\ext^{2}\std)$, and likewise $S^{4}\std$ can be regarded as a subspace of $S^{2}(S^{2}\std)$. The Bianchi identity implies $\ext^{4}\std$ and $\mcurv(\std)$ are orthogonal subspaces of $S^{2}(\ext^{2}\std)$ with trivial intersection, and similarly $S^{4}\std$ and $\mcurvsym(\std)$ are orthogonal subspaces of $S^{2}(S^{2}\std)$ with trivial intersection. Comparing dimensions shows $S^{2}(\ext^{2}\std) \simeq \mcurv(\std)\oplus \ext^{4}\std$ and $S^{2}(S^{2}\std) \simeq \mcurvsym(\std)\oplus S^{4}\std$.

For $\Psi \in S^{2}(\ext^{2}\std)$, $\Psi_{[ijkl]} = \Psi_{[ijk]l}$, and the orthogonal projection $\stwoproj:S^{2}(\ext^{2}\std) \to \mcurv(\std)$ is given by $\stwoproj(\Psi)_{ijkl} = \Psi_{ijkl} - \Psi_{[ijk]l}$, so $\stwoproj \oplus (\Id_{S^{2}(\ext^{2}\std)} - \stwoproj): S^{2}(\ext^{2}\std) \to \mcurv(\std)\oplus \ext^{4}\std$ is an isometric isomorphism.
Similarly, the orthogonal projection $\stwoprojdual:S^{2}(S^{2}\std) \to \mcurvsym(\std)$ is given by $\stwoprojdual(\Psi)_{ijkl} = \Psi_{ijkl} - \Psi_{(ijk)l}$, so $(\mdual\circ \stwoprojdual) \oplus (\Id_{S^{2}(S^{2}\std)} - \stwoprojdual): S^{2}(S^{2}\std) \to \mcurv(\std)\oplus S^{4}\std$ is an isometric isomorphism, where, for $\om_{ijkl} \in S^{2}(S^{2}\std)$, 
\begin{align}\label{mdualstwoprojdual}
(\mdual \circ \stwoprojdual)(\om)_{ijkl} = \tfrac{2}{\sqrt{3}}\om_{k[ij]l}.
\end{align}

\begin{lemma}\label{pairlemma}
Let $(\ste, h)$ be a finite dimensional metric vector space and let $k \geq 2$. For $\al, \be \in \tensor^{k}\std$, setting $\pair(\al, \be)$ equal to the $h$-orthogonal projection $\stwoproj(\prepair(\al, \be)) \in \mcurv(\std)$ of $\prepair(\al, \be)_{abcd} = \Psi(\al, \be)_{[ab][cd]} \in S^{2}(\ext^{2}\std)$ where 
\begin{align}\label{preprepair}
\Psi(\al, \be)_{abcd} = \tfrac{1}{k}\sum_{r \neq s}\al_{i_{1}\dots i_{r-1}b i_{r+1}\dots i_{s-1} c i_{s+1}\dots i_{k}}\be^{i_{1}\dots i_{r-1}}\,_{a}\,^{i_{r+1}\dots i_{s-1}}\,_{d}\,^{i_{s+1}\dots i_{k}},
\end{align}
defines a symmetric $O(h)$-equivariant bilinear map $\pair:\tensor^{k}\std \times \tensor^{k}\std \to \mcurv(\std)$. 
\end{lemma}
\begin{proof}
There need be checked only that $\pair(\be, \al)  = \pair(\al, \be)$.
The definition \eqref{preprepair} implies $\Psi(\be, \al)_{abcd} = \Psi(\al, \be)_{badc}$ and there follows
\begin{align}\label{prepair}
\begin{split}
\prepair(\be, \al)_{abcd} &=\Psi(\be, \al)_{[ab][cd]} = \tfrac{1}{4}\left(\Psi(\be, \al)_{abcd} -\Psi(\be, \al)_{bacd} -\Psi(\be, \al)_{abdc} +  \Psi(\be, \al)_{badc}\right) \\
&= \tfrac{1}{4}\left(\Psi(\al, \be)_{badc} -\Psi(\al, \be)_{abdc} -\Psi(\al, \be)_{bacd} +  \Psi(\al, \be)_{abcd}\right) \\
&= \Psi(\al, \be)_{[ab][cd]}  = \prepair(\al, \be)_{abcd},
\end{split}
\end{align}
so that $\pair(\be, \al) = \stwoproj(\prepair(\be, \al)) = \stwoproj(\prepair(\al, \be)) = \pair(\al, \be)$. That $\pair$ is $O(h)$-equivariant means that $\pair(g\cdot \al, g \cdot  \be) = g \cdot \pair(\al, \be)$ for all $\al, \be \in \tensor^{k}\std$ and $g \in O(h)$, and this is apparent from the manifest $O(h)$-equivariance of the construction of $\pair$.
\end{proof}
Because $\pair$ is $O(h)$-equivariant, its restriction to $\stw \times \stw$ where $\stw$ is any $O(h)$-submodule $\stw \subset \tensor^{k}\std$ takes values in some $O(h)$-submodule of $\mcurv(\std)$.

The symmetric group on $k$ elements, $S_{k}$, acts on the left on $\tensor^{k}\std$ by $(\si \cdot \al)_{i_{1}\dots i_{k}} = \al_{i_{\si^{-1}(1)}\dots i_{\si^{-1}(k)}}$ for $\si \in S_{k}$. A filling of a $k$ box Young diagram by distinct indices $i_{1},\dots, i_{k}$ determines the submodule of $\tensor^{k}\std$ comprising tensors antisymmetric in the indices in any column of the filled Young diagram, and such that there vanishes the antisymmetrization over a subset of indices comprising the indices in any given column plus any index from any column to the right of the given column. The tensors in such a submodule are said to have the type given by the filled Young diagram. By \cite[Theorems $5.7.$A and $5.7.$C]{Weyl}, an $O(n)$-module of covariant trace-free tensors on an $n$-dimensional vector space having symmetries corresponding to a Young diagram is nontrivial if and only if the sum of the lengths of the first two columns is no greater than $n$, and by \cite[Theorem $5.7$G]{Weyl} all irreducible finite-dimensional $O(n)$-modules correspond to some such Young diagram. For example, $\mcurvweyl(\std) = \ker \rictr \subset \mcurv(\std)$ corresponds with a filling of a $2\times 2$ square array of boxes, so $\mcurvweyl(\std) = \{0\}$ if $\dim \ste < 4$. If $n = \dim \ste \geq 4$, $\mcurvweyl(\std)$ is a nontrivial irreducible $O(n)$-module. When $\dim \ste = 4$, it decomposes as an $SO(n)$-module into two five-dimensional submodules (see Section \ref{4dsection}). 
A linear endomorphism is said to preserve the type of tensors if it maps tensors with the symmetries determined by a given filled Young diagram into tensors with the same symmetries. 
\begin{lemma}\label{curvoplemma}
Let $(\ste, h)$ be a finite dimensional metric vector space and let $k \geq 2$.
The linear map $\mcurv(\std) \to\eno(\tensor^{k}\std)$ associating with $\sX \in \mcurv(\std)$ the operator $\op{\sX} \in \eno(\tensor^{k}\std)$ defined by 
\begin{align}\label{curvopdefined}
\op{\sX}(\al)_{i_{1}\dots i_{k}} = \tfrac{1}{k}\sum_{r \neq s}\sX^{p}\,_{i_{r}i_{s}}\,^{q}\al_{i_{1}\dots i_{r-1}pi_{r+1}\dots i_{s-1} q i_{s+1}\dots i_{k}},
\end{align}
commutes with the action of $S_{k}$ on $\tensor^{k}\std$, so preserves the type of tensors; is $O(h)$-equivariant, meaning $\op{g \cdot \sX}(\al) = g\cdot \op{\sX}(g^{-1}\cdot \al)$ for all $\al \in \tensor^{k}\std$ and $g \in O(h)$; and has image in the self-adjoint endomorphisms, meaning $\op{\sX} \in \sa(\tensor^{k}\std, \lb\dum, \dum \ra)$ for all $\sX \in \mcurv(\std)$.
\end{lemma}

\begin{proof}
That $\op{\sX}$ preserves the type of tensors follows from $\op{\sX}(\si \cdot \al) = \si \cdot \op{\sX}(\al)$ for $\si \in S_{k}$ and $\al \in \tensor^{k}\std$. This is immediate from the definition \eqref{curvopdefined} and the symmetries of $\sX_{ijkl}$. The relation $\op{g \cdot \sX}(\al) = g\cdot \op{\sX}(g^{-1}\cdot \al)$ likewise follows from a straightforward computation (for this relation to hold it is necessary that $g$ be orthogonal because the metric is used in \eqref{curvopdefined} where indices are contracted).
It follows from \eqref{curvopdefined} that
\begin{align}\label{curvopinjective}
\begin{split}
\lb \op{\sX}(\al), \be \ra  = \lb \sX, \pair(\al, \be)\ra 
\end{split}
\end{align}
where $\pair(\al, \be)$ is as in Lemma \ref{pairlemma}. By Lemma \ref{pairlemma}, $\pair(\al, \be)$ is symmetric in $\al$ and $\be$ and by \eqref{curvopinjective} this implies that $\lb \op{\sX}(\al), \be \ra = \lb \al, \op{\sX}(\be)\ra$ for all $\al, \be \in \tensor^{k}\std$, so $\op{\sX}$ is self-adjoint.
\end{proof}

When it is known that $\op{\sX}$ preserves a subspace $\stw \subset \tensor^{k}\std$, there is written $\op{\sX}_{\stw}$ instead of $\op{\sX}$ when helpful for clarity. 
By Lemma \ref{curvoplemma}, for $\sX \in \mcurv(\std)$ and $k \geq 2$, $\op{\sX}$ preserves $\ext^{k}\std$ and $S^{k}\std$, so there are written $\op{\sX}_{\ext^{k}\std} \in \sa(\ext^{k}\std)$ and $\op{\sX}_{S^{k}\std} \in \sa(S^{k}\std)$ for the restrictions of $\op{\sX}$ to $\ext^{k}\std, S^{k}\std \subset \tensor^{k}\std$.
By Lemma \ref{curvoplemma}, $\op{\sX}_{S^{k}\std}(\al)_{i_{1}\dots i_{k}} = (k-1)\sX^{p}\,_{(i_{1}i_{2}}\,^{q}\al_{i_{3}\dots i_{k)}pq}$ for $\al \in S^{k}\std$ and $\op{\sX}_{\ext^{k}\std}(\al)_{i_{1}\dots i_{k}} = (k-1)\sX^{p}\,_{[i_{1}i_{2}}\,^{q}\al_{i_{3}\dots i_{k]}pq} = -\tfrac{k-1}{2}\sX^{pq}\,_{[i_{1}i_{2}}\al_{i_{3}\dots i_{k]}pq}$ for $\al \in \ext^{k}\std$.
In particular, when $k = 2$, $\op{\sX}_{S^{2}\std}(\al)_{ij} =  \al^{pq}\sX_{ipqj}$ and $\op{\sX}_{\ext^{2}\std}(\al)_{ij}  = \al^{pq}\sX_{ipqj} = -\tfrac{1}{2}\al^{pq}\sX_{ijpq}$.
For example, $\op{\sX}_{S^{2}\std}(h) = \rictr(\sX)$. 
For $\si_{ij} \in \tensor^{2}\std$, write $(\sym \si)_{ij} = \si_{(ij)}$ and $(\asym\si)_{ij} = \si_{[ij]}$. For $\sX \in \mcurv(\std)$, by Lemma \ref{curvoplemma}, $\op{\sX}_{\tensor^{2}\std}(\si)= \op{\sX}_{S^{2}\std}(\sym \si) + \op{\sX}_{\ext^{2}\std}(\asym \si)$.
This means that if $\tensor^{2}\std$ is regarded as the orthogonal direct sum $S^{2}\std \oplus \ext^{2}\std$ via $\si_{ij} = \si_{(ij)} + \si_{[ij]}$, then $\op{\sX}_{\tensor^{2}\std} = \op{\sX}_{S^{2}\std}\oplus \op{\sX}_{\ext^{2}\std}$ is an orthogonal direct sum too.

\begin{example}\label{hhidexample}
Because $h \kwedge h \in \mcurv(\std)$ satisfies $g\cdot (h\kwedge h) = h\kwedge h$ for $g \in O(h)$, when $\stw$ is an irreducible $O(n)$-submodule preserved by $\op{h \kwedge h}$, it follows from Lemma \ref{curvoplemma} and the Schur lemma that $\op{h\kwedge h}_{\stw}$ is a constant multiple of $\Id_{\stw}$. For example, $\op{h \kwedge h}_{\ext^{2}\std} = -\id_{\ext^{2}\std}$.
\end{example}

Transferring the canonical isomorphisms $\eno(\ext^{2}\std) \simeq \ext^{2}\std \tensor (\ext^{2}\std)^{\ast}$ and $\eno(S^{2}\std) \simeq S^{2}\std \tensor (S^{2}\std)^{\ast}$ via metric duality yields linear isomorphisms 
\begin{align}\label{sharp4defined}
\begin{aligned}
&\sharp:\tensor^{2}(\ext^{2}\std) = \{\Psi_{ijkl} \in \tensor^{4}\std: \Psi_{ijkl} = \Psi_{[ij]kl} = \Psi_{ij[kl]}\} \to \eno(\ext^{2}\std),&\\
&\sharp:\tensor^{2}(S^{2}\std) = \{\Phi_{ijkl} \in \tensor^{4}\std: \Phi_{ijkl} = \Phi_{(ij)kl} = \Phi_{ij(kl)}\} \to \eno(S^{2}\std),
\end{aligned}
\end{align}
defined by sending $\Psi \in \tensor^{2}(\ext^{2}\std)$ to $\Psi^{\sharp}(\al)_{ij} = \Psi_{ij}\,^{pq}\al_{pq}$ and $\Phi \in \tensor^{2}(S^{2}\std) $ to $\Phi^{\sharp}(\si)_{ij} = \Phi_{ij}\,^{pq}\si_{pq}$.
Note that, for $\sX \in \mcurv(\std)$, the tensor identified in this way with $\op{\sX}_{\ext^{2}\std}$ is $-\tfrac{1}{2}\sX_{ijkl}$. Said otherwise, for $\sX$ viewed as an element of $S^{2}(\ext^{2}\std)$, $\sX^{\sharp}  = -2\op{\sX}_{\ext^{2}\std}$.

\begin{lemma}\label{tracelemma}
Let $(\ste,h)$ be a Euclidean vector space. 
For $\Psi \in \tensor^{2}(\ext^{2}\std)$ and $\Phi \in \tensor^{2}(S^{2}\std)$, $\tr \Psi^{\sharp}  = \Psi_{pq}\,^{pq}$ and $\tr \Phi^{\sharp}  = \Phi_{pq}\,^{pq}$.
\end{lemma}
\begin{proof}
For an orthonormal basis $\{E^{(1)}_{i}, \dots, E^{(n)}_{i}\}$ of $\std$, $\{\tfrac{1}{\sqrt{2}} E^{(a)}\dwedge E^{(b)}: 1 \leq a < b \leq n\}$ is an orthonormal basis of $\ext^{2}\std$ and $\{\sqrt{2} E^{(a)}\sprod E^{(b)} : 1 \leq a < b \leq n\}\cup\{E^{(a)}\sprod E^{(a)}:1 \leq a \leq n\}$ is an orthonormal basis of $S^{2}\std$ so
\begin{align}
\begin{aligned}
\tr \Psi^{\sharp} & = \tfrac{1}{2}\sum_{1 \leq a < b \leq n}\lb\Psi^{\sharp}(E^{(a)}\dwedge E^{(b)}), E^{(a)}\dwedge E^{(b)}\ra 
=  \sum_{a = 1}^{n}\sum_{b = 1}^{n}\Psi^{ijkl}E^{(a)}_{i}E^{(b)}_{j}E^{(a)}_{k}E^{(b)}_{l} = \Psi_{pq}\,^{pq},\\
\tr \Phi^{\sharp} & = 2\sum_{1 \leq a < b \leq n}\lb\Phi^{\sharp}(E^{(a)}\sprod E^{(b)}), E^{(a)}\sprod E^{(b)}\ra + \sum_{1 \leq a  \leq n}\lb\Phi^{\sharp}(E^{(a)}\sprod E^{(a)}), E^{(a)}\sprod E^{(a)}\ra \\
&=  \sum_{a = 1}^{n}\sum_{b = 1}^{n}\Phi^{ijkl}E^{(a)}_{i}E^{(b)}_{j}E^{(a)}_{k}E^{(b)}_{l} = \Phi_{pq}\,^{pq}. 
\end{aligned}&\qedhere
\end{align}
\end{proof}
For example, Lemma \ref{tracelemma} implies that $\tr\op{\sX}_{\ext^{2}\std} = \tfrac{1}{2}\scal(\sX)$ for $\sX \in \mcurv(\std)$.

When $\op{\mcurv(\std)}$ preserves $\stw \subset \tensor^{k}\std$, Lemma \ref{curvoplemma} does not affirm that the linear map $\sX \to \op{\sX}_{\stw}$ is injective. From \eqref{curvopinjective} it is apparent that this is true if and only if $\{\pair(\al, \be): \al, \be \in \stw\}$ spans $\mcurv(\std)$. 
Corollary \ref{traceextcorollary} shows that $\sX \to \op{\sX}_{\ext^{2}\std}$ and $\sX \to \op{\sX}_{S^{2}\std}$ are injective.
It is used in the proof of Theorem \ref{mcurvweyltheorem}.
Note that viewing $\sX \in \mcurv(\std)$ as an element of $\mdual(S^{2}(S^{2}\std))$ yields $\mdual^{-1}(\sX)^{\sharp}  = -\tfrac{2}{\sqrt{3}}\op{\sX}_{S^{2}\std}$.

\begin{corollary}\label{traceextcorollary}
Let $(\ste,h)$ be an $n$-dimensional Euclidean vector space. For $\sX, \sY \in \mcurv(\std)$, 
\begin{align}\label{traceext}
&\tr(\op{\sX}_{\ext^{2}\std}\circ \op{\sY}_{\ext^{2}\std}) = \tfrac{1}{4}\lb \sX, \sY\ra, && \tr(\op{\sX}_{S^{2}\std}\circ \op{\sY}_{S^{2}\std}) = \tfrac{3}{4}\lb \sX, \sY\ra.
\end{align} 
The $O(n)$-equivariant linear maps $\mcurv(\std) \to \sa(\ext^{2}\std)$, $\mcurv(\std) \to \sa(S^{2}\std)$, and $\mcurv(\std) \to \sa(\tensor^{2}\std)$ given by $\sX \to 2\op{\sX}_{\ext^{2}\std} = -\sX^{\sharp}$, $\sX \to  \tfrac{2}{\sqrt{3}}\op{\sX}_{S^{2}\std} = - \mdual^{-1}(\sX)^{\sharp}$, and $\sX \to \op{\sX}_{\tensor^{2}\std}$ are isometric with respect to the trace-norms on $\sa(\ext^{2}\std)$, $\sa(S^{2}\std)$, and $\sa(\tensor^{2}\std)$ and injective.
\end{corollary}
\begin{proof}
Via $\sharp$, the endomorphisms $\op{\sX}_{\ext^{2}\std} \circ \op{\sY}_{\ext^{2}\std}$ and $\op{\sX}_{S^{2}\std} \circ \op{\sY}_{S^{2}\std}$ are identified with the tensors $\tfrac{1}{4}\sX_{kl}\,^{pq}\sY_{pqij} \in \tensor^{2}(\ext^{2}\std)$ and $\sX_{k(pq)l}\sY_{i}\,^{(pq)}\,_{j} \in \tensor^{2}(S^{2}\std)$, so \eqref{traceext} follows from Lemma \ref{tracelemma} and $\sX_{i(jk)l}\sY^{i(jk)l} = \tfrac{3}{4}\lb \sX, \sY\ra$.
The injectivity follows from the nondegeneracy of the pairing $\lb \dum, \dum \ra$.
\end{proof}

\section{Algebra structures on the space of metric curvature tensors}\label{curvmultsection}
This section defines the multiplication $\opmult$ and derives its basic properties.

The injectivity of $\op{\dum}_{\ext^{2}\std}$ and $\op{\dum}_{S^{2}\std}$ shown in Corollary \ref{traceextcorollary} means that commutative algebra structures $\opmulta$ and $\opmults$ on $\mcurv(\std)$ can be constructed as follows. By the injectivity there are unique $\sX \opmulta \sY \in \mcurv(\std)$ and $\sX \opmults \sY \in \mcurv(\std)$ such that $(\sX\opmulta \sY)^{\sharp} = -2\op{\sX \opmulta \sY }_{\ext^{2}\std}$ equals the orthogonal projection on $\op{\mcurv(\std)}_{\ext^{2}\std}$ of the Jordan product $(-2\op{\sX}_{\ext^{2}\std})\jmult (-2\op{\sY}_{\ext^{2}\std}) = \sX^{\sharp}\jmult \sY^{\sharp}\in \sa(\ext^{2}\std) $ and $\mdual^{-1}(\sX\opmults \sY)^{\sharp} = -\tfrac{2}{\sqrt{3}}\op{\sX \opmults \sY }_{S^{2}\std}$ equals the orthogonal projection on $\op{\mcurv(\std)}_{S^{2}\std}$ of the Jordan product $(-\tfrac{2}{\sqrt{3}}\op{\sX}_{S^{2}\std})\jmult (-\tfrac{2}{\sqrt{3}}\op{\sY}_{S^{2}\std}) = (\mdual^{-1}(\sX)^{\sharp}\jmult \mdual^{-1}(\sY)^{\sharp}  \in \sa(S^{2}\std))$. More precisely:
\begin{align}
\label{opmultasdefined2}
&\sX\opmulta \sY = -2\stwoproj\left(\left(\op{\sX}_{\ext^{2}\std}\jmult \op{\sY}_{\ext^{2}\std}\right)^{\flat}\right),&&
&\sX\opmults \sY = -\tfrac{2}{\sqrt{3}}\mdual\circ \stwoprojdual\left(\left(\op{\sX}_{S^{2}\std}\jmult \op{\sY}_{S^{2}\std}\right)^{\flat}\right),
\end{align}
where $\flat$ is the inverse of $\sharp$
\begin{lemma}
Let $(\ste, h)$ be a metric vector space.
For $\sX, \sY \in \mcurv(\std)$, the commutative multiplications $\opmulta$ and $\opmults$ on $\mcurv(\std)$ are given by
\begin{align}
\label{opmultadefined}
\begin{split} (\sX \opmulta \sY)_{ijkl} & =- \tfrac{1}{4}\left(\sX_{ij}\,^{pq}\sY_{klpq} + \sY_{ij}\,^{pq}\sX_{klpq} - 2\sX_{[ij}\,^{pq}\sY_{kl]pq}\right),
\end{split}\\
\label{opmultsdefined}\begin{split}(\sX \opmults \sY)_{ijkl} & =-\tfrac{1}{3}\left(\sX_{k(pq)i}\sY_{j}\,^{(pq)}\,_{l}-\sX_{k(pq)j}\sY_{i}\,^{(pq)}\,_{l} + \sX_{j(pq)l}\sY_{k}\,^{(pq)}\,_{i}- \sX_{i(pq)l}\sY_{k}\,^{(pq)}\,_{j}\right).
\end{split}
\end{align}
The Ricci and scalar traces are
\begin{align}
\label{opmultasrictrace}
\begin{aligned}
&\rictr(\sX \opmulta \sY)_{ij} = \tfrac{1}{2}\sX_{(i}\,^{abc}\sY_{j)abc}, \qquad \scal(\sX \opmulta \sY) = \tfrac{1}{2}\lb \sX, \sY\ra, &\\
&\rictr(\sX \opmults \sY)_{ij} = -\tfrac{1}{2}\sX_{(i}\,^{abc}\sY_{j)abc} + \tfrac{1}{3}\left(\op{\sX}_{S^{2}\std}(\rictr(\sY))_{ij} + \op{\sY}_{S^{2}\std}(\rictr(\sX))_{ij}  \right),  \\
&\scal(\sX \opmults \sY) =  -\tfrac{1}{2}\lb \sX, \sY\ra + \tfrac{2}{3}\lb \rictr(\sX), \rictr(\sY)\ra.
\end{aligned}
\end{align}
\end{lemma}
\begin{proof}
Each of \eqref{opmultadefined} and \eqref{opmultsdefined} is just a matter of unraveling the notation in \eqref{opmultasdefined2}. 
Straightforward computations using \eqref{opmultadefined} and \eqref{opmultsdefined} show \eqref{opmultasrictrace}.
\end{proof}

\begin{lemma}\label{stalgebralemma}
Let $(\ste, h)$ be a metric vector space. For $s, t \in \rea$ and $\opmult_{s, t} = s\opmulta + t \opmults$, the commutative algebra $(\mcurv(\std), \opmult_{s, t})$ is metrized by $\lb\dum, \dum\ra$ and $O(h)$ acts on it by isometric algebra automorphisms.
\end{lemma}
\begin{proof}
By \eqref{opmultasdefined2}, Lemma \ref{traceextcorollary}, and using $\sZ^{\sharp}  = -2\op{\sZ}_{\ext^{2}\std}$ and $\mdual^{-1}(\sZ)^{\sharp}  = -\tfrac{2}{\sqrt{3}}\op{\sZ}_{S^{2}\std}$,
\begin{align}\label{opmultasinvariance}
\begin{aligned}
\lb \sX &\opmulta \sY, \sZ\ra = -2\left\lb\stwoproj\left(\left(\op{\sX}_{\ext^{2}\std}\jmult \op{\sY}_{\ext^{2}\std}\right)^{\flat}\right), \sZ \right\ra  
= 4\left\lb \left(\op{\sX}_{\ext^{2}\std}\jmult \op{\sY}_{\ext^{2}\std}\right)^{\flat}, \op{\sZ}_{\ext^{2}\std}^{\flat} \right\ra\\
&= 2\tr\left(\op{\sX}_{\ext^{2}\std}\circ \op{\sY}_{\ext^{2}\std}\circ \op{\sZ}_{\ext^{2}\std} + \op{\sY}_{\ext^{2}\std}\circ \op{\sX}_{\ext^{2}\std}\circ \op{\sZ}_{\ext^{2}\std}  \right),\\
\lb \sX &\opmults \sY, \sZ\ra =-\tfrac{2}{\sqrt{3}}\left\lb\mdual\circ \stwoprojdual\left(\left(\op{\sX}_{S^{2}\std}\jmult \op{\sY}_{S^{2}\std}\right)^{\flat}\right), \sZ \right\ra  
= \tfrac{4}{3}\left\lb \left(\op{\sX}_{S^{2}\std}\jmult \op{\sY}_{S^{2}\std}\right)^{\flat}, \op{\sZ}_{S^{2}\std}^{\flat} \right\ra \\
&= \tfrac{2}{3}\tr\left(\op{\sX}_{S^{2}\std}\circ \op{\sY}_{S^{2}\std}\circ \op{\sZ}_{S^{2}\std} + \op{\sY}_{S^{2}\std}\circ \op{\sX}_{S^{2}\std}\circ \op{\sZ}_{S^{2}\std}  \right).
\end{aligned}
\end{align}
The right-hand sides of \eqref{opmultasinvariance} are completely symmetric in $\sX$, $\sY$, and $\sZ$, and this shows the invariance with respect to $\lb\dum, \dum\ra$ of $\opmulta$ and $\opmults$ and so also of any linear combination $s\opmulta + t \opmults$.
That $O(h)$ acts on $(\mcurv(\std), \opmult_{s, t})$ by isometric algebra automorphisms follows from the manifest $O(h)$ invariance of the construction of $\opmulta$ and $\opmults$.
\end{proof}

Because rescaling the multiplication of a commutative algebra yields an isomorphic algebra, the products $\opmult_{s, t}$ are best viewed as parameterized by $[s,t] \in \proj(\rea)$. 

For $\sX, \sY \in \mcurv(\std)$ define $B_{ijkl} = B(\sX, \sY)_{ijkl}  \in S^{2}(\tensor^{2}\std)$ by 
\begin{align}\label{bsharpdefined}
B(\sX, \sY)^{\sharp} = \op{\sX}_{\tensor^{2}\std}\jmult \op{\sY}_{\tensor^{2}\std} = \op{\sX}_{\ext^{2}\std}\jmult \op{\sY}_{\ext^{2}\std} + \op{\sX}_{S^{2}\std}\jmult \op{\sY}_{S^{2}\std}, 
\end{align}
the second equality by Lemma \ref{curvoplemma}, so that
\begin{align}
\label{bdefined}
B_{ijkl} = B(\sX, \sY)_{ijkl} &= \tfrac{1}{2}\left(\sX_{ipjq}\sY_{k}\,^{p}\,_{l}\,^{q} + \sY_{ipjq}\sX_{k}\,^{p}\,_{l}\,^{q} \right) = \left(\op{\sX}_{\tensor^{2}\std}\jmult \op {\sY}_{\tensor^{2}\std}\right)^{\flat}_{ijkl}.
\end{align}

\begin{lemma}
Let $(\ste, h)$ be a metric vector space.
For $\sX, \sY \in \mcurv(\std)$, the commutative multiplications $\opmulta$ and $\opmults$ on $\mcurv(\std)$ are given by
\begin{align}
\label{opmultadefinedb}
\begin{split} (\sX \opmulta \sY)_{ijkl} 
& = -2(B_{[ij]kl} - B_{[ijkl]}) = -\tfrac{2}{3}\left(2B_{[ij]kl} - B_{k[ijl]} + B_{[i|k|j]l}\right),
\end{split}\\
\label{opmultsdefinedb}\begin{split}(\sX \opmults \sY)_{ijkl} 
& = -\tfrac{2}{3}\left(B_{k[ij]l} + B_{[i|k|j]l} \right) = -\tfrac{2}{3}\left(B_{[ij]kl} + 2B_{k[ij]l} - 3B_{[ijkl]} \right).
\end{split}
\end{align}
\end{lemma}

\begin{proof}
That $B_{ijkl}$ is contained in $S^{2}(\tensor^{2}\std)$ is equivalent to the symmetry $B_{klij} = B_{ijkl}$, that is apparent from \eqref{bdefined}. There holds
\begin{align}
2B_{jilk} = \sX_{jpiq}\sY_{l}\,^{p}\,_{k}\,^{q} + \sY_{jpiq}\sX_{l}\,^{p}\,_{k}\,^{q} = \sX_{iqjp}\sY_{k}\,^{q}\,_{l}\,^{p} + \sY_{iqjp}\sX_{k}\,^{q}\,_{l}\,^{p} = 2B_{ijkl}.
\end{align}
Hence $2B_{[ij]kl} = B_{ijkl} - B_{jikl} = B_{ijkl} - B_{ijlk} = 2B_{ij[kl]}$. It follows that $2B_{[ij][kl]} = B_{ij[kl]} - B_{ji[kl]} = B_{[ij]kl} - B_{[ji]kl} = 2B_{[ij]kl}$. 
This means that $B_{[ij]kl} = B_{[ij][kl]} = B_{ij[kl]}$ is the orthogonal projection of $B_{ijkl}$ onto $S^{2}(\ext^{2}\std)$. It follows that $B_{ij(kl)} = B_{(ij)kl} = B_{(ij)(kl)} = B_{(kl)(ij)}$ is the orthogonal projection of $B_{ijkl}$ onto $S^{2}(S^{2}\std)$. 
From the second equality of \eqref{bdefined} it follows that $(\sX \opmulta \sY)_{ijkl}$ and $(\sX \opmults \sY)_{ijkl}$ are given by applying $-2\stwoproj$ to $B_{[ij]kl}$ and $-\tfrac{2}{\sqrt{3}} \mdual\circ \stwoprojdual$ to $B_{(ij)kl}$, respectively, and so doing yields the first equalities of \eqref{opmultadefinedb} and \eqref{opmultsdefinedb}. The second equalities of \eqref{opmultadefined} and \eqref{opmultsdefinedb} follow from $3B_{[ijkl]} = B_{[ij]kl} + B_{k[ij]l} - B_{[i|k|j]l}$.
\end{proof}

Since, by Lemma \ref{curvoplemma}, the operator $\op{\sX}$ preserves type, if $\sX, \sY \in \mcurv(\std)$ then $\op{\sX}(\sY), \op{\sY}(\sX) \in \mcurv(\std)$. Theorem \ref{opalgebratheorem} shows that, remarkably, $\op{\sX}(\sY) = \op{\sY}(\sX)$, so that a commutative multiplication of curvature tensors can be defined by $\sX \opmult \sY = \op{\sX}(\sY)$, and , moreover, $\opmult = \tfrac{3}{2}(\opmulta + \opmults)$.

\begin{theorem}\label{opalgebratheorem}
Let $(\ste, h)$ be a metric vector space.
The bilinear map $\opmult:\mcurv(\std) \times \mcurv(\std) \to \mcurv(\std)$ defined by $\sX \opmult \sY  =\op{\sX}(\sY)$ for $\sX, \sY \in \mcurv(\std)$ is $O(h)$-equivariant and symmetric, so determines a commutative algebra structure $\opmult$ on $\mcurv(\std)$ metrized by the pairing $\lb\dum, \dum\ra$ given by complete contraction with $h$ and on which $O(h)$ acts by algebra automorphisms. 
Moreover, $\opmult = \tfrac{3}{2}(\opmulta + \opmults)$ and $\opmult$ equals that multiplication given by polarizing the quadratic term of the expression \eqref{ricciflowcurv} for the evolution of the curvature tensor under the Ricci flow.
\end{theorem}

\begin{proof}
Evaluation of $\op{\sX}(\sY)_{ijkl}$ using \eqref{curvopdefined} and $\pair(\sX, \sY)_{ijkl}$ using \eqref{preprepair} and \eqref{prepair} yields
\begin{align}\label{opmultcommutes}
\begin{split}
4\op{\sX}(\sY)_{ijkl} 
&= - 2\sX_{i}\,^{p}\,_{k}\,^{q}\sY_{jplq}- 2\sY_{i}\,^{p}\,_{k}\,^{q}\sX_{jplq} + 2\sX_{i}\,^{p}\,_{l}\,^{q}\sY_{jpkq} +  2\sY_{i}\,^{p}\,_{k}\,^{q}\sX_{jpkq} \\
&\qquad -  \sX_{ij}\,^{pq}\sY_{pqkl} - \sX_{kl}\,^{pq}\sY_{pqij} = 4\pair(\sX, \sY)_{ijkl},
\end{split}
\end{align}
which is evidently symmetric in $\sX$ and $\sY$, showing that $\sX\opmult \sY = \op{\sX}(\sY) = \op{\sY}(\sX)= \sY \opmult \sX$.
By Lemma \ref{curvoplemma}, $\op{\sY}$ is self-adjoint, so, for $\sX, \sY, \sZ \in \mcurv(\std)$,
\begin{align}\label{opmultinvariance}
\begin{split}
\lb \sX &\opmult \sY, \sZ\ra =\lb \op{\sY}(\sX), \sZ\ra = \lb \sX, \op{\sY}(\sZ)\ra = \lb \sX, \sY \opmult \sZ\ra.
\end{split}
\end{align}
This shows the complete symmetry of $\lb \sX \opmult \sY, \sZ\ra$.
That $O(h)$ acts by algebra automorphisms follows from \eqref{opmultcommutes} and the $O(h)$-equivariance of $\pair$ established in Lemma \ref{pairlemma}.

From the symmetries of $B_{ijkl} = B(\sX, \sY)_{ijkl}$ it follows that
\begin{align}\label{prebelskew}
\begin{split}
\tfrac{1}{2}\left(\sX_{ij}\,^{pq}\sY_{pqkl} \right.&\left. + \sY_{ij}\,^{pq}\sX_{pqkl}\right)  \\
&=  \left(\sX_{jpiq} + \sX_{pijq} \right)\left(\sY^{q}\,_{k}\,^{p}\,_{l} + \sY_{k}\,^{pq}\,_{l} \right) + \left(\sY_{jpiq} + \sY_{pijq} \right)\left(\sX^{q}\,_{k}\,^{p}\,_{l} + \sX_{k}\,^{pq}\,_{l} \right) \\
& = B_{ijkl} + B_{jilk} - B_{ijlk} - B_{jikl} = 4B_{[ij][kl]} = 4B_{ij[kl]} = 4B_{[ij]kl}.
\end{split}
\end{align}
Combining \eqref{opmultcommutes} and \eqref{prebelskew} shows
\begin{align}\label{curvmultdefined}
-(\sX \opmult \sY)_{ijkl} &= B_{ijkl} - B_{ijlk} + B_{ikjl} - B_{iljk}= B_{ijkl} - B_{jikl} + B_{ikjl} - B_{iljk}.
\end{align}
That $\opmult = \tfrac{3}{2}(\opmulta + \opmults)$ follows from \eqref{opmultadefinedb}, \eqref{opmultsdefinedb}, and \eqref{curvmultdefined}. By \cite[section $7$]{Hamilton-four} the polarization of the quadratic term of the expression \eqref{ricciflowcurv} for the evolution of the curvature tensor under the Ricci flow equals the right-hand side of \eqref{curvmultdefined}.
\end{proof}

\begin{remark}
It is not self-evident that the right-hand side of \eqref{curvmultdefined} determines an element of $\mcurv(\std)$. Here this is seen as following from the construction of $\opmult$. It also can be checked directly using the symmetries of $B_{ijkl}$ as in \cite{Hamilton, Hamilton-four}. The invariance of $\lb\dum, \dum \ra$ with respect to $\opmult$ is attributed to G. Huisken in \cite{Bohm-Wilking}. The proof given here is conceptually different than the usual proofs by direct computation or as in \cite[section $1$]{Bohm-Wilking}). 
\end{remark}

A metrized commutative $\rea$-algebra $(\alg, \mlt, h)$ is determined up to isometric isomorphism by the $O(h)$-orbit of the associated homogeneous cubic polynomial, $P_{(\alg, \mlt)}(x)= (1/6)h(x\mlt x, x)$, because $P$ and the metric $h$ determine the multiplication $\mlt$ via polarization.

\begin{corollary}
The cubic polynomial of $(\mcurv(\std), \opmult)$ has the form
\begin{align}
P_{(\mcurv(\std), \opmult)}(\sX) = \tr \op{\sX}_{\ext^{2}\std}^{3} + \tfrac{1}{3}\tr \op{\sX}_{S^{2}\std}^{3}.
\end{align}
\end{corollary}

\begin{proof}
This follows from \eqref{opmultasinvariance} in conjunction with Theorem \ref{opalgebratheorem}.
\end{proof}

\begin{lemma}\label{opweylsubalgebralemma}
Let $(\ste, h)$ be a metric vector space. For $\sX, \sY \in \mcurv(\std)$, 
\begin{align}
\label{opricxy}&
\rictr(\sX \opmult \sY) =  \tfrac{1}{2}\left( \op{\sX}_{S^{2}\std}(\rictr(\sY)) + \op{\sY}_{S^{2}\std}(\rictr(\sX))\right),&&
\scal(\sX \opmult \sY)  = \lb \rictr(\sX), \rictr(\sY)\ra.
\end{align}
In particular, the subspace $\mcurvweyl(\std) = \ker \rictr \subset \mcurv(\std)$ is a subalgebra of $(\mcurv(\std),\opmult)$. 
\end{lemma}

\begin{proof}
The identities \eqref{opricxy} are immediate from \eqref{opmultasrictrace} and Theorem \ref{opalgebratheorem}.
\end{proof}

\begin{lemma}\label{weakpositivitylemma}
Let $(\ste, h)$ be a Euclidean vector space.
\begin{enumerate}
\item (Due to \cite{Richard-curvaturecones}) For all $\sX \in \mcurv(\std)$, $\scal(\sX\opmult \sX) \geq 0$ with equality if and only if $\sX \in \mcurvweyl(\std)$. 
\item\label{sxsxsx} If $\sX \in \mcurv(\std)$ is such that $\op{\sX}_{S^{2}\std}$ is nonnegative, then $\scal((\sX \opmult \sX)\opmult \sX) \geq 0$, with equality if and only if $\rictr(\sX) \in \ker \op{\sX}_{S^{2}\std}$.
\item\label{sxsysz} If $\sY \in \mcurvweyl(\std)$, then $\scal((\sX \opmult \sY)\opmult \sZ - \sX \opmult(\sY \opmult \sZ)) = 0$ for all $\sX, \sZ \in \mcurv(\std)$.
\end{enumerate}
\end{lemma}

\begin{proof}
If $\sX \in \mcurv(\std)$ satisfies $\scal(\sX \opmult \sX) =0$, then, by \eqref{opricxy}, $0 = \scal(\sX \opmult \sX) = |\rictr(\sX)|^{2}$, so $\rictr(\sX) = 0$ and $\sX \in \mcurvweyl(\std)$. By \eqref{opricxy}, $\scal((\sX \opmult \sX)\opmult \sX) = \lb \op{\sX}_{S^{2}\std}(\rictr(\sX)), \rictr(\sX)\ra$, from which \eqref{sxsxsx} follows. From \eqref{opricxy} and the self-adjointness of $\op{\sX}_{S^{2}\std}$ and $\op{\sZ}_{S^{2}\std}$ it follows that $2\scal((\sX \opmult \sY)\opmult \sZ - \sX \opmult(\sY \opmult \sZ)) = \lb \op{\sX}_{S^{2}\std}(\rictr(\sZ)) -  \op{\sZ}_{S^{2}\std}(\rictr(\sX)) , \rictr(\sY)\ra$ for all $\sX, \sY, \sZ \in \mcurv(\std)$. Claim \eqref{sxsysz} follows.
\end{proof}
In particular, Lemma \ref{weakpositivitylemma} shows that, in Euclidean signature, if $\sX \opmult \sX = 0$ then $\sX \in \mcurvweyl(\std)$. That is, a square-zero element of $(\mcurv(\std), \opmult)$ must be contained in $\mcurvweyl(\std)$.
The first claim of Lemma \ref{weakpositivitylemma} is an example of a claim that depends on the assumption that $h$ have definite signature. In other signatures the same proof shows only that $\rictr(\sX)$ is null.

\begin{lemma}\label{mcsubspacelemma}
Let $(\ste, h)$ be a Euclidean vector space and let $\Pi \in \eno(\ste)$ be the $h$-orthogonal projection onto the subspace $\stw \subset \ste$. The linear map $\imt:\tensor^{4}\stw^{\ast} \to \tensor^{4}\std$ defined by $\imt(\sX)(A, B, C, D) = \sX(\Pi A, \Pi B,\Pi C,\Pi D)$ restricts to an injective algebra homomorphism that maps $(\mcurv(\stwd), \opmult)$ into $(\mcurv(\std), \opmult)$ and $(\mcurvweyl(\stw^{\ast}), \opmult)$ into $(\mcurvweyl(\std), \opmult)$.
\end{lemma}

\begin{proof}
Straightforward calculations show that $B(\imt(\sX), \imt(\sY)) = \imt(B(\sX, \sY))$ for $\sX, \sY \in \mcurv(\stw)$. Since $\imt:\tensor^{4}\stw^{\ast} \to \tensor^{4}\std$ commutes with permutations of the factors, this suffices to show $\imt(\sX)\opmult \imt(\sY) = \imt(\sX \opmult \sY)$ for $\sX, \sY \in \mcurv(\stw^{\ast})$. Similarly, $\rictr(\imt(\sX)) = \imt(\rictr(\sX))$ (where $\imt$ is defined on $\tensor^{2}\stw$) and so $\imt(\mcurvweyl(\stw^{\ast})) \subset \mcurvweyl(\std)$.
\end{proof}

\begin{remark}\label{signremark}
As commented in Remark \ref{pullbackremark}, there is no purely algebraic reason to prefer $\opmult$ to its pullback $\opmult_{t}$ via rescaling by $t \in \reat$, as these algebras are isomorphic. In particular there is no algebraic reason to prefer $\opmult$ to $-\opmult = \opmult_{-1}$. However there are at least three aesthetic and geometric reasons for preferring $\opmult$ to $-\opmult$:
\begin{itemize}
\item (geometric) By Lemma \ref{alhbehlemma}, a \emph{positive} multiple of the curvature tensor of the round metric on the sphere is an $\opmult$-idempotent.
\item (aesthetic) By Lemma \ref{weakpositivitylemma}, the scalar curvature of an $\opmult$-square is nonnegative. (Morally, squares should be \emph{positive}.)
\item (aesthetic) For the induced curvature operator on $\mcurv(\std)$, $\sX \opmult \sY = \op{\sX}_{\mcurv(\std)}(\sY)$. 
\end{itemize}
When comparing with the literature, care has to be taken. For example that some signs in \cite{Richard-curvaturecones} at first glance appear inconsistent with those here is because \cite{Richard-curvaturecones} works directly with curvature operators, essentially with what is here called $\op{\sX}_{\ext^{2}\std}$, which is the image under the map here called $\sharp$ of the tensor $-\tfrac{1}{2}\sX_{ijkl}$, whereas here results are stated directly in terms of the tensor $\sX_{ijkl}$. 
\end{remark}

\section{Calculation of products in \texorpdfstring{$(\mcurv(\std), \opmult)$}{}}\label{productssection}
There are two ways to construct curvature tensors from simpler objects, a commutative product $\kwedge$ of symmetric two-tensors and a commutative product $\cdot$ of antisymmetric two-tensors. Lemma \ref{productsmultlemma} records the $\opmult$ products of curvature tensors obtained in these ways. They are used later in the construction of idempotents in $(\mcurvweyl(\std), \opmult)$. Although the full strength of these computations is not needed in this paper, their proofs illustrate the use of $\opmulta$ and $\opmults$ in the computation of $\opmult$.

By straightforward computations using \eqref{preprepair}, Lemma \ref{pairlemma}, and \eqref{mdualstwoprojdual}, 
\begin{align}\label{kwedgedefined}
(\al \kwedge \be)_{ijkl}  = -2\pair(\al, \be)_{ijkl} = \al_{k[i}\be_{j]l} - \al_{l[i}\be_{j]k} = \sqrt{3}\left((\mdual \circ \stwoprojdual)(\al \tensor \be + \be \tensor \al)\right)_{ijkl}
\end{align}
defines a symmetric bilinear map $\kwedge:S^{2}(\std) \times S^{2}(\std) \to \mcurv(\std)$. When $k = 2$, $\kwedge$ is half what is usually called the Kulkarni-Nomizu product.
Again by \eqref{preprepair}, for $\al, \be \in \ext^{2}\std$, $\prepair(\al, \be)_{ijkl} = \tfrac{1}{2}\left( \al_{k[i}\be_{j]l} - \al_{l[i}\be_{j]k}\right)$, so that, by Lemma \ref{pairlemma}, an $O(h)$-equivariant symmetric bilinear map $\cdot:\ext^{2}(\std) \times \ext^{2}(\std) \to \mcurv(\std)$ is defined by
\begin{align}\label{cdotdefined}
\begin{split}
\al \cdot \be &= -6\pair(\al, \be)_{ijkl} = -6\stwoproj(\prepair(\al, \be))_{ijkl} = -3\left( \al_{k[i}\be_{j]l} - \al_{l[i}\be_{j]k} - 2\al_{[ij}\be_{kl]}\right)\\
&= \tfrac{3}{2}(\al_{ij}\be_{kl} + \al_{kl}\be_{ij}) - \tfrac{1}{2}(\al \dwedge \be)_{ijkl}
= \tfrac{3}{2}\stwoproj(\al \tensor \be + \be \tensor \al)_{ijkl}.
\end{split}
\end{align}
For a commutative algebra $(\alg, \mlt)$ define 
\begin{align}
&\qp_{\mlt}(x, y) = 2L_{\mlt}(x)\jmult L_{\mlt}(y) - L_{\mlt}(x\mlt y), & &\qp_{\mlt}(x) = \qp_{\mlt}(x, x) = 2L_{\mlt}(x)^{2} - L_{\mlt}(x\mlt x).
\end{align}
For a Jordan algebra, $\qp_{\mlt}(x)$ is what is usually called the \emph{quadratic representation} \cite{Jacobson-jordan, Koecher}. For the special case of $\al, \be, \ga \in \tensor^{2}\std$, 
\begin{align}
\label{qpid0}&\qp_{\jmult}(\al, \be)(\ga) = \tfrac{1}{2}(\al \circ \ga \circ \be + \be \circ \ga \circ \al), && \qp_{\jmult}(\al)(\ga) = \al \circ \ga \circ \al.
\end{align}
The following identities relate $\qp_{\jmult}$ with $\jmult$ and $[\dum, \dum]$.
\begin{align}
\label{qpid1}
\begin{aligned}
\qp_{\jmult}(\al \jmult \be) &+ \tfrac{1}{4}\qp_{\jmult}([\al, \be]) = \tfrac{1}{2}\qp_{\jmult}(\al\circ \be) + \tfrac{1}{2}\qp_{\jmult}(\be \circ\al),\\
\qp_{\jmult}(\al \jmult \be) &- \tfrac{1}{4}\qp_{\jmult}([\al, \be]) =\qp_{\jmult}(\al) \jmult \qp_{\jmult}(\be).
\end{aligned}
\end{align}
Define $\tr:S^{k}\std \to S^{k-2}\std$ by $\tr(\om)_{i_{1}\dots i_{k-2}} = \om_{i_{1}\dots i_{k-2}p}\,^{p}$.
For $\al \in \tensor^{2}\std$, let $\al^{\sharp} = \lb \al, \dum\ra \in \tensor^{2}\ste$, so that, in the notation of \eqref{sharp4defined}, $2(\al \sprod \be)^{\sharp} = \al \tensor \be^{\sharp} + \be \tensor \al^{\sharp}$, where $\sprod$ is the symmetrized tensor product of elements of $\tensor^{2}\std$. From the definitions there follow
\begin{align}\label{symopalbega}
&\begin{aligned}
&\op{\ga \kwedge \si}_{S^{2}\std} = \qp_{\jmult}(\ga, \si) -(\ga\sprod \si)^{\sharp} ,\\
&\op{(\ga \kwedge h)}_{S^{2}\std}(\tau) = \ga \jmult \tau - \tfrac{1}{2}\tr(\tau)\ga- \tfrac{1}{2}\lb \ga, \tau \ra h,\\
&\op{h \kwedge h}_{S^{2}\std}(\tau) = \tau - \tr(\tau)h,
\end{aligned}&&\ga, \si, \tau \in S^{2}(\std),\\
\label{albega}
&\begin{aligned}
\op{(\al \cdot \be)}_{\ext^{2}\std}   =  \qp_{\jmult}(\al, \be) - (\al\sprod \be)^{\sharp} , \end{aligned}&&\al, \be \in \ext^{2}(\std),\\
\label{gasiopal}
&\begin{aligned}
&\op{\ga \kwedge \si}_{\ext^{2}\std} = -\qp_{\jmult}(\ga, \si) , &&\op{\si \kwedge h}_{\ext^{2}\std}(\al) = - \si \jmult \al,\\
&\op{\al \cdot \be}_{S^{2}\std} =3\qp_{\jmult}(\al, \be),&& \op{\al \cdot \be}_{S^{2}\std}(h) =3\al \jmult \be,
\end{aligned}&& \begin{aligned}\ga, \si \in S^{2}\std,\\\al, \be \in\ext^{2}\std. \end {aligned}
\end{align}
By \eqref{curvopinjective}, \eqref{kwedgedefined}, and \eqref{cdotdefined}, for $\sX \in \mcurv(\std)$, $\ga, \si \in S^{2}\std$, and $\al, \be \in \ext^{2}\std$,
\begin{align}\label{sxalwedgebe}
\lb \ga \kwedge \si , \sX\ra & = -2 \lb \pair(\ga, \si), \sX\ra = -2\lb \op{\sX}(\ga), \si \ra = -2\lb \ga, \op{\sX}(\si)\ra,\\
\label{sxalcdotbe}
\lb \al \cdot \be , \sX\ra & = -6 \lb \pair(\al, \be), \sX\ra = -6\lb \op{\sX}(\al), \be \ra = -6\lb \al, \op{\sX}(\be)\ra.
\end{align}
For $\al, \be, \ga, \si \in \ext^{2}(\std)$, by \eqref{sxalcdotbe} and \eqref{albega}, 
\begin{align}\label{alcdotbegacdotsinorm}
\begin{split}
\lb \al \cdot \be, \ga \cdot \si \ra&  = - 6\lb \op{\al\cdot \be}_{\ext^{2}\std}(\ga), \si\ra \\
& = 3\left(\lb \al, \ga \ra\lb \be, \si \ra +  \lb \al, \si \ra \lb \be, \ga \ra + \tr(\al\circ \ga \circ \be \circ \si+ \al\circ \si \circ \be \circ \ga)\right).
\end{split}
\end{align}
For $\al, \be, \ga, \si \in S^{2}\std$, by \eqref{sxalwedgebe} and \eqref{symopalbega},
\begin{align}\label{alwedgebegawedgesinorm}
\begin{split}
\lb \al \kwedge \be, \ga \kwedge \si\ra &= -2\lb \op{\al \kwedge \be}_{S^{2}\std}(\ga), \si\ra \\
&= \lb \al, \ga \ra\lb \be, \si \ra +  \lb \al, \si \ra \lb \be, \ga \ra - \tr(\al\circ \ga \circ \be \circ \si + \be \circ \ga\circ  \al\circ \si).
\end{split}
\end{align}
For $\al, \be \in \ext^{2}\std$ and $\ga, \si \in S^{2}\std$, by \eqref{gasiopal} together with \eqref{sxalwedgebe} or \eqref{sxalcdotbe},
\begin{align}\label{alcdotbegawedgesinorm}
\begin{split}
\lb \al \cdot \be, \ga \kwedge \si \ra = -3\tr(\ga \circ \al \circ \si \circ \be + \si\circ \al \circ \ga\circ \be).
\end{split}
\end{align}
Using \eqref{mdualstwoprojdual} it is straightforward to check that for $\si \in S^{2}\std$ and $\tau \in \ext^{2}\std$, there hold
\begin{align}
\label{tabidentity2}
&\si_{i[k}\si_{l]j}  = \tfrac{1}{2}(\si \kwedge \si)_{ijkl},& &\implies& & -2\stwoproj((\qp_{\jmult}(\si)_{\ext^{2}\std})^{\flat}) = -\si \kwedge \si,\\
\label{tabidentity}
&\tau_{i[k}\tau_{l]j} - \tau_{[ik}\tau_{lj]} = -\tfrac{1}{6}(\tau\cdot \tau)_{ijkl}, &  &\implies&&-2\stwoproj((\qp_{\jmult}(\tau)_{\ext^{2}\std})^{\flat}) = \tfrac{1}{3}\tau \cdot \tau,\\
\label{symtabidentity}
&\si_{i(k}\si_{l)j} - \si_{(ik}\si_{lj)} = -\tfrac{\sqrt{3}}{6}\mdual^{-1}(\si \kwedge \si)_{ijkl},& &\implies&& -\tfrac{2}{\sqrt{3}}\mdual \circ \stwoprojdual((\qp_{\jmult}(\si)_{S^{2}\std})^{\flat}) = \tfrac{1}{3}\si \kwedge \si,\\
\label{symtabidentity2}
&\tau_{i(k}\tau_{l)j} =-\tfrac{\sqrt{3}}{6}\mdual^{-1}(\tau \cdot \tau)_{ijkl},& &\implies&&-\tfrac{2}{\sqrt{3}}\mdual \circ \stwoprojdual((\qp_{\jmult}(\tau)_{S^{2}\std})^{\flat}) =\tfrac{1}{3}\tau \cdot \tau.
\end{align}

\begin{lemma}\label{productsmultlemma}
Let $(\ste, h)$ be a metric vector space. For $\al, \be, \ga, \si \in S^{2}(\std)$, 
\begin{align}
\label{symalbecmgasi}
\begin{split}
(\al &\kwedge \be) \opmult (\ga \kwedge \si)  = -\tfrac{1}{4}\left(\lb \al, \ga\ra \be \kwedge \si + \lb \al, \si \ra \be \kwedge \ga + \lb \be, \ga \ra \al \kwedge \si + \lb \be, \si \ra \al \kwedge \ga \right)\\
&\quad + \tfrac{1}{2}\left(\al \kwedge \qp_{\jmult}(\ga, \si)(\be) 
+\be \kwedge \qp_{\jmult}(\ga, \si)(\al) 
+\ga \kwedge \qp_{\jmult}(\al, \be)(\si) 
+\si \kwedge \qp_{\jmult}(\al, \be)(\ga)\right) \\
&\quad - \tfrac{1}{8}\left([\al, \ga]\cdot [\be, \si] + [\al, \si] \cdot [\be, \ga] \right) - \tfrac{1}{2}\left((\al \jmult \ga) \kwedge (\be \jmult \si) + (\al \jmult \si) \kwedge (\be \jmult \ga) \right).
\end{split}
\end{align}
For $\al, \be \in \ext^{2}\std$ and $\ga, \si \in S^{2}(\std)$,
\begin{align}
\label{cdw}
\begin{split}
(\al \cdot \be) &\opmult (\ga \kwedge \si)  \\
&= -\tfrac{1}{2}\left(\al \cdot \qp_{\jmult}(\ga, \si)(\be) + \be \cdot \qp_{\jmult}(\ga, \si)(\al) \right)
 + \tfrac{3}{2}\left(\ga \kwedge \qp_{\jmult}(\al, \be)(\si)  + \si \kwedge \qp_{\jmult}(\al, \be)(\ga)\right)\\
& +\tfrac{1}{2}\left( (\al \jmult \ga)\cdot (\be \jmult \si)  +  (\al \jmult \si)\cdot (\be \jmult \ga)\right)
-\tfrac{3}{8}\left([\ga, \al]\kwedge[\si, \be] + [\si, \al]\kwedge [\ga, \be]\right).
\end{split}
\end{align}
For $\al, \be, \ga, \si \in \ext^{2}(\std)$,
\begin{align}
\label{albecmultgasi}
\begin{split}
(\al \cdot \be)&\opmult(\ga \cdot \si)  =
-\tfrac{1}{4}\left(\lb \al, \ga \ra \be \cdot \si + \lb \al, \si \ra \be \cdot \ga +\lb \be, \ga \ra \al \cdot \si + \lb \be, \si \ra \al \cdot \ga\right)\\
&\quad +\tfrac{1}{2}\left(\qp_{\jmult}(\al, \be)(\ga)\cdot \si+\ga \cdot \qp_{\jmult}(\al, \be)(\si) +\qp_{\jmult}(\ga, \si)(\al) \cdot \be+ \al \cdot \qp_{\jmult}(\ga, \si)(\be) \right) \\
&\quad - \tfrac{5}{8}\left([\al, \ga] \cdot [\be, \si]+[\be, \ga] \cdot [\al, \si]\right)  
+  \tfrac{3}{2}\left( (\al \jmult \ga)\kwedge (\be \jmult \si) +(\be \jmult \ga)\kwedge (\al \jmult \si)\right).
\end{split}
\end{align}
\end{lemma}

\begin{proof}
For $\al, \be \in S^{2}\std$, by \eqref{gasiopal}, \eqref{symopalbega}  and \eqref{qpid1}, 
\begin{align}\label{oms1}
&\op{\al \kwedge \al}_{\ext^{2}\std}\jmult \op{\be \kwedge \be}_{\ext^{2}\std} = \qp_{\jmult}(\al)\jmult \qp_{\jmult}(\be) = \qp_{\jmult}(\al \jmult \be) - \tfrac{1}{4}\qp_{\jmult}([\al, \be]).\\
\label{oms2}
\begin{split}
&\op{\al \kwedge \al}_{S^{2}\std}\jmult \op{\be \kwedge \be}_{S^{2}\std} = (\qp_{\jmult}(\al) - (\al \tensor \al)^{\sharp})\jmult (\qp_{\jmult}(\be) - (\be \tensor \be)^{\sharp})\\
& = \qp_{\jmult}(\al \jmult \be) - \tfrac{1}{4}\qp_{\jmult}([\al, \be]) + \lb\al, \be\ra(\al\sprod \be)^{\sharp}- \left(\qp_{\jmult}(\al)(\be)\sprod \be\right)^{\sharp} -  \left(\qp_{\jmult}(\be)(\al)\sprod \al\right)^{\sharp}.
\end{split}
\end{align}
By \eqref{opmultasdefined2}, applying $-2\stwoproj$ to \eqref{oms1} and $-\tfrac{2}{\sqrt{3}}\mdual \circ \stwoprojdual$ to \eqref{oms2}, and using \eqref{cdotdefined} and \eqref{tabidentity2}-\eqref{symtabidentity2} with $\si  = \al \jmult \be$ and $\tau = [\al, \be]$ to simplify the results yields
\begin{align}\label{opmultasalbe}
\begin{aligned}
(\al \kwedge \al) \opmulta (\be \kwedge \be) &=-(\al\jmult \be) \kwedge (\al\jmult \be) - \tfrac{1}{12}[\al, \be]\cdot [\al, \be],\\
(\al \kwedge \al) \opmults (\be \kwedge \be)  &= \tfrac{2}{3}\left(\al \kwedge\qp_{\jmult}(\be)(\al)+\be \kwedge\qp_{\jmult}(\al)(\be) - \lb \al, \be\ra(\al \kwedge \be)\right) \\
&\quad + \tfrac{1}{3}(\al\jmult\be) \kwedge (\al\jmult\be) - \tfrac{1}{2}[\al, \be]\cdot[\al, \be].
\end{aligned}
\end{align}
Using $\opmult = \tfrac{3}{2}(\opmulta + \opmults)$ and \eqref{opmultasalbe} to evaluate $(\al \kwedge \al) \opmult (\be \kwedge \be)$ yields 
\begin{align}
\label{symalalcmbebe}
\begin{split}
(\al \kwedge \al) \opmult (\be \kwedge \be) & 
= -\lb \al, \be \ra \al \kwedge \be  + \al \kwedge \qp_{\jmult}(\be)(\al) + \be \kwedge \qp_{\jmult}(\al)(\be)\\
&\quad  -\tfrac{1}{4} [\al, \be] \cdot [\al, \be]  - (\al \jmult \be) \kwedge (\al\jmult \be) .
\end{split}
\end{align}
Polarizing \eqref{symalalcmbebe} first in $\al$ then in $\be$ yields \eqref{symalbecmgasi}. 

For $\al \in \ext^{2}\std$ and $\ga \in S^{2}\std$, by \eqref{albega}, \eqref{gasiopal}, \eqref{symopalbega}, and \eqref{qpid1},
\begin{align}\label{oms1b}
\begin{split}
\op{\al \cdot \al}_{\ext^{2}\std}\jmult \op{\ga \kwedge \ga}_{\ext^{2}\std}&=  (\qp_{\jmult}(\al) - (\al \tensor \al)^{\sharp})\jmult \qp_{\jmult}(\ga) \\&= - \qp_{\jmult}(\al\jmult \ga) + \tfrac{1}{4}\qp_{\jmult}([\al, \ga]) + \left(\qp_{\jmult}(\ga)(\al)\sprod \al\right)^{\sharp} ,
\end{split}\\
\label{oms2b}
\begin{split}
\op{\al \cdot \al}_{S^{2}\std}\jmult \op{\ga \kwedge \ga}_{S^{2}\std} &= 3\qp_{\jmult}(\al)  \jmult \left( \qp_{\jmult}(\ga) - (\ga \tensor \ga)^{\sharp}\right) \\&= 3 \qp_{\jmult}(\al\jmult \ga) - \tfrac{3}{4}\qp_{\jmult}([\al, \ga]) - 3\left(\qp_{\jmult}(\al)(\ga)\sprod \ga\right)^{\sharp}.
\end{split}
\end{align}
By \eqref{opmultasdefined2}, applying $-2\stwoproj$ to \eqref{oms1b} and $-\tfrac{2}{\sqrt{3}}\mdual \circ \stwoprojdual$ to \eqref{oms2b} and using \eqref{cdotdefined}, \eqref{kwedgedefined}, and 
\eqref{tabidentity2}-\eqref{symtabidentity2} with $\si  = [\al, \ga]$ and $\tau = \al \jmult \ga$ to simplify the results yields
\begin{align}\label{opmultasalalgaga}
\begin{aligned}
(\al \cdot \al) \opmulta (\ga\kwedge \ga) &= - \tfrac{2}{3}\al \cdot \qp_{\jmult}(\ga)(\al) - \tfrac{1}{3}(\al \jmult \ga)\cdot (\al \jmult \ga) - \tfrac{1}{4}[\al, \ga] \kwedge [\al, \ga],\\
(\al \cdot \al) \opmults (\ga\kwedge \ga) &= 2\ga \kwedge \qp_{\jmult}(\al)(\ga) +  (\al \jmult \ga)\cdot (\al \jmult \ga) - \tfrac{1}{4}[\al, \ga] \kwedge [\al, \ga].
\end{aligned}
\end{align}
Using $\opmult = \tfrac{3}{2}(\opmulta + \opmults)$ and\eqref{opmultasalalgaga} to evaluate $(\al \cdot \al) \opmult (\ga \kwedge \ga)$ yields 
\begin{align}
\label{cdwpre}
\begin{split}
(\al \cdot \al) &\opmult (\ga \kwedge \ga)  = 
-\al \cdot \qp_{\jmult}(\ga)(\al) + 3\ga \kwedge\qp_{\jmult}(\al)(\ga) + (\al \jmult \ga)\cdot (\al \jmult \ga) - \tfrac{3}{4}[\al, \ga] \kwedge [\al, \ga].
\end{split}
\end{align}
Polarizing \eqref{cdwpre} first in $\al$ then in $\ga$ yields \eqref{cdw}. 

For $\al, \be \in \ext^{2}\std$, by \eqref{albega}, \eqref{gasiopal}, and \eqref{qpid1},
\begin{align}\label{oms1c}
\begin{split}
&\op{\al \cdot \al}_{\ext^{2}\std}\jmult \op{\be \cdot \be}_{\ext^{2}\std} = (\qp_{\jmult}(\al) - (\al \tensor \al)^{\sharp})\jmult(\qp_{\jmult}(\be) - (\be \tensor \be)^{\sharp})\\
& = \qp_{\jmult}(\al \jmult \be) - \tfrac{1}{4}\qp_{\jmult}([\al, \be])  + \lb \al, \be \ra \left(\al \sprod \be\right)^{\sharp} -\left( \qp_{\jmult}(\be)(\al) \sprod \al\right)^{\sharp} -\left( \qp_{\jmult}(\al)(\be) \sprod \be\right)^{\sharp} ,
\end{split}\\
\label{oms2c}
\begin{split}
&\op{\al \cdot \al}_{S^{2}\std}\jmult \op{\be \cdot \be}_{S^{2}\std} = 9\qp_{\jmult}(\al)\jmult \qp_{\jmult}(\be) = 9\qp_{\jmult}(\al \jmult \be) - \tfrac{9}{4}\qp_{\jmult}([\al, \be]).
\end{split}
\end{align}
By \eqref{opmultasdefined2}, applying $-2\stwoproj$ to \eqref{oms1c} and $-\tfrac{2}{\sqrt{3}}\mdual \circ \stwoprojdual$ to \eqref{oms2c} and using \eqref{cdotdefined} and \eqref{tabidentity2}-\eqref{symtabidentity2} with $\tau = [\al, \be]$ and $\si  = \al \jmult \be$ to simplify the results yields
\begin{align}\label{opmultasalalbebecdot}
\begin{aligned}
(\al \cdot \al) \opmulta (\be\cdot \be) &= \tfrac{2}{3}\left(\al \cdot \qp_{\jmult}(\be)(\al) + \be \cdot \qp_{\jmult}(\al)(\be)  - \lb \al, \be\ra\al \cdot \be\right) \\
&\quad - (\al \jmult \be)\kwedge (\al \jmult \be)   - \tfrac{1}{12}[\al, \be] \cdot [\al, \be],\\
(\al \cdot \al) \opmults (\be\cdot\be) &=3(\al \jmult \be)\kwedge (\al \jmult \be) -\tfrac{3}{4} [\al, \be] \cdot [\al, \be].
\end{aligned}
\end{align}
Using $\opmult = \tfrac{3}{2}(\opmulta + \opmults)$ and \eqref{opmultasalalbebecdot} to evaluate $(\al \cdot \al) \opmult (\be\cdot\be)$ yields 
\begin{align}
\label{alalcmultbebe}
\begin{split}
(\al \cdot \al) \opmult (\be\cdot \be)  &
= -\lb \al, \be\ra\al \cdot \be + \al \cdot \qp_{\jmult}(\be)(\al) +\be \cdot \qp_{\jmult}(\al)(\be)  
\\& \quad 
+3(\al \jmult \be)\kwedge (\al \jmult \be)  - \tfrac{5}{4}[\al, \be] \cdot [\al, \be].
\end{split}
\end{align}
Polarizing \eqref{alalcmultbebe} first in $\al$ then in $\be$ yields \eqref{albecmultgasi}. 
\end{proof}

\begin{remark}
Taking $\be = \al\in S^{2}\std$ in \eqref{symalalcmbebe} yields 
\begin{align}
\label{symalalcmalal}
\begin{split}
(\al \kwedge \al) &\opmult (\al \kwedge \al)  = -|\al|^{2} \al \kwedge \al + 2\al \kwedge (\al\circ \al \circ \al) - (\al \circ \al) \kwedge (\al\circ \al).
\end{split}
\end{align}
Taking $\ga = \al\in S^{2}\std$ and $\si = \be\in S^{2}\std$ in \eqref{symalbecmgasi} yields 
\begin{align}
\label{symalbecmalbe}
\begin{split}
(\al &\kwedge \be) \opmult (\al \kwedge \be)  
= -\tfrac{1}{4}|\al|^{2} \be \kwedge \be - \tfrac{1}{4}|\be|^{2}\al \kwedge \al - \tfrac{1}{2}\lb \al, \be \ra \al \kwedge \be 
 +\tfrac{1}{8} [\al, \be] \cdot [\al, \be] 
\\&\quad + \al \kwedge (\al\jmult(\be \circ \be)) + \be \kwedge (\be \jmult(\al \circ \al))
- \tfrac{1}{2}\left((\al \circ \al) \kwedge (\be \circ \be) + (\al \jmult \be) \kwedge (\al\jmult \be) \right).
\end{split}
\end{align}
The identities \eqref{symalalcmalal} and \eqref{symalbecmalbe} suggest that, with further conditions on $\al$ and $\be$, $\opmult$-idempotents can be constructed from $\al \kwedge \al$, $\be \kwedge \be$, and $\al \kwedge \be$. This is shown to be the case in Lemma \ref{projectionkwedgelemma}.
Similarly, taking $\al = \be = \ga = \si \in \ext^{2}\std$ in \eqref{albecmultgasi} yields 
\begin{align}
\label{alalsquared}
\begin{split}
(\al \cdot \al) \opmult (\al \cdot \al) 
&= -|\al|^{2}\al \cdot \al + 2\al \cdot (\al \circ \al \circ \al) + 3(\al \circ \al) \kwedge (\al \circ \al).
\end{split}
\end{align}
The identity \eqref{alalsquared} suggests that with some further condition on $\al$, the element $\al \cdot \al$, or its trace-free part $\tf(\al \cdot \al)$ might be a $\opmult$-idempotent. Corollary \ref{kwedgeidempotentcorollary} and Lemma \ref{weylidempotentslemma} shows that this works. In this regard see also Lemma \ref{indefiniteszlemma} which uses \eqref{alalcmultbebe} and \eqref{alalsquared} to construct square-zero elements in $(\mcurvweyl(\std), \opmult)$ when $h$ has indefinite signature.
\end{remark}

\section{Simplicity of \texorpdfstring{$(\mcurvweyl(\std), \opmult)$}{} and fusion rules for \texorpdfstring{$(\mcurv(\std), \opmult)$}{}}\label{fusionsection}
Lemma \ref{opweylsubalgebralemma} shows that the Weyl curvature tensors constitute a subalgebra of $(\mcurv(\std), \opmult)$. The first part of the section is dedicated to constructing idempotents in $(\mcurvweyl(\std), \opmult)$. This is used in the proof of Theorem \ref{einsteintheorem} and to show the nontriviality of the subalgebra $(\mcurvweyl(\std), \opmult)$, which in turn yields Corollary \ref{mcurvweylsimplecorollary} showing the simplicity of $(\mcurvweyl(\std), \opmult)$ when $\dim \ste > 4$, but is also interesting in its own right, as experience with Jordan and axial algebras suggests that detailed information about idempotents and the spectra of their multiplication endomorphisms is useful for understanding the internal structure of an algebra such as $(\mcurvweyl(\std), \opmult)$. 

Next there are deduced explicit formulas for products in $(\mcurv(\std), \opmult)$ and these are used to deduce the main result of this section, Theorem \ref{fusiontheorem}, that describes the interaction of the subspaces $\mcurvscal(\std)$, $\mcurvric(\std)$, and $\mcurvweyl(\std)$ with respect to $\opmult$. 

\begin{lemma}
Let $(\ste, h)$ be an $n$-dimensional Euclidean vector space. 
For an $O(n)$-module of tensors $\stw$, let $\tf\in \eno(\stw)$ denote the orthogonal projection onto the $O(n)$-submodule $\stw_{0} \subset \stw$ comprising trace-free elements. 
Define $S^{2}_{0}(\std) = \{\al \in S^{2}(\std): \tr\al = 0\}$, $\mcurvric(\std) = \{h \kwedge \al: \al \in S^{2}_{0}(\std)\}\subset \mcurv(\std)$, and $\mcurvscal(\std) = \spn\{h \kwedge h\} \subset \mcurv(\std)$. The orthogonal projections $\projric, \projscal \in \eno(\mcurv(\std))$ on $\mcurvric(\std)$ and $\mcurvscal(\std)$ are given by 
\begin{align}\label{projricdefined}
&\projric(\sX) = -\tfrac{2}{n-2}\rictrz(\sX) \kwedge h, && \projscal(\sX)  = -\tfrac{1}{n(n-1)}\scal(\sX)h \kwedge h,
\end{align}
where $\rictrz(\sX) = \rictr(\sX) - \tfrac{1}{n}\scal(\sX)h$, and the trace-free part $\tf(\sX)$ of $\sX \in \mcurv(\std)$ is given by 
\begin{align}\label{tfsy}
\begin{split}
\tf(\sX)  
&= \sX + \tfrac{2}{n-2} \rictrz(\sX)\kwedge h + \tfrac{1}{n(n-1)}\scal(\sX) h\kwedge h = \sX + \tfrac{2}{n-2} \rictr(\sX)\kwedge h - \tfrac{1}{(n-2)(n-1)}\scal(\sX) h\kwedge h.
\end{split}
\end{align}
\end{lemma}
\begin{proof}
For $\al, \be \in S^{2}(\std)$, computations using the definitions show
\begin{align}\label{rictralkwedgebe}
&\rictr(\al \kwedge \be)  = \al \jmult \be -\tfrac{1}{2} \tr(\al)\be -\tfrac{1}{2} \tr(\be) \al,&
&\scal(\al \kwedge \be)  = \lb \al, \be \ra - \lb \tr\al, \tr \be \ra,\\
\label{rictralh}
&\rictr(\al \kwedge h) = \tfrac{2 - n}{2}\al - \tfrac{1}{2}\tr(\al)h ,&
& \scal(\al \kwedge h) = (1-n)\tr \al,\\
\label{rictrhh}
&\rictr(h \kwedge h) = (1-n)h, & &\scal(h \kwedge h) = -n(n-1).
\end{align}
When $n >2$, $\mcurv(\std) = \mcurvweyl(\std) \oplus \mcurvric(\std) \oplus \mcurvscal(\std)$ is an orthogonal decomposition into irreducible $O(n)$-modules (although $\mcurvweyl(\std)$ is trivial if $\dim \ste = 3$). 
By \eqref{rictralh}, if $n > 2$, the map $\al \to h \kwedge \al$ is a linear isomorphism from $S^{2}_{0}(\std)$ onto its image in $\mcurv(\std)$, which is $\mcurvric(\std)$. The expressions \eqref{projricdefined} and \eqref{tfsy} follow from \eqref{rictralh} and \eqref{rictrhh}.
\end{proof}

\begin{example}
For $\al, \be \in S^{2}\std$, taking $\sX = \al \kwedge \be$ in \eqref{tfsy} and using \eqref{rictralkwedgebe} yields
\begin{align}\label{tfalkwedgebe}
\begin{split}
\tf(\al &\kwedge \be)  = \al \kwedge \be
+ \tfrac{1}{n-2}\left(2\al \jmult \be - \tr(\al)\be - \tr(\be)\al + \tfrac{1}{n-1}\left(\tr(\al)\tr(\be) - \lb \al, \be \ra\right)h\right) \kwedge h.
\end{split}
\end{align}
Similarly, by \eqref{tfsy} and \eqref{rictralcdotbe}, for $\al, \be \in \ext^{2}(\std)$,
\begin{align}\label{rictralcdotbe}
&\rictr(\al \cdot \be)=  3 \al \jmult \be,\qquad  \scal(\al \cdot \be)  =  -3\lb \al, \be \ra,\\
\label{tfalcdotbe}
\begin{split}
\tf(\al \cdot \be)
& = \al \cdot \be + \tfrac{6}{n-2}(\al \jmult \be) \kwedge h + \tfrac{3}{(n-1)(n-2)}\lb \al, \be \ra h \kwedge h.
\end{split}
\end{align}
Since, by \eqref{sxalwedgebe}, $\lb \sX, h \kwedge h \ra = -2\scal(\sX)$ for any $\sX \in \mcurv(\std)$, by \eqref{rictralcdotbe} there holds $\lb \al \cdot \be, h \kwedge h\ra = 6\lb \al, \be \ra$.
Alternatively, this is a special case of \eqref{sxalcdotbe} or a special case of \eqref{alcdotbegawedgesinorm}.
\end{example}

\begin{lemma}\label{indefiniteszlemma}
Let $(\ste, h)$ be a metric vector space of dimension $n > 3$. If $h$ has indefinite signature with minimal inertial index $k \geq 1$, then $(\mcurvweyl(\std), \opmult)$ is spanned by square-zero elements and contains a trivial subalgebra of dimension $n- 2k - 1$.
\end{lemma}

\begin{proof}
By assumption there are a nonzero $h$-isotropic $w \in \std$ and a unimodular $h$-orthogonal basis $\{z^{(\al)}:0 \leq \al \leq n-2k-1\}$ of a codimension $2k$ subspace of $\std$ orthogonal to $w$ and on which $h$ has definite signature. Concretely, there is $\ep \in \{\pm 1\}$ such that $h(w, z^{(\al)}) = 0$ and $h(z^{(\al)}, z^{(\al)}) = \ep$ for $0 \leq \al \leq n-2k-1$. For $1 \leq \al \leq n- 2k-1$, the tensor $\sZ^{(\al)} = (w \dwedge z^{(0)})\cdot (w \dwedge z^{(0)}) - (w \dwedge z^{(\al)})\cdot (w \dwedge z^{(\al)})$ is nontrivial. Suppose $1 \leq \al < \be \leq n - 2k -1$. Because $(w\dwedge z^{(\al)}) \circ (w\dwedge z^{(\al)}) = -\ep w \tensor w = (w \dwedge z^{(\be)})\tensor (w \dwedge z^{(\be)})$, by \eqref{rictralcdotbe}, $\rictr(\sZ^{(\al)}) =0$, so $\sZ^{(\al)} \in \mcurvweyl(\std)$. Because $(w \dwedge z^{(\al)})\circ (w \dwedge z^{(\al)}) \circ (w\dwedge z^{(\al)}) = 0$ and $(w\tensor w)\kwedge (w \tensor w) = 0$, by \eqref{alalsquared}, $((w \dwedge z^{(\al)})\tensor (w \dwedge z^{(\al)}))\opmult ((w \dwedge z^{(\al)})\tensor (w \dwedge z^{(\al)})) = 0 = ((w \dwedge z^{(\be)})\tensor (w \dwedge z^{(\be)}))\opmult ((w \dwedge z^{(\be)})\tensor (w \dwedge z^{(\be)}))$, and, because $(w \dwedge z^{(\al)})\circ (w \dwedge z^{(\be)}) = 0$,by \eqref{alalcmultbebe}, $((w \dwedge z^{(\al)})\tensor (w \dwedge z^{(\al)}))\opmult ((w \dwedge z^{(\be)})\tensor (w \dwedge z^{(\be)})) = 0$, so $\sZ^{(\al)} \opmult \sZ^{(\al)} = 0$ and $\sZ^{(\al)} \opmult \sZ^{(\be)} = 0$. It follows that $\spn\{\sZ^{(\al)}:1 \leq \al \leq n-2k-1\}$ is a trivial subalgebra of $(\mcurvweyl(\std), \opmult)$.

Since $(\mcurvweyl(\std), \opmult)$ contains a nonzero square-zero element $\sZ$, the span of the $O(h)$-orbit of $\sZ$ is a nontrivial $O(h)$-invariant subspace of the irreducible $O(h)$-module $\mcurvweyl(\std)$, so equals $\mcurvweyl(\std)$, and hence $\mcurvweyl(\std)$ is spanned by square-zero elements.
\end{proof}

Let $\idem(\alg, \mlt)$ denote the set of idempotent elements in the algebra $(\alg, \mlt)$. Two idempotents $e,f \in \idem(\alg, \mlt)$ are \emph{orthogonal} if $e\mlt f = 0 = f\mlt e$.

If $\al \in \idem(S^{2}\std, \jmult)$, then $\tr \al = |\al|^{2}$ is the rank of the orthogonal proejction $\al_{i}\,^{j}$, so $\tr \al$ is said to be the \emph{rank} of $\al$. If $\al, \be \in \idem(S^{2}\std, \jmult)$ are orthogonal idempotents, then $\al \circ \be = 0 = \be \circ \al$, for $\al \circ \be = \al \circ \al \circ \be \circ \be = \be \circ \be \circ \al \circ \al = \be \circ \al = -\al \circ \be$. 

\begin{lemma}\label{projectionkwedgelemma}
Let $(\ste, h)$ be an $n$-dimensional Euclidean vector space.  For orthogonal idempotents $\al, \be \in \idem(S^{2}\std, \jmult)$ with $a = \tr \al$ and $b = b$ there hold
\begin{align}\label{hbidempotent}
\begin{aligned}
&(\be \kwedge \be) \opmult (\be \kwedge \be) = (1 - b)\be \kwedge \be,&
&(\al\kwedge \be)\opmult(\al\kwedge \be) = -\tfrac{1}{2}(\al\kwedge \be) - \tfrac{b}{4}\al\kwedge \al - \tfrac{a}{4}\be \kwedge \be,&
 \\
&(\be\kwedge\be)\opmult (\be\kwedge\al) = \tfrac{1-b}{2}\al\kwedge\be,&
&(\al\kwedge \al)\opmult (\be \kwedge \be)  = 0,&
\end{aligned}
\end{align}
Moreover, $|\be \kwedge \be|^{2} = 2(b - 1 )b$, so $\be \kwedge \be \neq 0$ if and only if $b \neq 1$. In this case $\tfrac{1}{1-b}\be \kwedge \be \in \idem(\mcurv(\std), \opmult)$. If $a \neq 1$ and $b \neq 1$, then $\tfrac{1}{1-b}\be \kwedge \be$ and $\tfrac{1}{1-a}\al\kwedge \al$ are orthogonal idempotents in $(\mcurv(\std), \opmult)$.
\end{lemma}
\begin{proof}
That $|\be \kwedge \be|^{2} = 2(b - 1 )b$ follows from \eqref{alwedgebegawedgesinorm}. (Note that $b = 1$ if and only if $\be = u\tensor u$ for a unit norm $u \in \std$, in which case $\be \kwedge \be = 0$.) 
Specializing \eqref{symalbecmgasi}, \eqref{symalalcmalal}, and \eqref{symalbecmalbe} yields \eqref{hbidempotent} and the remaining claims. 
\end{proof}

\begin{corollary}\label{kwedgeidempotentcorollary}
Let $(\ste, h)$ be a Euclidean vector space of dimension $n > 2$.  If orthogonal idempotents $\al, \be \in \idem(S^{2}\std, \jmult)$ satisfy $a = \tr \al \neq 1$ and $b = \tr \be \neq 1$, then
\begin{align}\label{sabdefined}
\begin{split}
\sB(\al, \be) & =\tfrac{1}{1 - a - b}\left( \tfrac{b}{a - 1}\al \kwedge \al + \tfrac{a}{b - 1}\be\kwedge \be - 2\al \kwedge \be\right)
\end{split}
\end{align}
is an idempotent in $(\mcurvweyl(\std), \opmult)$ satisfying $|\sB(\al, \be)|^{2} = \tfrac{2(a + b -2)ab}{(a + b -1)(a - 1)(b - 1)}$.
In particular, if $b \notin \{1, n-1\}$, $\hat{\be} = h - \be$, and $\hat{b} = \tr \hat{\be}$ then 
\begin{align}\label{sbdefined}
\begin{split}
\sB(\be)  = \sB(\be, h-\be) 
&= \tfrac{2-n}{(b - 1)(n-1 -b)}\left(\be \kwedge \be - \tfrac{2(b - 1)}{n-2}\be \kwedge h + \tfrac{b(b - 1)}{(n-1)(n-2)}h\kwedge h\right)\\
&= \tfrac{2-n}{(b - 1)(n-1 -b)}\tf(\be \kwedge \be)
  = \tfrac{2-n}{( \hat{b} - 1)(n-1 -\hat{b})}\tf(\hat{\be} \kwedge \hat{\be}),
\end{split}
\end{align}
is an idempotent in $(\mcurvweyl(\std), \opmult)$ satisfying $|\sB(\be)|^{2} = \tfrac{2(n-2)(n-b)b}{(n-1)(n-1 - b)(b - 1)}$.
\end{corollary}

\begin{proof}
That $\sB(\al, \be)$ defined by the first equality of \eqref{sbdefined} is an idempotent follows from \eqref{hbidempotent}. That $\rictr(\sB(\al, \be)) =0$ follows from \eqref{rictralkwedgebe}.
The claimed value of $|\sB(\al, \be)|^{2}$ follows from \eqref{sabdefined} and \eqref{alwedgebegawedgesinorm} by straightforward computations.
Because $(h - \be) \circ \be = 0$ and $\tr(h - \be) = n - \tr\be$, that \eqref{sbdefined} is an idempotent in  $(\mcurvweyl(\std), \opmult)$ with $|\sB(\be)|^{2}$ having the claimed value is a special case of the preceding.
The first equality of \eqref{sbdefined} follows upon substituting $\hat{\be} = h - \be$ in \eqref{sabdefined}. Specializing \eqref{tfalkwedgebe} yields the second equality of \eqref{sbdefined}. Finally, $\tf(\be \kwedge \be) = \tf(\hat{\be}\kwedge \hat{\be})$ because $\hat{\be}\kwedge \hat{\be}= \be \kwedge \be + (h- 2\be)\kwedge h$. 
\end{proof}

\begin{corollary}\label{mcurvweylsimplecorollary}
Let $(\ste, h)$ be an $n$-dimensional Euclidean vector space. 
If $n > 3$, then $\mcurvweyl(\std) \opmult \mcurvweyl(\std) = \mcurvweyl(\std)$, and if $n > 4$, then $(\mcurvweyl(\std), \opmult)$ is a simple algebra.
\end{corollary}

\begin{proof}
By Lemma \ref{opweylsubalgebralemma}, $\mcurvweyl(\std)$ is a subalgebra of $(\mcurv(\std), \opmult)$. Let $x, y \in \std$ be orthogonal unit norm vectors. 
Because $\be = x\tensor x + y \tensor y$ satisfies $\be \circ \be = \be$ and $\tr \be = 2$, by Corollary \ref{kwedgeidempotentcorollary}, $|\sB(\be)|^{2} = \tfrac{4(n-2)^{2}}{(n-1)(n-3)} \neq 0$, so $\sB(\be)$ is a nontrivial idempotent in ($\mcurvweyl(\std), \opmult)$. If $n > 3$, this shows $(\mcurvweyl(\std), \opmult)$ is a nontrivial algebra and $\mcurvweyl(\std) \opmult  \mcurvweyl(\std)$ is a nontrivial $O(n)$-submodule of the irreducible $O(n)$-module $\mcurvweyl(\std)$, so must equal $\mcurvweyl(\std)$. If $n > 4$, $SO(n)$ acts on $(\mcurvweyl(\std), \opmult)$ irreducibly by automorphisms, so, by Theorem \ref{simpletheorem}, $(\mcurvweyl(\std), \opmult)$ is simple. 
\end{proof}

\begin{lemma}\label{sxwedgelemma}
Let $(\ste, h)$ be a metric vector space. For $\sX \in \mcurv(\std)$ and $\al \in S^{2}(\std)$,
\begin{align}
\label{xalh}
&\sX \opmult (\al \kwedge h) =  \tfrac{1}{2}\left(\op{\sX}(\al)\kwedge h + \al \kwedge \rictr(\sX)\right),&
&\sX \opmult (h \kwedge h) = \rictr(\sX) \kwedge h.
\end{align}
\end{lemma}

\begin{proof}
For $\al \in S^{2}\std$ and $\sX, \sY \in \mcurv(\std)$,
\begin{align}\label{sxwedgecompute}
\begin{split}
\lb& (\al \kwedge h)\opmult \sX, \sY\ra  =  \lb \al \kwedge h, \sX \opmult \sY\ra =  -2\lb \al, \op{\sX \opmult \sY}(h)\ra = -2\lb \al, \rictr(\sX \opmult \sY)\ra \\
&= -\lb \al, \op{\sX}(\rictr(\sY)) + \op{\sY}(\rictr(\sX))\ra= -\lb \op{\sX}(\al), \op{\sY}(h)\ra + \lb \al, \op{\sY}(\rictr(\sX))\ra = \tfrac{1}{2}\lb \op{\sX}(\al)\kwedge h+ \al \kwedge \rictr(\sX), \sY\ra ,
\end{split}
\end{align}
the first equality by the invariance of $\lb\dum, \dum\ra$, 
the second equality by \eqref{sxalwedgebe}, 
the third and fourth equalities because $\op{\sZ}(h) = \rictr(\sZ)$ for any $\sZ \in \mcurv(\std)$,
the fifth equality again by \eqref{sxalwedgebe}, 
and the last equality by \eqref{opricxy}. By the nondegeneracy of $\lb\dum, \dum \ra$, \eqref{sxwedgecompute} implies the first equality of \eqref{xalh}. Because $\op{\sX}(h) = \rictr(\sX)$, taking $\al = h$ in the first equality of \eqref{xalh} yields its second equality.
\end{proof}

\begin{lemma}\label{alhbehlemma}
Let $(\ste, h)$ be an $n$-dimensional metric vector space. For $\al, \be \in S^{2}(\std)$,
\begin{align}
\label{alhbeh}
\begin{split}
(h \kwedge \al) \opmult (h \kwedge \be) &=\tfrac{2-n}{4}\al \kwedge \be + \tfrac{1}{2} \left(\al \jmult \be  - \tfrac{1}{2}\tr(\al) \be - \tfrac{1}{2}\tr(\be)\al - \tfrac{1}{2}\lb \al, \be \ra h\right)\kwedge h,
\end{split}\\
\label{alhhh}
\begin{split}
(h\kwedge \al) \opmult (h \kwedge h)& = \tfrac{2-n}{2} h \kwedge \al - \tfrac{1}{2}\tr(\al)h\kwedge h,
\end{split}\\
\label{hhhh}
(h\kwedge h)\opmult(h\kwedge h) &= (1-n)h\kwedge h.
\end{align}
\end{lemma}

\begin{proof}
Taking $\sX = \be \kwedge h$ in \eqref{xalh} and simplifying using \eqref{rictralh} and \eqref{symopalbega} yields \eqref{alhbeh}, and \eqref{alhhh} and \eqref{hhhh} are special cases of \eqref{alhbeh}. Alternatively, \eqref{alhbeh}-\eqref{hhhh} are special cases of \eqref{symalbecmgasi}.
\end{proof}

\begin{lemma}\label{mcurvweylspanlemma}
Let $(\ste, h)$ be a Euclidean vector space of dimension $n \geq 4$. 
Let 
\begin{align}
\label{spanset}
\begin{split}
\stw_{1} &= \spn \left\{(x\sprod y) \kwedge (z \sprod w): x, y, z, w \in \std \,\,\text{are pairwise orthogonal}\right\},\\
\stw_{2}  &=  \spn \left\{\al \kwedge \be: \al, \be \in S^{2}_{0}\std, \al \jmult \be = 0\right\},\\
\stw_{3} &= \spn \left\{(x\dwedge y) \cdot (z \dwedge w): x, y, z, w \in \std \,\,\text{are pairwise orthogonal}\right\},\\
\stw_{4}  &=  \spn \left\{\al \cdot \be: \al, \be \in \ext^{2}\std, \al \jmult \be  = 0\right\}.
\end{split}
\end{align}
Then $\mcurvweyl(\std)  = \stw_{1} = \stw_{2} = \stw_{3} = \stw_{4}$.
\end{lemma}
\begin{proof}
By \eqref{rictralkwedgebe} and \eqref{rictralcdotbe}, $\stw_{2}, \stw_{4} \subset \mcurvweyl(\std)$.
By \eqref{alcdotbegacdotsinorm} and \eqref{alwedgebegawedgesinorm}, $|(x\sprod y) \kwedge (z \sprod w)|^{2} = 1/4$ and $|(x\dwedge y) \cdot (z \dwedge w)|^{2} = 12$, so $\stw_{1}$ and $\stw_{3}$ are nontrivial. If $x, y, z, w$ are pairwise orthogonal, then $x \sprod y, z \sprod w \in S^{2}_{0}\std$ satisfy $(x\sprod y)\circ(z \sprod w) = 0$, so $\stw_{1} \subset \stw_{2}$, and $x \dwedge y, z \dwedge w\in \ext^{2}\std$ satisfy $(x\dwedge y)\circ(z \dwedge w) = 0$, so $\stw_{3} \subset \stw_{4}$. Since $\stw_{1}$ and $\stw_{3}$ are nontrivial $O(n)$-submodules of the $O(n)$-irreducible module $\mcurvweyl(\std)$, they equal $\mcurvweyl(\std)$. 
\end{proof}

\begin{theorem}\label{faithfulmultiplicationtheorem}
Let $(\ste, h)$ be a Euclidean vector space of dimension $n \geq 4$. If $\sX \in \mcurv(\std)$ satisfies $\sX \opmult \sY = 0$ for all $\sY \in \mcurv(\std)$, then $\sX = 0$. 
\end{theorem}
\begin{proof}
Because $\sX \opmult \sY = \op{\sX}_{\mcurv(\std)}(\sY)$, an equivalent claim is that the map $\op{\dum}_{\mcurv(\std)}:\mcurv(\std) \to  \eno(\mcurv(\std))$ is injective. If $\sX \opmult \sY = 0$ for all $\sY \in \mcurv(\std)$, then $0 = \lb \sX \opmult \sY, \sZ \ra = \lb \sX, \sY \opmult \sZ\ra$ for all $\sY, \sZ \in \mcurv(\std)$. With $\sY = \sZ = h\kwedge h$, by \eqref{hhhh} and \eqref{sxalwedgebe} this yields $0 = \lb \sX, h \kwedge h\ra = -2\scal(\sX)$. Taking $\sY = h \kwedge h$ and $\sZ = \al \kwedge h$ for $\al \in S^{2}_{0}\std$, by \eqref{alhhh} and \eqref{sxalwedgebe} this yields $0 = 2\lb \sX, (h\kwedge h)\opmult (\al \kwedge h)\ra = (2-n)\lb \sX, \al \kwedge h\ra = 2(n-2)\lb \op{\sX}(h), \al \ra = 2(n-2)\lb \rictrz(\sX), \al\ra$. Since $\al \in S^{2}_{0}\std$ is arbitrary, this shows $\rictr(\sX) = 0$, so $\sX \in \mcurvweyl(\std)$. Finally, if $\al, \be \in S^{2}_{0}\std$, then, by \eqref{alhbeh} and the preceding, $0 = 4\lb \sX , (\al \kwedge h)\opmult (\be \kwedge h)\ra = (2-n)\lb \sX, \al \kwedge \be\ra$. Because the set $\stw_{2}$ of Lemma \ref{mcurvweylspanlemma} spans $\mcurvweyl(\std)$, this shows that $\sX$ is orthogonal to $\mcurvweyl(\std)$, so $\sX = 0$.
\end{proof}

\begin{lemma}\label{subspaceproductslemma}
Let $(\ste, h)$ be an $n$-dimensional Euclidean vector space. The projections onto the $O(h)$-irreducible summands of $\mcurv(\std)$ of $\sX, \sY \in \mcurv(\std)$ satisfy
\begin{align}
\label{psps}\projscal(\sX)\opmult \projscal(\sY) &= -\tfrac{1}{n^{2}(n-1)}\scal(\sX)\scal(\sY)h \kwedge h = \tfrac{1}{n}\scal(\sX)\projscal(\sY) = \tfrac{1}{n}\scal(\sY)\projscal(\sX),\\
\label{pspr}\projscal(\sX)\opmult \projric(\sY) &= -\tfrac{1}{n(n-1)}\scal(\sX)\rictrz(\sY)\kwedge h = \tfrac{(n-2)}{2n(n-1)}\scal(\sX)\projric(\sY),\\
\label{tfps}\tf(\sX)\opmult \projscal(\sY)  &= 0,\\
\label{tfpr}\tf(\sX)\opmult \projric(\sY)  &=  \projric(\tf(\sX)\opmult \sY),\\
\label{prpr}\projric(\sX) \opmult \projric(\sY) &= \tfrac{1}{2-n}\tf(\rictrz(\sX)\kwedge \rictrz(\sY)) + \tfrac{2}{2-n}\projric(\rictrz(\sX)\kwedge \rictrz(\sY)) + \projscal(\rictrz(\sX)\kwedge \rictrz(\sY)).
\end{align}
\end{lemma}

\begin{proof}
The identities \eqref{psps}-\eqref{tfps} follow from \eqref{xalh}, \eqref{alhhh}, and \eqref{hhhh} and the definitions \eqref{projricdefined} of $\projric$ and $\projscal$. By \eqref{projricdefined} and \eqref{xalh}, $\tf(\sX) \opmult \projric(\sY)  = \tfrac{1}{2-n}\op{\tf(\sX)}(\rictrz(\sY))\kwedge h$. By \eqref{opricxy}, $2\rictr(\tf(\sX)\opmult \sY) = \op{\tf(\sX)}(\rictr(\sY)) =\op{\tf(\sX)}(\rictrz(\sY))$, the last equality because $\op{\tf(\sX)}(h) = \rictr(\tf(\sX)) = 0$. Combining the preceding observations yields \eqref{tfpr}. Finally, \eqref{prpr} follows straightforwardly from \eqref{alhbeh}, using $\rictrz(\rictrz(\sX) \kwedge \rictrz(\sY)) = \tf(\rictrz(\sX) \jmult \rictrz(\sY))$ and $\scal(\rictrz(\sX) \kwedge \rictrz(\sY)) = \lb \rictrz(\sX), \rictrz(\sY)\ra$. 
\end{proof}

An irreducible $O(n)$-submodule of $\tensor^{k}\std$ comprises the completely trace-free tensors of a given type (this statement is false for even $k = n$ if $O(n)$ is replaced by $SO(n)$). By Lemma \ref{curvoplemma}, if $G \subset O(n)$ is a Lie subgroup and $\stu \subset \mcurv(\std)$ and $\stw \subset \tensor^{k}\std$ are $G$-submodules, then $\op{\stu}(\stw)$ is a $G$-submodule of $\tensor^{k}\std$, for if $g \in O(n)$, $\sX \in \stu$, and $\sY \in \stw$, $g\cdot \op{\sX}(\sY) = \op{g\cdot \sX}(g\cdot \sY) \in \op{\stu}(\stw)$. In particular, because $\op{\sX}$ preserves type, if $\stu_{1}, \stu_{2} \subset \mcurv(\std)$ are $G$-submodules, then $\op{\stu_{1}}(\stu_{2})$ is a $G$-submodule of $\mcurv(\std)$.
By Theorem \ref{opalgebratheorem}, $\op{\stu_{2}}(\stu_{1}) = \stu_{2}\opmult \stu_{1} = \stu_{1}\opmult \stu_{2}= \op{\stu_{1}}(\stu_{2})$. For example, Corollary \ref{mcurvweylsimplecorollary} shows that $\op{\stu}(\stu) = \stu \opmult \stu = \stu$ for $\stu = \mcurvweyl(\std)$. Theorem \ref{fusiontheorem} describes completely the submodules $\stu_{1}\opmult \stu_{2}$ for $\stu_{1}, \stu_{2}$ among the $O(n)$-irreducible summands of $\mcurv(\std)$. 

\begin{theorem}\label{fusiontheorem}
Let $(\ste, h)$ be an $n$-dimensional Euclidean vector space. 
The products of the $O(n)$-irreducible submodules of $\mcurv(\std)$ satisfy:
\begin{align}
\label{subspaceproductsss} \mcurvscal(\std) &\opmult \mcurvscal(\std) = \mcurvscal(\std), \quad \text{if} \,\,n > 1,\\
\label{subspaceproductsrs} \mcurvric(\std) &\opmult \mcurvscal(\std) = \mcurvric(\std), \quad \text{if} \,\,n > 2, \\
\label{subspaceproductsws} \mcurvweyl(\std) &\opmult \mcurvscal(\std) = \{0\},\\
\label{subspaceproductswr} \mcurvweyl(\std) &\opmult \mcurvric(\std) = \mcurvric(\std), \quad \text{if} \,\,n > 3,  \\ 
\label{subspaceproductsrr} \tf(\mcurvric(\std) &\opmult \mcurvric(\std)) = \mcurvweyl(\std), \quad \text{if} \,\,n > 3,\\
\label{subspaceproductsww}  \mcurvweyl(\std) &\opmult  \mcurvweyl(\std) = \mcurvweyl(\std), \quad \text{if} \,\,n > 3. 
\end{align}
\end{theorem}

\begin{proof}
The containments in \eqref{subspaceproductsss}-\eqref{subspaceproductsrr} of subspaces are consequences of polarizing \eqref{psps}-\eqref{prpr}. The equalities require further justification. The product \eqref{subspaceproductsws} is immediate from \eqref{tfps}. By \eqref{alhhh},  multiplication by $h \kwedge h$, which spans $\mcurvscal(\std)$, is invertible on $\mcurvric(\std)$ when $n > 2$ and on $\mcurvscal(\std)$ when $n > 1$, and this suffices to show the equalities \eqref{subspaceproductsss} and \eqref{subspaceproductsrs}. 

Let $\al \in S^{2}(\std)$. If $\sX \in \mcurvweyl(\std)$, then $\tr \op{\sX}(\al)  = \lb \rictr(\sX), \al \ra = 0$, so, by \eqref{xalh}, $2\sX \opmult (\al \kwedge h) = \op{\sX}(\al)\kwedge h  \in \mcurvric(\std)$.
This shows the containment of $O(n)$-modules, $\mcurvweyl(\std) \opmult \mcurvric(\std) \subset \mcurvric(\std)$. By the irreducibility of $\mcurvric(\std)$, to show equality it suffices to exhibit a nonzero element of $\mcurvweyl(\std) \opmult \mcurvric(\std)$. Let $u, v \in \ste$ be such that $|u|^{2} = 2 = |v|^{2}$ and $\lb u, v \ra = 0$. Then $\al = u\sprod v \in S^{2}_{0}\std$ satisfies $|\al|^{2} = 2$, $2 \al \circ \al = u \tensor u + v \tensor v$, $\lb \al, \al \circ \al \ra = 0$, and $\al \circ \al \circ \al = \al$. From these and $-4\al \kwedge \al = (u \dwedge v)\tensor (u \dwedge v)$ there follow $\rictr(\al \kwedge \al) = \al \circ \al$,  $|\rictr(\al\kwedge \al)|^{2} = 2$, $|\rictrz(\al\kwedge \al)|^{2} = \tfrac{2(n-2)}{n}$, and $\scal(\al \kwedge \al) =  |\al|^{2} = 2$. 
By \eqref{tfalkwedgebe}, $\tf(\al\kwedge \al) = \al \kwedge \al + \tfrac{2}{n-2}(\al \circ \al)\kwedge h - \tfrac{1}{(n-2)(n-1)}|\al|^{2}h \kwedge h$, and by \eqref{symopalbega} and the preceding observations,
$\op{\tf(\al \kwedge \al)}_{S^{2}\std}(\al) = -\tfrac{n-3}{n-1}\al$, so, by \eqref{xalh},
\begin{align}
\tf(\al \kwedge \al) \opmult (\al \kwedge h) & = \tfrac{1}{2}\op{\tf(\al \kwedge \al)}(\al) \kwedge h = \tfrac{3-n}{2(n-1)}\al \kwedge h,
\end{align}
which shows that $\mcurvweyl(\std) \opmult \mcurvric(\std)$ is nontrivial if $n > 3$ and so proves the equality in \eqref{subspaceproductswr}.

Suppose $\dim \std > 3$. Let $x, y, z, w \in \std$ be pairwise orthogonal unit norm vectors. Then $\al = x \sprod y$ and $\be = z \sprod w$ are in $S^{2}_{0}\std$, so $\al \kwedge h, \be \kwedge h \in \mcurvric(\std)$. Because $\al \jmult \be = 0$ and $\lb \al, \be \ra = 0$, by \eqref{alhbeh}, $4(\al \kwedge h)\opmult(\be \kwedge h) = (2-n)\al \kwedge \be$. By \eqref{rictralkwedgebe}, $\rictr(\al \kwedge \be) =0$, so $4(\al \kwedge h)\opmult(\be \kwedge h) = (2-n)\al \kwedge \be \in \mcurvweyl(\std)$. Since, by Lemma \ref{mcurvweylspanlemma}, $\mcurvweyl(\std)$ is spanned by elements of the form $\al \kwedge \be$, this shows the equality in \eqref{subspaceproductsrr}. The equality \eqref{subspaceproductsww} is Corollary \ref{mcurvweylsimplecorollary}.
\end{proof}

\begin{remark}\label{fusionremark}
Lemma \ref{subspaceproductslemma} and Theorem \ref{fusiontheorem} give fusion rules (in the sense of \cite{Hall-Rehren-Shpectorov}) for $(\mcurv(\std), \opmult)$. Precisely, for the idempotent $\sH = \tfrac{1}{1-n}h \kwedge h$, the subspaces $\mcurvscal(\std)$, $\mcurvric(\std)$, and $\mcurvweyl(\std)$ are the eigenspaces of $L_{\opmult}(\sH)$ with eigenvalues $1$, $\tfrac{n-2}{2(n-1)}$, and $0$. Lemma \ref{subspaceproductslemma} shows that their products satisfy the fusion rule $\star: \Phi \times \Phi \to 2^{\Phi}$ indicated in Table \ref{fusionrule}, where $\Phi = \{1, \tfrac{n-2}{2(n-1)}, 0\}$. A subset of $\Phi$ indicates the sum of the eigenspaces corresponding to this subset and an entry in the table means that the $\opmult$ product of the eigenspaces corresponding with $\al, \be \in \Phi$ is contained in the sum of the eigenspaces corresponding with elements of $\al \star \be$.

\begin{table}
\caption{Fusion rules for $(\mcurv(\std), \opmult)$}\label{fusionrule}
\begin{tabular}{|c|c|c|c|}
\hline
$\star$  & $1$  & $0$ & $\tfrac{n-2}{2(n-1)}$\\
\hline
$1$ & $\{1\}$ & $\emptyset$ & $\{\tfrac{n-2}{2(n-1)}\}$ \\
\hline
$0$ & $\emptyset$ & $\{0\}$ & $\{\tfrac{n-2}{2(n-1)}\}$ \\
\hline
$\tfrac{n-2}{2(n-1)}$ &$\{\tfrac{n-2}{2(n-1)}\}$& $\{\tfrac{n-2}{2(n-1)}\}$ & $\{1, 0, \tfrac{n-2}{2(n-1)}\}$ \\
\hline
\end{tabular}
\end{table}

Note that Theorem \ref{fusiontheorem} gives more information than does Table \ref{fusionrule} because it asserts the equalities of the products of subspaces, rather than mere containment relations.
\end{remark}

As an application of Lemma \ref{subspaceproductslemma} there is given a simple proof of \cite[Theorem $2$]{Bohm-Wilking}.
For $\al, \be \in \rea$ define an $O(n)$-equivariant endomorphism $\bw_{\al, \be} \in \eno(\mcurv(\std))$ by
\begin{align}\label{bwdefined}
\bw_{\al, \be}(\sX) = \tf(\sX) + \be \projric(\sX) + \al \projscal(\sX) = \sX + (\be - 1)\projric(\sX) + (\al - 1)\projscal(\sX).
\end{align}
For example, \eqref{prpr} can be rewritten as $(2-n)\projric(\sX) \opmult \projric(\sX) = \bw_{2-n, 2}(\rictrz(\sX) \kwedge \rictrz(\sX))$. 

Note that $\bw_{1, 1} = \Id_{\mcurv(\std)}$. Because $\bw_{\al, \be}\circ \bw_{\bar{\al}, \bar{\be}} = \bw_{\al\bar{\al}, \be\bar{\be}}$, $\bw_{\al, \be}$ is invertible if and only if $\al \neq 0$ and $\be \neq 0$, in which case $\bw_{\al, \be}^{-1} = \bw_{\al^{-1}, \be^{-1}}$. If $\al = 1 + 2(n-1)a$ and $\be = 1 + (n-2)b$, then $\bw_{\al, \be}$ equals the endomorphism called $l_{a, b}$ introduced by C. Böhm and B. Wilking in \cite{Bohm-Wilking}. The reason for working with the parameters $\al$ and $\be$ is that the map $(\al, \be) \in \reat \times \reat \to \bw_{\al, \be} \in \eno(\mcurv(\std))$ is an injective group homomorphism.

The key point of Theorem \ref{bohmwilkingtheorem} for its applications is that \eqref{dalbe} does not depend on $\tf(\sX)$. 

\begin{theorem}[C. Böhm and B. Wilking {\cite[Theorem $2$]{Bohm-Wilking}}]\label{bohmwilkingtheorem}
Let $(\ste, h)$ be an $n$-dimensional Euclidean vector space. If $\al, \be \in \rea\setminus\{0\}$ and $\dbw_{\al, \be}(\sX) =  \bw_{\al, \be}^{-1}\left(\bw_{\al, \be}(\sX) \opmult \bw_{\al, \be}(\sX) \right) - \sX \opmult \sX$, where $\bw_{\al, \be} \in \eno(\mcurv(\std))$ is defined in \eqref{bwdefined}, then
\begin{align}\label{dalbe}
\begin{split}
\dbw_{\al, \be}(\sX) 
& = \tfrac{1-\be^{2}}{n-2}\tf(\rictrz(\sX) \kwedge \rictrz(\sX)) 
+ \projric\left(\tfrac{2(1-\be)}{n-2}\rictrz(\sX) \kwedge \rictrz(\sX) + \tfrac{(n-2)(\al-1)}{n(n-1)}\scal(\sX)\sX \right)\\
&\quad + \projscal\left( (\tfrac{\be^{2}}{\al} - 1)\rictrz(\sX) \kwedge \rictrz(\sX) + \tfrac{\al - 1}{n}\scal(\sX)\sX\right)\\
& = \tfrac{1-\be^{2}}{n-2}\tf(\rictrz(\sX) \kwedge \rictrz(\sX)) + \tfrac{4(\be - 1)}{(n-2)^{2}}\tf\left(\rictrz(\sX) \jmult \rictrz(\sX)\right)\kwedge h + \tfrac{2(1-\al)}{n(n-1)}\scal(\sX)\rictrz(\sX)\kwedge h \\
&\quad + \left( \tfrac{\al - \be^{2}}{n(n-1)\al})|\rictrz(\sX)|^{2} + \tfrac{1- \al}{n^{2}(n-1)}\scal(\sX)\right)h \kwedge h.
\end{split}
\end{align}
\end{theorem}
\begin{proof}
Straightforward calculations using \eqref{psps}-\eqref{prpr} yield
\begin{align}\label{psiabpsiab}
\begin{split}
\bw_{\al, \be}(\sX) &\opmult \bw_{\al, \be}(\sX)  
= \tf(\sX) \opmult \tf(\sX) - \tfrac{\be^{2}}{n-2}\tf(\rictrz(\sX) \kwedge \rictrz(\sX)) \\
&  + \projric\left(-\tfrac{2\be^{2}}{n-2}\rictrz(\sX) \kwedge \rictrz(\sX) + 2\be \tf(\sX)\opmult \sX + \tfrac{\al \be(n-2)}{n(n-1)}\scal(\sX)\sX \right)
\\&\quad 
+ \projscal\left( \be^{2}\rictrz(\sX) \kwedge \rictrz(\sX) + \tfrac{\al^{2}}{n}\scal(\sX)\sX\right).
\end{split}
\end{align}
The special case $\al = 1 = \be$ yields
\begin{align}\label{sxsx}
\begin{split}
\sX \opmult \sX & = \tf(\sX) \opmult \tf(\sX) - \tfrac{1}{n-2}\tf(\rictrz(\sX) \kwedge \rictrz(\sX)) +  \projscal\left( \rictrz(\sX) \kwedge \rictrz(\sX) + \tfrac{1}{n}\scal(\sX)\sX\right)\\
&\quad+ \projric\left(-\tfrac{2}{n-2}\rictrz(\sX) \kwedge \rictrz(\sX) + 2 \tf(\sX)\opmult \sX + \tfrac{(n-2)}{n(n-1)}\scal(\sX)\sX \right).
\end{split}
\end{align}
Combining \eqref{psiabpsiab} and \eqref{sxsx} yields \eqref{dalbe}.
After rewriting \eqref{dalbe} in terms of the parameters $a$ and $b$ and a bit of computation it can be seen that \eqref{dalbe} recovers the conclusion of \cite[Theorem $2$]{Bohm-Wilking}. 
\end{proof}

\section{Characterization of \texorpdfstring{$(\mcurv(\std), \opmult)$}{} when \texorpdfstring{$\dim \ste = 3$}{}}\label{3dsection}
If $\dim \ste = 2$, the $1$-dimensional algebra $(\mcurv(\std), \opmult)$ is generated by $h \kwedge h$ and, by \eqref{hhhh}, it is isomorphic to the field $\rea$ of real numbers by the map sending $-h \kwedge h$ to $1 \in \rea$. This section identifies $(\mcurv(\std), \opmult)$ in a similarly explicit manner when $(\ste, h)$ is a $3$-dimensional Euclidean vector space. Pulling the multiplication $\opmult$ back via an $O(3)$-equivariant linear isomorphisms $\Psi: S^{2}\std \to \mcurv(\std)$ yields an $O(3)$-equivariant commutative multiplication $\tprod$ on $S^{2}\std$. By Lemma \ref{normalizationlemma}, requiring that the associated map $\al \in S^{2}\std \to 2\op{\Psi(\al)}_{\ext^{2}\std} \in \sherm(\ext^{2}\std, \lb\dum, \dum \ra)$ be a Jordan algebra isomorphism determines $\Psi$ uniquely and this determines a standard model $(S^{2}\std, \tprod)$ of $(\mcurv(\std), \opmult)$ realizing it as a deformation of Jordan product $\jmult$ on $S^{2}\std$ by terms built from the metric and the trace. Most of the section is devoted to formulating and proving Theorem \ref{3dcharacterizationtheorem} which specifies algebraic conditions that characterize the resulting algebra $(S^{2}\std, \tprod)$ up to isomorphism. 

Because $S^{2}\std = S^{2}_{0}\std \oplus \spn \{h\}$ and $\mcurv(\std) = \mcurvric(\std)\oplus \mcurvscal(\std)$ are decompositions into $O(3)$-irreducible submodules, by the Schur Lemma the most general $O(3)$-equivariant linear map $S^{2}\std \to \mcurv(\std)$ has the form $\Psi_{p, \tau}(\al)  = \psi_{p,\tau}(\al)\kwedge h$ for some $p,\tau \in \rea$ and $\psi_{p, \tau} \in \eno(S^{2}\std)$ defined by $\psi_{p, \tau}(\al)=  p(\al +\tfrac{\tau - 1}{3} (\tr \al)h)$. Because $\psi_{p, \tau}\circ \psi_{\bar{p}, \bar{\tau}} = \psi_{p\bar{p}, \tau\bar{\tau}}$, $\psi_{p, \tau}$ is invertible if and only if $p\tau \neq 0$, in which case $\psi_{p, \tau}^{-1} = \psi_{p^{-1}, \tau^{-1}}$, and $\Psi_{p, \tau}\circ \psi_{\bar{p}, \bar{\tau}} = \Psi_{p\bar{p}, \tau\bar{\tau}}$. 
By \eqref{rictralkwedgebe},
\begin{align}
&\rictr(\Psi_{p, \tau}(\al))  = -\tfrac{p}{2}\left(\al + \tfrac{4\tau - 1}{3}(\tr \al)h\right) = \psi_{-p/2, 4\tau}(\al),&&
\scal(\Psi_{p, \tau}(\al))  = -2p\tau\tr \al,
\end{align}
so, if $p\tau \neq 0$, 
\begin{align}\label{3dpsiinverse}
\Psi_{p, \tau}^{-1}(\sX) =- \tfrac{2}{p}\left(\rictr(\sX) + \tfrac{1 -4\tau}{12\tau}\scal(\sX)h\right) =- \tfrac{2}{p}\left(\rictrz(\sX) + \tfrac{1}{12\tau}\scal(\sX)h\right)= \psi_{-2/p, 1/4\tau}(\rictr(\sX)).
\end{align}
That $\Psi_{p, \tau}\circ \psi_{\bar{p}, \bar{\tau}} = \Psi_{p\bar{p}, \tau\bar{\tau}}$ means that the pullbacks of $\opmult$ via any $\Psi_{p, \tau}$ with $p\tau \neq 0$ yield isomorphic algebras. Lemma \ref{normalizationlemma} shows that certain natural conditions determine a unique choice. For its statement, observe that $\Psi_{p, \tau}$ determines a linear isomorphism $S^{2}\std \to \sherm(\ext^{2}\std, \lb \dum, \dum \ra)$ by $\al \to \op{\Psi_{p, \tau}(\al)}_{\ext^{2}\std}$, so it makes sense to say that $\Psi_{p, \tau}$ maps rank one elements to rank one elements if $\op{\Psi_{p, \tau}(\al)}_{\ext^{2}\std}$ has rank one whenever $\al$ has rank one, where the \emph{rank} of an element of $S^{2}\std$ means its rank as a bilinear form.

\begin{lemma}\label{normalizationlemma}
Let $(\ste, h)$ be a $3$-dimensional Euclidean vector space. For an $O(3)$-equivariant linear isomorphism $\Psi:S^{2}\std \to \mcurv(\std)$, the following are equivalent:
\begin{enumerate}
\item\label{3dnorm1} $\Psi$ has the form $\Psi_{1, -1/2}$.
\item\label{3dnorm2} $\Psi$ is isometric, $\Psi$ maps $h$ to an idempotent in $(\mcurv(\std), \opmult)$, and $\op{\Psi(\dum)}_{\ext^{2}\std}: S^{2}\std \to \sherm(\ext^{2}\std)$ maps rank one elements to rank one elements.
\item\label{3dnorm3} $2\op{\Psi(\dum)}_{\ext^{2}\std}: (S^{2}\std, \jmult) \to \sherm(\ext^{2}\std, \lb\dum, \dum \ra)$ is a Jordan algebra isomorphism.
\end{enumerate}
\end{lemma}
\begin{proof}
For $\al, \be \in S^{2}\std$, by \eqref{alwedgebegawedgesinorm},
\begin{align}
\begin{split}
 \lb \Psi_{p, \mu}(\al), \Psi_{p, \mu}(\be)\ra & = \lb \psi_{p, \tau}(\al), \psi_{p, \tau}(\be)\ra + \tr \psi_{p,\tau}(\al)\tr \psi_{p, \tau}(\be) \\&= p^{2}\left(\lb \al, \be \ra + \left(\tfrac{4\tau^{2} - 1}{3}\right)\tr \al \tr \be \right),
\end{split}
\end{align}
from which it follows that $ \lb \Psi_{p, \mu}(\al), \Psi_{p, \mu}(\be)\ra  = \lb \al, \be\ra$ if and only if $p, 2\tau \in \{\pm 1\}$. As $\Psi_{p, \tau}(h) = p \tau h\kwedge h$, by \eqref{hhhh}, $\Psi_{p, \tau}(h)$ is idempotent in $(\mcurv(\std), \opmult)$ if and only if moreover $2p\tau = -1$, in which case $\Psi$ has the form $\Psi_{1, -1/2}$ or $\Psi_{-1, 1/2}$.

For $\al \in S^{2}\std$ define $\al^{\sharp} \in \eno(\std)$ by $\al(u)_{i} = \al_{i}\,^{p}u_{p}$ for $u \in \std$. If $u, v \in \std$, then $\al \jmult (u\dwedge v) = \tfrac{1}{2}(\al^{\sharp}(u)\dwedge v + u \dwedge \al^{\sharp}(v))$. 
By \eqref{gasiopal}, $\op{\be \kwedge h}_{\ext^{2}\std} = - L_{\jmult}(\be)$ for $\be \in S^{2}\std$, so 
\begin{align}\label{oppsial}
\op{\Psi_{p, \tau}(\al)}_{\ext^{2}\std} = - p(L_{\jmult}(\al)  + \tfrac{\tau-1}{3}\tr(\al)\id_{\ext^{2}\std}), 
\end{align}
and, hence,
\begin{align}\label{jrdeig}
\op{\Psi_{p, \tau}(\al)}_{\ext^{2}\std}(u\dwedge v) = -\tfrac{p}{2}\left(\al^{\sharp}(u)\dwedge v + u \dwedge \al^{\sharp}(v) + \tfrac{2(\tau-1)}{3}\tr(\al)u\dwedge v\right).
\end{align}
Let $\{u_{1}, u_{2}, u_{3}\}$ be an orthonormal basis of $\std$ comprising eigenvectors of $\al^{\sharp}$ with respective eigenvalues $\mu_{1}$, $\mu_{2}$, and $\mu_{3}$. By \eqref{jrdeig}, $u_{2}\dwedge u_{3}$ is an eigenvector of $\op{\Psi_{p, \tau}(\al)}_{\ext^{2}\std}$ with eigenvalue $-\tfrac{p}{6}((2\tau+1)(\mu_{2} + \mu_{3}) + (2\tau-2)\mu_{1})$ and similarly for permutations of the indices. If $\al$ has rank one, then it can be supposed that $\mu_{2} = \mu_{3} = 0$ and $\mu_{1} \neq 0$ and it results that the eigenvalues of $\op{\Psi_{p, \tau}(\al)}_{\ext^{2}\std}$ are $-\tfrac{p}{3}(\tau-1)\mu_{1}$ with multiplicity $1$ and $-\tfrac{p}{6}(2\tau + 1)\mu_{1}$ with multiplicity $2$. Consequently, that $\op{\Psi_{p, \tau}(\al)}_{\ext^{2}\std}$ have rank one is possible if and only if $\tau = -1/2$ and in this case $\op{\Psi_{p, -1/2}(\al)}_{\ext^{2}\std}$ has rank one for any $\al \in S^{2}\std$ having rank one, for any $p \neq 0$. This completes the proof of the equivalence of \eqref{3dnorm1} and \eqref{3dnorm2}.

Because $\dim \ste = 3$, $\tf(\al \kwedge \be) =0$ for $\al, \be \in S^{2}\std$, so, by \eqref{tfalkwedgebe},
\begin{align}\label{3dalkwedgebe}
\begin{split}
\al \kwedge \be & = \left(-2(\al \jmult\be) + (\tr \al)\be + (\tr \be)\al - \tfrac{1}{2}(\tr \al)(\tr \be)h + \tfrac{1}{2}\lb \al, \be\ra h\right)\kwedge h.
\end{split}
\end{align}
By \eqref{gasiopal} and \eqref{3dalkwedgebe},
\begin{align}\label{gasiopalderiv}
\begin{split}
2L_{\jmult}(\al)\jmult L_{\jmult}(\be) + L_{\jmult}(\al \jmult \be) &= - \op{\al \kwedge \be}_{\ext^{2}\std} + 2L{\jmult}(\al \jmult \be)\\
&= L_{\jmult}\left (\tr \al)\be + (\tr \be)\al - \tfrac{1}{2}(\tr \al)(\tr \be)h + \tfrac{1}{2}\lb \al, \be\ra h\right).
\end{split}
\end{align}
Combining \eqref{oppsial} and \eqref{gasiopalderiv} yields
\begin{align}
\begin{split}
\op{\Psi_{p, \tau}(\al)}_{\ext^{2}\std}&\jmult\op{\Psi_{p, \tau}(\be)}_{\ext^{2}\std} - \tfrac{p}{2}\op{\Psi_{p, \tau}(\al \jmult \be)}_{\ext^{2}\std}\\
& = \tfrac{p^{2}(2\tau + 1)}{6}L_{\jmult} \left(\tr(\al)\be + \tr(\be)\al + \tfrac{2\tau - 5}{6}\tr(\al)\tr(\be)h + \tfrac{1}{2}\lb \al, \be \ra h \right).
\end{split}
\end{align}
It follows that $\tfrac{2}{p}\op{\Psi_{p, \tau}(\dum)}_{\ext^{2}\std}: (S^{2}\std, \jmult) \to \sherm(\ext^{2}\std, \lb\dum, \dum \ra)$ is a Jordan algebra isomorphism if and only if $\tau = -1/2$. This shows the equivalence of \eqref{3dnorm1} and \eqref{3dnorm3}.
\end{proof}

\begin{lemma}\label{3dlemma}
Let $(\ste, h)$ be a $3$-dimensional Euclidean vector space. The pullback of $\opmult$ via $\Psi_{1, -1/2}$ yields on $S^{2}\std$ an $O(3)$-invariant commutative multiplication $\tprod$ having the form 
\begin{align}\label{tproddefined}
\begin{split}
\al \tprod \be &= \al \jmult \be - \tfrac{1}{4}\left((\tr\al)\be + (\tr \be)\al + \lb \al, \be \ra h- (\tr \al \tr \be) h\right).
\end{split}
\end{align}
\end{lemma}
\begin{proof}
Let $p, \tau \in \rea$ satisfy $p\tau \neq 0$.
Write $\Psi = \Psi_{p, \tau}$. Using \eqref{alhbeh}, \eqref{alhhh}, \eqref{hhhh}, and \eqref{3dalkwedgebe} yields
\begin{align}\label{3dpsiproducts}
\begin{split} 
&\Psi(\al)\opmult \Psi(\be) \\&= p^{2}\left(\al \jmult \be - \left(\tfrac{\tau + 2}{6}\right)\left((\tr\al)\be + (\tr\be)\al\right) + \left((\tfrac{-16\tau^{2} + 8\tau + 17}{72})(\tr\al)(\tr\be)- \tfrac{3}{8}\lb \al, \be\ra\right)h \right)\kwedge h,
\end{split}
\end{align}
By \eqref{rictralh} applied to \eqref{3dpsiproducts},
\begin{align}\label{rictr3dpsiproducts}
\begin{split} 
\rictr(&\Psi(\al)\opmult \Psi(\be))\\
& = -\tfrac{p^{2}}{2}\left(\al \jmult \be - \left(\tfrac{\tau + 2}{6}\right)\left((\tr\al)\be + (\tr\be)\al\right) -\left( \tfrac{1}{2}\lb \al, \be\ra + (\tfrac{(8\tau - 5)(2\tau + 1)}{18})(\tr\al)(\tr\be) \right)h\right).
\end{split}
\end{align}
Tracing \eqref{rictr3dpsiproducts} yields $\scal(\Psi(\al)\opmult \Psi(\be)) = \tfrac{p^{2}}{4}\left(\lb \al, \be\ra + (\tfrac{16\tau^{2} - 1}{3})(\tr\al)(\tr\be) \right)$.
Substituting this and \eqref{rictr3dpsiproducts} in \eqref{3dpsiinverse} shows that the pullback $\Psi_{p, \tau}^{-1}(\Psi_{p, \tau}(\al)\opmult \Psi_{p, \tau}(\be))$ equals
\begin{align}\label{3dpsipullback}
\begin{split} 
p\left( \al \tprod \be - \tfrac{(2\tau + 1)}{12}\left((\tr\al)\be + (\tr \be)\al + \tfrac{1}{2\tau}\left(\lb \al, \be \ra  + \tfrac{4\tau - 1}{3}\tr \al \tr \be\right)h \right)\right).
\end{split}
\end{align}
Specializing $(p, \tau) = (1, -1/2)$ yields \eqref{tproddefined}.
\end{proof}


\begin{lemma}\label{3dnonunitallemma}
Let $(\ste, h)$ be a $3$-dimensional Euclidean vector space. The multiplication $\tprod$ on $S^{2}\std$ defined in \eqref{tproddefined} is not unital. In particular, $\tprod$ is not isomorphic to the Jordan product on $S^{2}\std$.
\end{lemma}

\begin{proof}
Were $\al \in S^{2}\std$ a unit, then $4h = 4\al \tprod h = \al + (\tr \al)h$. Tracing this yields $\tr \al = 3$, so $4h = \al + 3h$, implying that $\al = h$. However $h$ is not a unit for, if $\be \in S^{2}_{0}\std$, then $h\tprod \be = \tfrac{1}{4}\be$.
\end{proof}

\begin{remark}
If $\al \in S^{2}\std$ satisfies $\al \circ \al = \al$ and $\tr \al = 1$, then, by Lemma \ref{projectionkwedgelemma}, $-(h-\al)\kwedge (h-\al)$ is idempotent in $(\mcurv(\std), \opmult)$. By the proof of Lemma \ref{projectionkwedgelemma}, $\al \kwedge \al =0$, so $-(h-\al)\kwedge (h-\al) = 2\al \kwedge h - h\kwedge h = \Psi_{1, -1/2}(2\al)$, so, by Lemma \ref{3dlemma}, $2\al$ is idempotent in $(S^{2}\std, \tprod)$. This observation motivates formulating a characterization of $(\mcurv(\std), \opmult)$ in terms of rank one idempotents in $S^{2}\std$.
\end{remark}

\begin{theorem}\label{3dcharacterizationtheorem}
Let $(\ste, h)$ be a $3$-dimensional Euclidean vector space and let $O(3) = O(h)$. On $S^{2}\std$ there is up to algebra isomorphism a unique commutative multiplication $\gprod$ satisfying:
\begin{enumerate}
\item\label{gprod1} $O(3)$ acts on $(S^{2}\std, \gprod)$ by algebra automorphisms.
\item\label{gprod2} $(S^{2}\std, \gprod)$ is metrized by an $O(3)$-invariant inner product.
\item\label{gprod3} $(S^{2}\std, \gprod)$ contains no nonzero square-zero element.
\item\label{gprod4} There is an idempotent in $(S^{2}\std, \gprod)$ having rank one.
\item\label{gprod5} Any idempotent in $(S^{2}\std, \gprod)$ not a multiple of $h$ has rank one.
\item\label{gprod6} For a rank one idempotent $e$ in $(S^{2}\std, \gprod)$, $1/2$ is a multiplicity $3$ eigenvalue of $L_{\gprod}(e)$. 
\end{enumerate} 
The algebra $(S^{2}\std, \gprod)$ is isomorphic to $(S^{2}\std, \tprod)$ where $\tprod$ is defined in \eqref{tproddefined}.
\end{theorem}
\begin{proof}
Let $g$ be an $O(3)$-invariant inner product on $S^{2}\std$. By the Schur Lemma, an $O(3)$-invariant bilinear form on $S^{2}\std$ has the form $k(\al, \be) = A\lb \al, \be\ra + B\tr(\al)\tr(\be)$ for all $\al, \be \in S^{2}\std$. It is positive definite if and only if $A > 0$ and $A + 3B > 0$. A calculation shows
\begin{align}
k(\psi_{p, \tau}(\al), \psi_{p, \tau}(\be)) = p^{2}\left(A\lb \al, \be\ra + \left(\tfrac{\tau^{2} - 1}{3}A + \tau^{2}B\right)\tr(\al)\tr(\be)\right).
\end{align}
Taking $p^{2} = 1/A$ and $\tau^{2} = A/(A+3B)$ yields $k(\psi_{p, \tau}(\al), \psi_{p, \tau}(\be)) = \lb \al, \be\ra$. Hence it can and will be assumed that the $O(3)$-invariant inner product metrizing $(S^{2}\std, \gprod)$ is $\lb\dum, \dum \ra$.

Decomposing $S^{2}(S^{2}\ste)\tensor S^{2}\std$ into its irreducible components, it can be seen that the most general $O(3)$-equivariant commutative bilinear map $\gprod:S^{2}\std \times S^{2}\std \to S^{2}\std$ has the form
\begin{align}\label{generalgprod}
\al \gprod \be & = r\al \jmult \be + s(\tr(\al)\be + \tr(\be)\al) + t\lb \al, \be \ra h + u(\tr \al \tr \be)h,
\end{align}
for some $r, s, t, u \in \rea$. 
Because $\lb \al \gprod \be, \ga\ra - \lb \al, \be \gprod \ga\ra = (s- t)((\tr\al)\lb \be, \ga\ra - (\tr \ga)\lb \al, \be\ra)$, the invariance of $\lb\dum, \dum\ra$ with respect to $\gprod$ is equivalent to $t = s$.

By assumption $O(h)$ stabilizes $h\gprod h$ so there is $c \in \rea$ such that $h \gprod h = ch$. The assumption that there is no nonzero square-zero element implies $c \neq 0$. 
By \eqref{generalgprod}, $ch = h \gprod h = (r + 6s + 3t + 9u)h = (6 + 9s + 9u)h$, so $r + 9s + 9u = c$.

Suppose there is a rank one element of $S^{2}\std$ that is a $\gprod$-idempotent. Then there is $\si \in S^{2}\std$ satisfying $\si \circ \si = \si$, $\tr(\si) = 1 = |\si|^{2}$, and $\si \gprod \si = \la \si$ for some $\la \in \reat$ (so the rank one idempotent is $\la^{-1}\si$). For $\si_{0} =\si - \tfrac{1}{3}h$, there holds
\begin{align}
\la\si_{0} + \tfrac{\la}{3}h = \la \si = \si \gprod \si = (c + 2s)\si + (s+ u)h = (c+ 2s)\si_{0} + \tfrac{1}{3}(c + 5s + 3u)h,
\end{align}
so $c + 2s = \la = c + 5s + 3u$, which imply $s = (\la - c)/2$ and $u = -s$. This shows $\gprod$ has the form
\begin{align}\label{gprodlac}
\al \gprod_{\la, c} \be =c \al \jmult \be + \tfrac{\la - c}{2}\left((\tr\al)\be + (\tr \be)\al + \lb \al, \be \ra h- (\tr \al \tr \be) h\right).
\end{align}
Pulling $\gprod$ back via the dilation by $c^{-1}$ yields the product $\gprod_{\la} = \gprod_{\la, 1}$ on $S^{2}\std$ defined by 
\begin{align}\label{gprodladefined}
&\al \gprod_{\la} \be = \al \jmult \be + \tfrac{\la - 1}{2}\left((\tr\al)\be + (\tr \be)\al + \lb \al, \be \ra h- (\tr \al \tr \be) h\right), &&\al, \be \in S^{2}\std.
\end{align}
This establishes that if $(S^{2}\std, \gprod)$ satisfies conditions \eqref{gprod1}-\eqref{gprod4}, then there is $\la \in \reat$ such that $(S^{2}\std, \gprod)$ is isomorphic as an algebra to $(S^{2}\std, \gprod_{\la})$. For example $\la = 1$ yields the usual Jordan algebra structure $\tprod_{1} = \jmult$. However, a nontrivial $\jmult$-idempotent can have rank $1$, $2$, or $3$. 

For an $h$-orthonormal basis $\{u_{1}, u_{2}, u_{3}\}$ of $\std$, define $\ga = (2\la - 1)(u_{1}\tensor u_{1} + u_{2}\tensor u_{2}) + (1 - \la)(u_{3}\tensor u_{3})$. Observe that $\ga$ does not have rank $1$ provided that $2\la \neq 1$ and $\ga$ is not a multiple of $h$ provided that $3\la \ne 2$. A straightforward calculation using \eqref{gprodladefined} shows that $\ga \gprod_{\la} \ga = (3\la^{2} - 3\la  + 1)\ga$. Since $3\la^{2} - 3\la + 1 \geq 1/4 > 0$ for all $\la \in \rea$, this shows that $(S^{2}\std, \gprod_{\la})$ contains an idempotent that is neither rank one nor a multiple of $h$ provided that $(2\la - 1)(3\la - 2) \neq 0$.
		
Next it is shown that for $\la \in \{1/2, 2/3\}$ the algebra $(S^{2}\std, \gprod_{\la})$ contains no nonzero square-zero element and satisfies \eqref{gprod5}.
Suppose $(\al_{0} + zh)\gprod_{\la} (\al_{0} + z h) = \ep (\al_{0} + zh)$ for $\ep \in \{0, 1\}$ and $\al_{0} \in S^{2}_{0}\std$ and $z \in \rea$. Separating $(\al_{0} + zh)\tprod (\al_{0} + z h) - \ep (\al_{0} + zh)$ into its trace-free and pure trace parts yields the equations
\begin{align}\label{gprodlaidem0}
\begin{aligned}
&\al_{0} \circ \al_{0} - \tfrac{1}{3}|\al_{0}|^{2}h + ((3\la - 1)z - \ep)\al_{0} =0,& &z^{2} - \ep z + \tfrac{3\la - 1}{6}|\al_{0}|^{2} = 0.
\end{aligned}
\end{align}
If $\ep =0$ and $\la > 1/3$, the second equation of \eqref{gprodlaidem0} implies $z =0$ and $\al_{0} = 0$; this shows $(S^{2}\std, \gprod_{\la})$ contains no nonzero square-zero element if $\la \in \{1/2, 2/3\}$. Henceforth assume $\ep = 1$.
There are $h$-orthonormal $u_{1}, u_{2}, u_{3} \in \std$ and $x_{1}, x_{2} \in \rea$ such that $\al_{0}= x_{1}u_{1}\tensor u_{1} + x_{2}u_{2}\tensor u_{2} - (x_{1} + x_{2})u_{3}\tensor u_{3}$ and $|\al_{0}|^{2}= 2(x_{1}^{2} + x_{2}^{2} +x_{1}x_{2})$. Contracting the first equation of \eqref{gprodlaidem0} with each of $u_{1}\tensor u_{1}$ and $u_{2}\tensor u_{2}$ yields the equations
\begin{align}\label{gprodlaidem1}
\begin{aligned}
&0 = x_{1}^{2} - 2x_{2}^{2} - 2x_{1}x_{2} + 3((3\la -1)z - 1)x_{1},&\\&
0  = -2x_{1}^{2} +x_{2}^{2} - 2x_{1}x_{2} + 3((3\la -1)z - 1)x_{2}.
\end{aligned}
\end{align}
Appropriate linear combinations of the equations \eqref{gprodlaidem1} yield
\begin{align}\label{gprodlaidem2}
&(x_{1} - x_{2})(x_{1} + x_{2} + (3\la- 1)z - 1) = 0,& &(2x_{2} + x_{1})((3\la- 1)z - 1 - x_{1}) = 0,
\end{align}
while the second equation of \eqref{gprodlaidem0} becomes
\begin{align}\label{gprodlaidem3}
0 & = z^{2} -  z + \tfrac{3\la - 1}{3}(x_{1}^{2} + x_{2}^{2} + x_{1}x_{2}). 
\end{align}
If $x_{1} = x_{2}$, the second equation of \eqref{gprodlaidem2} yields $x_{1}((3\la - 1)z -1 - x_{1}) = 0$, so either $x_{1} = 0$, in which case $z \in \{0, 1\}$ and $\ga$ is a multiple of $h$, or $x_{1} = (3\la - 1)z - 1$. In the latter case, \eqref{gprodlaidem3} yields 
\begin{align}\label{gprodlaidem4} 
\begin{split}
0 &=z^{2} - z + (3\la - 1)x_{1}^{2} = ((3\la - 1)^{3} + 1)z^{2} - (2(3\la - 1)^{2} + 1)z + (3\la -1)\\
& = 9\la(3\la^{2} - 3\la + 1)z^{2} + 3(6\la^{2} - 4\la + 1)z + 3\la - 1\\
&^= (3(3\la^{2} - 3\la + 1)z - (3\la - 1))(3\la z - 1).
\end{split}
\end{align}
If $z = 1/(3\la)$ there results $\al_{0} + zh = \la^{-1}u_{3}\tensor u_{3}$, which has rank $1$. If $z = \tfrac{3\la - 1}{3(3\la^{2} - 3\la + 1)}$ there results 
\begin{align}
\al_{0} + z h = \tfrac{1}{(3\la^{2} - 3\la + 1)}\left((2\la - 1)(u_{1}\tensor u_{1} + u_{2}\tensor u_{2}) + (1 - \la)u_{3}\tensor u_{3}\right).
\end{align}
As observed before, this element is idempotent, but it has rank one if $\la = 1/2$ and is a multiple of $h$ if $\la = 2/3$.

If $x_{1} \neq x_{2}$, the first equation of \eqref{gprodlaidem2} yields $(3\la - 1)z = 1 - x_{1} - x_{2}$ and the second equation of \eqref{gprodlaidem2} becomes $0 = (2x_{2} + x_{1})(2x_{1} + x_{2})$. Without loss of generality it can be assumed that $x_{2} = -2x_{1}$ (otherwise interchange the indices $1$ and $2$) in which case $(3\la - 1)z - 1  = - x_{1} - x_{2} =  x_{1}$. As in \eqref{gprodlaidem4}, in \eqref{gprodlaidem3} this yields
\begin{align}
\begin{split}
0 &=z^{2} - z + \tfrac{(3\la - 1)}{3}(x_{1}^{2}  + 4x_{1}^{2} - x_{1}^{2})= z^{2} - z + (3\la - 1)x_{1}^{2}\\&= (3(3\la^{2} - 3\la + 1)z - (3\la - 1))(3\la z - 1).
\end{split}
\end{align}
If $z = 1/(3\la)$ there results $x_{1} = -1/(3\la)$ and $x_{2} = 2/(3\la)$ so that $\al_{0} + zh = \la^{-1}u_{2}\tensor u_{2}$, which has rank $1$. If $z = \tfrac{3\la - 1}{3(3\la^{2} - 3\la + 1)}$, and $\la = 1/2$, then $z = 2/3$, $x_{1} = -2/3$, and $x_{2} = 4/3$, which yields $\al_{0} + zh = 2u_{2}\tensor u_{2}$, which has rank one; while if $\la = 2/3$, then $z = 1$, $x_{1} = 0$, and $x_{2} =0$, which yields $\al_{0} + zh = h$. This shows that $\gprod_{\la}$ satisfies \eqref{gprod5} for $\la \in \{1/2, 2/3\}$.

For an $h$-orthonormal basis $\{u_{1}, u_{2}, u_{3}\}$ of $\std$, define $e_{i} = u_{i}\tensor u_{i} \in S^{2}\std$ and $f_{i\join j} = \sqrt{2}u_{i}\sprod u_{j} \in S^{2}_{0}\std$, where distinct indices take distinct values from $\{1, 2, 3\}$ and $i\join j$ is the complement of $\{i, j\}$ in $\{1, 2, 3\}$ (so $1 \join 2 = 3$ and $i\join (i\join j) = j$). Then $\{e_{i}, f_{i}: 1 \leq i \leq 3\}$ is an orthonormal basis of $S^{2}\std$. Calculations using \eqref{gprodlac} show 
\begin{align}\label{gprodrelations}
\begin{aligned}
&e_{i}\gprod_{\la} e_{i} = \la e_{i},&& e_{i}\gprod_{\la} e_{j} = \tfrac{1-\la}{2}e_{i\join j},&& e_{i}\gprod_{\la} f_{i} = \tfrac{\la - 1}{2}f_{i},&\\
& e_{i}\gprod_{\la} f_{j} = \tfrac{\la}{2}f_{j},&& f_{i}\gprod_{\la} f_{i} = \tfrac{\la}{2}h - \tfrac{1}{2}e_{i}, &&f_{i}\gprod_{\la} f_{j} = \tfrac{1}{2\sqrt{2}}f_{i\join j}. 
\end{aligned}
\end{align}
By \eqref{gprodrelations}, the basis $e_{i}$, $f_{j}$, $f_{i \join j}$, $e_{j} + e_{i\join j}$, $f_{i}$, and $e_{j} - e_{i\join j}$ comprises eigenvectors of of $L_{\gprod_{\la}}(\la^{-1}e_{i})$ having respective eigenvalues $1$, $1/2$, $1/2$, $\tfrac{1-\la}{2\la}$, $\tfrac{\la - 1}{2\la}$, and $\tfrac{\la - 1}{2\la}$. The four values $1$, $1/2$, $\tfrac{1-\la}{2\la}$, and $\tfrac{\la - 1}{2\la}$ are pairwise distinct if $\la \notin \{-1, 1/3, 1/2\}$. The $1/2$ eigenspace always contains $f_{j}$ and $f_{i \join j}$, and the multiplicity of the eigenvalue $1/2$ is greater than $2$ if and only if $(1 - \la)/(2\la) = 1/2$, which occurs if and only if $\la = 1/2$. The eigenvalues of $L_{\gprod_{1/2}}(2e_{i})$ are $1$, with eigenspace spanned by $e_{i}$; $1/2$ with eigenspace spanned by $f_{j}$, $f_{i \join j}$, and $e_{j} + e_{i\join j}$; and $-1/2$, with eigenspace spanned by $f_{i}$ and $e_{j} - e_{i\join j}$. Because any rank one idempotent is in the $O(3)$ orbit of a multiple of $e_{1}$, this proves \eqref{gprod6} and completes the proof.
\end{proof}

\begin{remark}
That $1/2$ is an eigenvalue of multiplicity at least $2$ of $L_{\gprod_{\la}}(\la^{-1}e_{i})$ is a consequence of the $O(3)$-invariance of $\gprod_{\la}$. Differentiating along a one-parameter family of rank one idempotents passing through $\la^{-1}e_{i}$ shows that a vector tangent to the $O(3)$-orbit passing through $\la^{-1}e_{i}$ is an eigenvector of $L_{\gprod_{\la}}(\la^{-1}e_{i})$ with eigenvalue $1/2$. The content of \eqref{gprod6} of Theorem \ref{3dcharacterizationtheorem} is that, for $\la = 1/2$, the $1/2$ eigenspace of $L_{\gprod_{1/2}}(2e_{i})$ has an extra third dimension.
\end{remark}

\begin{remark}
The conditions in Theorem \ref{3dcharacterizationtheorem} are all necessary and serve to exclude certain particularly symmetric $O(3)$-invariant algebra structures on $S^{2}\std$.

As mentioned in the proof, the algebra $\sherm(\ste, h)$ satisfies conditions \eqref{gprod1}-\eqref{gprod4} of Theorem \ref{3dcharacterizationtheorem}, but fails \eqref{gprod5} (and also \eqref{gprod6}) because it contains rank two idempotents.

Consider $S^{2}\std$ with the $O(3)$-invariant multiplication 
\begin{align}
\al \ctimes \be = \al \jmult \be - \tfrac{5}{12}\left(\tr(\al)\be + \tr(\be)\al + \lb \al, \be\ra h\right) + \tfrac{4}{9}\tr(\al)\tr(\be)h. 
\end{align}
It can be checked that $(S^{2}\std, \ctimes)$ is exact and Killing metrized with $\tau_{\ctimes} = \tfrac{5}{8}\lb \dum, \dum \ra$, so that $\Aut(S^{2}\std, \ctimes) = O(3)$ (modulo inconsequential scalar factors, this algebra is what is in \cite{Fox-simplicial} called the conformal extension of $(\sherm_{0}(\ste, h), \jrd)$). That $(S^{2}\std, \ctimes)$ contains no square-zero element can be checked directly or follows from results in \cite{Fox-simplicial}. However, $(S^{2}\std, \ctimes)$ contains no rank one idempotent, for if $\si \in S^{2}\std$ has rank $1$, then $36 \si \ctimes \si = 6\tr(\si)\si + \tr(\si)^{2}h$. 

As the proof of Theorem \ref{3dcharacterizationtheorem} shows, \eqref{gprod6} excludes $(S^{2}\std, \gprod_{2/3})$ which is an interesting algebra because, in addition to satisfying \eqref{gprod1}-\eqref{gprod5}, it is Killing metrized, as follows from Lemma \ref{3dlapropertieslemma}.
\end{remark}

\begin{lemma}\label{3dlapropertieslemma}
Let $(\ste, h)$ be a $3$-dimensional Euclidean vector space. For $\la \in \reat$, let $\gprod_{\la}$ be as in \eqref{gprodladefined}. 
\begin{enumerate}
\item\label{gprod8} $(S^{2}\std, \gprod_{\la})$ is simple if $\la \notin \{1, 1/3\}$.
\item\label{gprod9} There holds $\tr L_{\gprod_{\la}} = \tfrac{5\la - 1}{2}\lb h, \dum \ra$. In particular, $\gprod_{\la}$ is exact if and only if $\la = 1/5$.
\item\label{gprod10} For all $\al, \be \in S^{2}\std$, there hold 
\begin{align}\label{gprodlatraceforms}
\begin{aligned}
4(5\la - 1)\tr L_{\gprod_{\la}}(\al \gprod_{\la} \be) + 4(\la - 1)\tr L_{\gprod_{\la}}(\al)\tr L_{\gprod_{\la}}(\be) = (5\la - 1)^{2}(3\la - 1)\lb \al, \be \ra,\\
4(5\la - 1)^{2}\tau_{\gprod_{\la}}(\al , \be) -4\la (3\la - 2)\tr L_{\gprod_{\la}}(\al)\tr L_{\gprod_{\la}}(\be) = (5\la - 1)^{2}(6\la^{2} - 4\la + 3)\lb \al, \be \ra.
\end{aligned}
\end{align}
In particular, $(S^{2}\std, \gprod_{\la})$ is Killing metrized if and only if $\la = 2/3$.
\item\label{gprod11} If $\la \neq 1/5$, the automorphism group $\Aut(S^{2}\std, \gprod_{\la})$ is isomorphic to $SO(3)$ in its induced action on $S^{2}\std$.
\end{enumerate}
\end{lemma}

\begin{proof}
Suppose $\ideal \subset S^{2}\std$ is a $\gprod_{\la}$-ideal and there are $\al_{0} \in S^{2}_{0}\std$ and $z \in \rea$ such that $0 \neq \al_{0} + zh \in \ideal$. Suppose $\la \notin\{1, 1/3\}$. Then $\tfrac{3\la - 1}{2}\al_{0} + zh = h\gprod_{\la}(\al_{0} + zh) \in \ideal$, so $\tfrac{3}{2}(1 - \la)zh = h\gprod_{\la}(\al_{0} + zh)  + \tfrac{1 - 3\la}{2}(\al_{0} + zh) \in \ideal$. If $z \neq 0$ this implies $h \in \ideal$. Otherwise there is $0 \neq \al_{0} \in \ideal \cap S^{2}_{0}\std$. As $4h\gprod_{\la}(\al_{0}\gprod_{\la}\al_{0}) = 2(3\la - 1)\al_{0}\gprod_{\la}\al_{0} - (\la - 1)(3\la - 1)|\al_{0}|^{2}h$, in this case, $(1-\la)(3\la - 1)|\al_{0}^{2}h = 4h\gprod_{\la}(\al_{0}\gprod_{\la}\al_{0}) + 2(1 - 3\la)\al_{0}\gprod_{\la}\al_{0} \in \ideal$. Because $|\al_{0}|^{2} \neq 0$, this implies $h \in \ideal$. In either case, because $L_{\gprod_{\la}}(h)$ is invertible, that $h \in \ideal$ implies $\ideal = S^{2}\std$. This shows $(S^{2}\std, \gprod_{\la})$ is simple if $\la \notin \{1, 1/3\}$. (When $\la = 1/3$, $L_{\gprod_{1/3}}(h)$ annihilates $S^{2}_{0}\std$, so $h$ generates a proper ideal, while when $\la = 1$, $\gprod_{1} = \jmult$ and $S^{2}_{0}\std$ is a proper ideal.)

Claims \eqref{gprod9} and \eqref{gprod10} can be proved by straightforward though tedious calculations using \eqref{gprodrelations}. Their principal relevance here is to prove \eqref{gprod10}. An alternative approach is the following. Identify $S^{2}\std_{0} \oplus \rea$ with $S^{2}\std$ via the map $(\al_{0}, a) \to \al_{0} + a h$. With respect to this identification the multiplication endomorphism $L_{\gprod_{\la}}(\al_{0}, a)$ has the block form
\begin{align}\label{gprodlablock}
L_{\gprod_{\la}}(\al_{0}, a) = \begin{pmatrix} L_{\jrd}(\al_{0}) + \tfrac{3\la - 1}{2}a\Id_{S^{2}_{0}\std} & \tfrac{3\la - 1}{2}\al_{0}\\ \tfrac{3\la - 1}{6}\lb \al_{0}, \dum\ra & a \end{pmatrix},
\end{align}
in which $\jrd$ is the trace-free Jordan product $\al_{0}\jrd \be_{0} = \al_{0}\jmult\be_{0} - \tfrac{1}{3}\lb \al_{0}, \be_{0}\ra h$. Claim \eqref{gprod9} follows by tracing \eqref{gprodlablock}, while \eqref{gprodlatraceforms} follow by straightforward computations using \eqref{gprodlablock} and the fact that $\tau_{\jrd}(\al_{0}, \be_{0}) = \tfrac{7}{12}\lb \al_{0}, \be_{0}\ra$, which is proved in \cite{Fox-simplicial}.

Suppose $\la \neq 1/5$. By \eqref{gprod10}, an algebra automorphism $\phi$ of $\gprod_{\la}$ preserves $\lb\dum, \dum \ra$. By \eqref{gprod9}, $\tfrac{5\la - 1}{2}\lb \phi(h), \dum \ra = \tr L_{\gprod_{\la}}(\phi(h))  = \tr L_{\gprod_{\la}}(h)= \tfrac{5\la - 1}{2}\lb h, \dum \ra$, so $\phi(h) = h$. It follows that $0 = \phi(\al)\gprod_{\la}\phi(\be) - \phi(\al \gprod_{\la}\be) = \phi(\al)\jmult \phi(\be) - \phi(\al \jmult \be)$, so that $\phi$ is an automorphism of the Jordan algebra $(S^{2}\std, \jmult)$. Every automorphism of $(S^{2}\std, \jmult)$ is given by the action of an element of $O(3)$ \cite[Theorem VII.$13$]{Koecher}. This proves \eqref{gprod10}.
\end{proof}

\begin{remark}
The $\la = 1$ case of \eqref{gprod9} and \eqref{gprod10} of Lemma \ref{3dlapropertieslemma} recovers identities for the Jordan algebra $(\sherm(\ste, h), \jmult)$ that can be found in \cite[Proposition III.4.2 and Lemma VI.1.1]{Faraut-Koranyi}.
\end{remark}

Corollary \ref{3dpropertiescorollary} summarizes the result of combining Theorem \ref{3dcharacterizationtheorem} and Lemma \ref{3dlapropertieslemma} for $\tprod$.
\begin{corollary}\label{3dpropertiescorollary}
Let $(\ste, h)$ be a $3$-dimensional Euclidean vector space. The map $\Psi: (S^{2}\std, \tprod) \to (\mcurv(\std), \opmult)$ defined by $\Psi(\al) = (\al - \tfrac{1}{2}\tr(\al)h)\kwedge h)$ is an algebra isomorphism, where the commutative multiplication $\tprod$ on $S^{2}\std$ is defined by
\begin{align}\label{tproddefined0}
\al \tprod \be &= \al \jmult \be - \tfrac{1}{4}\left((\tr\al)\be + (\tr \be)\al + \lb \al, \be \ra h- (\tr \al \tr \be) h\right)
\end{align}
and is characterized as the unique, up to algebra isomorphism, commutative multiplication on $S^{2}\std$ satisfying: 
\begin{enumerate}
\item\label{tprod1} $O(3)$ acts on $(S^{2}\std, \tprod)$ by algebra automorphisms.
\item\label{tprod2} $(S^{2}\std, \tprod)$ is metrized by an $O(3)$-invariant inner product.
\item\label{tprod3} $(S^{2}\std, \tprod)$ contains no nonzero square-zero element.
\item\label{tprod4} There is an idempotent in $(S^{2}\std, \tprod)$ having rank one.
\item\label{tprod5} Any idempotent in $(S^{2}\std, \tprod)$ not a multiple of $h$ has rank one.
\item\label{tprod6} For a rank one idempotent $e$ in $(S^{2}\std, \tprod)$, the spectrum of $L_{\tprod}(e)$ contains $1/2$ with multiplicity $3$.
\end{enumerate} 
Moreover, the multiplication $\tprod$ has the following properties:
\begin{enumerate}
\item\label{tprod5b} An idempotent in $(S^{2}\std, \tprod)$ distinct from $h$ has the form $\al = 2u \tensor u$ for a unit norm $u \in \std$.
\item\label{tprod7} For an idempotent $\al$ as in \eqref{tprod5}, the eigenvalues of $L_{\tprod}(\al)$  are $1$, with multiplicity $1$, $1/2$ with multiplicity $3$, and $-1/2$ with multiplicity $2$.
\item\label{tprod8} $(S^{2}\std, \tprod)$ is simple.
\item\label{tprod9} For all $\al, \be \in S^{2}\std$,
\begin{align}
\tfrac{16}{3}\left(\tr L_{\tprod}(\al \tprod \be) - \tfrac{1}{3}\tr L_{\tprod}(\al)\tr L_{\tprod}(\be)\right) = \lb \al, \be \ra = \tfrac{8}{5}\left(\tau_{\tprod}(\al, \be) + \tr L_{\tprod}(\al)\tr L_{\tprod}(\be)\right) 
\end{align}
\item\label{tprod10} The full automorphism group of $(S^{2}\std, \tprod)$ is $O(3)$ in its induced action on $S^{2}\std$.
\end{enumerate}
\end{corollary}

The idempotents in $(S^{2}\std, \tprod)$ are parameterized by the disjoint union of a point, corresponding with $h$, and a projective plane, corresponding with the $O(3)$-orbit of a rank one symmetric bilinear form having norm $2$.

\section{Characterization of the subalgebra of anti-self-dual Weyl tensors when \texorpdfstring{$\dim \ste = 4$}{}}\label{4dsection}
This section proves Theorem \ref{mcurvweyltheorem}, that shows that, when $\dim \ste = 4$, the subalgebra of anti-self-dual Weyl tensors is isomorphic to the space of trace-free endomorphisms of a $3$-dimensional vector space equipped with the trace-free Jordan product. The proof is conceptual in the sense that it relies on the description of $\opmult$ in terms of curvature operators. On the other hand, the approach is special to $\dim \ste = 4$. Lemma \ref{4dsubalgebralemma} yields an alternative proof, that, while more computational, is based on an approach viable in all dimensions.

Let $(\ste, h)$ be an $n$-dimensional Euclidean vector space. Let $\ep_{i_{1}\dots i_{n}}$ be the volume $n$-form determined a by choice of orientation of $\ste$ and evaluating to $1$ when paired with the wedge product of the vectors of an ordered $h$-orthonormal basis consistent with the chosen orientation. The polyvector $\ep^{i_{1}\dots i_{n}}$ obtained by raising indices satisfies
$\ep_{i_{1}\dots i_{p}k_{1}\dots k_{n-p}}\ep^{j_{1}\dots j_{p}k_{1}\dots k_{n-p}} = p!(n-p)!\delta_{[i_{1}}\,^{[j_{1}}\delta_{i_{2}}\,^{j_{2}}\dots\delta_{i_{p-1}}\,^{j_{p-1}} \delta_{i_{p}]}\,^{j_{p}]}$.
Suppose $\dim \ste = 4$. Then this yields the identities
\begin{align}\label{epid4}
&\ep_{abcd}\ep^{abcd} = 4!,&& \ep_{iabc}\ep^{jabc} = 6\delta_{i}\,^{j}, &
&\ep_{ijab}\ep^{klab} = 4\delta_{[i}\,^{[k}\delta_{j]}\,^{l]},&& \ep_{ijkp}\ep^{abcp} = 6\delta_{[i}\,^{[a}\delta_{j}\,^{b}\delta_{k]}\,^{c]}.
\end{align}
The Hodge star operator $\star \in \eno(\ext^{2}\std)$ is defined by $(\star \al)_{ij} = \tfrac{1}{2}\ep_{ij}\,^{pq}\al_{pq}$. By \eqref{epid4}, $\star \circ \star = \Id_{\ext^{2}\std}$, so $\ext^{2}\std$ decomposes into the two three-dimensional $\star$-eigenspaces $\pmdf\std$, denominated the self-dual and anti-self-dual two-forms on $\ste$. Note that, with the conventions used here, $\al \dwedge \star \be = \tfrac{1}{2}\lb \al, \be\ra \ep$.

\begin{lemma}\label{albecontractlemma}
Let $(\ste, h, \ep)$ be an oriented $4$-dimensional Euclidean vector space. 
\begin{enumerate}
\item\label{4dso1} If $\al, \be \in \sdf\std$ or $\al, \be \in \asdf\std$, then $ (\al \jmult \be) =  -\tfrac{1}{4}\lb \al, \be \ra h$. 
\item\label{4dso3} For $\al, \be \in \ext^{2}\std$, $[\star \al, \star \be] = [\al, \be]$ and $\star [\al, \be] = [\star \al, \be]$. In particular, $[\sdf\std, \asdf\std] =\{0\}$ and $\lb \sdf\std, \asdf\std\ra = \{0\}$.
Under the identification of $(\ext^{2}\std, [\dum, \dum])$ with $\so(4)$, the subspaces $\pmdf\std$ are commuting Lie ideals identified with commuting ideals of $\so(4)$ isomorphic to $\so(3)$.
\end{enumerate}
\end{lemma}
\begin{proof}
For $\al, \be \in \ext^{2}\std$, using \eqref{epid4} yields
\begin{align}\label{staralstarbecomp}
\begin{split}
((\star \be)\circ (\star \al))_{ij} & = -(\star \al)_{pi}(\star \be)^{pk}h_{kj} 
= -\tfrac{1}{4}\ep_{iabp}\ep^{kcdp}\al^{ab}\be_{cd}h_{kj}= -\tfrac{3}{2}\delta_{[i}\,^{[k}\delta_{a}\,^{c}\delta_{b]}\,^{d]}\al^{ab}\be_{cd}h_{kj} \\
&= -\tfrac{1}{2}\lb \al, \be \ra h_{ij} - (\al \circ \be)_{ij}.
\end{split}
\end{align}
Symmetrizing \eqref{staralstarbecomp} in $\al$ and $\be$ yields \eqref{4dso1}. Antisymmetrizing \eqref{staralstarbecomp} in $\al$ and $\be$ yields $[\star \al, \star \be] = [\al, \be]$. Because $\star$ is self-adjoint and $\ad(\al) = [\al, \dum]$ is anti-self-adjoint, for any $\ga \in \ext^{2}\std$,
\begin{align}
\lb \star [\al, \be], \ga\ra = \lb [\al, \be], \star \ga\ra = -\lb \be, [\al, \star \ga]\ra = -\lb \be, [\star \al, \ga]\ra = \lb [\star \al, \be], \ga\ra,
\end{align}
showing that $\star [\al, \be] = [\star \al, \be]$. It follows that $\pmdf\std$ are commuting Lie ideals in $\ext^{2}\std$. This shows the first part of \eqref{4dso3}. The claimed isomorphisms with $\so(4)$ and $\so(3)$ follow from standard representation theory and are omitted. 
\end{proof}

\begin{lemma}\label{starmcurvweylemma}
Let $(\ste, h, \ep)$ be an oriented $4$-dimensional Euclidean vector space. For $\sX \in \mcurvweyl(\std)$, define $(\star \sX)_{ijkl} = \tfrac{1}{2}\ep_{ij}\,^{ab}\sX_{abkl}$. Then $(\star \sX)_{ijkl} \in \mcurvweyl(\std)$, and $\star:\mcurvweyl(\std) \to \mcurvweyl(\std)$ is a linear involution satisfying $\op{\star \sX}_{\ext^{2}\std} = \star \circ \op{\sX}_{\ext^{2}\std} = \op{\sX}_{\ext^{2}\std}\circ \star$.
\end{lemma}

\begin{proof}
By definition $(\star \sX)_{ijkl} = (\star \sX)_{[ij]kl} = (\star\sX)_{ij[kl]}$, so $\star \sX \in S^{2}(\ext^{2}\std)$. To show $\star \sX \in \mcurv(\std)$ it suffices to show that $(\star \sX)_{[ijk]l}$ vanishes. For $\sX \in \mcurvweyl(\std)$, if $n = \dim \ste$, tracing $\sX_{[ij}\,^{[ab}\delta_{k]}\,^{c]}$ in $k$ and $c$ yields a multiple of $(n-4)\sX_{ij}\,^{ab}$. Since this vanishes if $n = 4$, and, by \cite[Theorem $5.7.$A]{Weyl}, an $O(n)$-module of covariant trace-free tensors on an $n$-dimensional vector space having symmetries corresponding to a Young diagram for which the sum of the lengths of the first two columns is greater than $n$ is trivial, when $\dim \ste = 4$, 
\begin{align}\label{lovelockid}
\sX_{[ij}\,^{[ab}\delta_{k]}\,^{c]} = 0
\end{align}
for any $\sX \in \mcurvweyl(\std)$.
(The identity is sometimes called a \emph{Lovelock identity} because similar identities generalizing it are discussed in \cite{Lovelock}.)
Contracting \eqref{lovelockid} with $\ep_{abcl}$ yields $0 = \ep_{abcl}\sX_{[ij}\,^{[ab}\delta_{k]}\,^{c]} = -2(\star \sX)_{l[ijk]}$, so $\star \sX \in \mcurv(\std)$. There holds $\rictr(\star\sX)_{jk}  = \tfrac{1}{2}\ep_{j}\,^{pab}\sX_{abpk} = \tfrac{1}{2}\ep_{j}\,^{abp}\sX_{[abp]k} = 0$, so $\star \sX \in \mcurvweyl(\std)$. 
For $\al_{ij} \in \ext^{2}\std$,
\begin{align}\label{staropsx1}
\begin{split}
(\star \op{\sX}_{\ext^{2}\std}(\al))_{ij} &= \tfrac{1}{2}\ep_{ij}\,^{ab}\op{\sX}_{\ext^{2}\std}(\al)_{ab} = -\tfrac{1}{4}\ep_{ij}\,^{ab}\al^{pq}\sX_{abpq} = -\tfrac{1}{2}\al^{pq}(\star \sX)_{ijpq} = \op{\star \sX}_{\ext^{2}\std}(\al)_{ij}.
\end{split}
\end{align}
That $\star \sX \in \mcurvweyl(\std)$ implies $(\star\sX)_{ijkl}$ has all the other symmetries that this inclusion implies, for example $\tfrac{1}{2}\ep_{ij}\,^{ab}\sX_{abkl} = (\star \sX)_{ijkl} = (\star \sX)_{klij}  = \tfrac{1}{2}\ep_{kl}\,^{ab}\sX_{abij}$. 
It follows that
\begin{align}
\begin{split}
\op{\sX}_{\ext^{2}\std}(\star \al)_{ij} &= -\tfrac{1}{2}(\star \al)^{ab}\sX_{ijab} = -\tfrac{1}{4}\ep^{abpq}\al_{pq}\sX_{ijab}= -\tfrac{1}{4}\al^{pq}\ep_{pq}\,^{ab}\sX_{abij}\\
&= -\tfrac{1}{4}\al^{pq}\ep_{ij}\,^{ab}\sX_{abpq} = -\tfrac{1}{2}\al^{pq}(\star \sX)_{ijpq} = \op{\star\sX}_{\ext^{2}\std}(\al)_{ij},
\end{split}
\end{align}
which, with \eqref{staropsx1}, shows that $\op{\sX}_{\ext^{2}\std} \circ \star = \op{\star\sX}_{\ext^{2}\std} = \star \circ \op{\sX}_{\ext^{2}\std}$. 
\end{proof}

For a $4$-dimensional oriented Euclidean vector space $(\ste, h, \ep)$, define $\mcurvweylpm(\std) = \{\sX \in \mcurvweyl(\std): \star \sX = \pm \sX\}$.
Because $\lb \star \sX, \sY \ra = \lb \sX, \star \sY \ra$, $\mcurvweylp(\std)$ and $\mcurvweylm(\std)$ are orthogonal complements.

\begin{lemma}\label{albeselfduallemma}
Let $(\ste, h, \ep)$ be  an oriented $4$-dimensional Euclidean vector space. Write $\al = \al^{+} + \al^{-}$ for the decomposition of $\al$ into its self-dual and anti-self-dual parts $\al^{\pm} \in \pmdf\std$. For $\al, \be \in \ext^{2}\std$, 
\begin{align}\label{tfsdasd}
&\tf(\star \al \cdot \star \be) = \tf(\al \cdot \be) = \star \tf((\star \al)\cdot \be),&&
\tf(\al^{\pm}\cdot \be^{\pm}) = \al^{\pm}\cdot \be^{\pm} - \tfrac{1}{4}\lb \al^{\pm}, \be^{\pm}\ra h \kwedge h.
\end{align}
In particular, $\tf(\al^{+}\cdot \be^{+}) \in \mcurvweylp(\std)$, $\tf(\al^{-}\cdot \be^{-}) \in \mcurvweylm(\std)$, and $\tf(\al^{+}\cdot \be^{-}) = 0$ so that
\begin{align}\label{mixedcdot}
\al^{+}\cdot \be^{-} = -3(\al^{+}\jmult \be^{-})\kwedge h \in \mcurvric(\std).
\end{align}
\end{lemma}

\begin{proof}
By \eqref{sxalcdotbe} and Lemma \ref{starmcurvweylemma}, for $\sX \in \mcurvweyl(\std)$ and $\al, \be \in \ext^{2}\std$,
\begin{align}
\begin{split}
\lb \tf(\star \al \cdot \star \be), \sX\ra& =\lb \star \al \cdot \star \be, \sX\ra =  -6\lb \op{\sX}(\star \al), \star \be \ra \\&= -6\lb \star \op{\sX}(\al), \star \be \ra = -6\lb \op{\sX}(\al), \be \ra = \lb \al \cdot \be, \sX\ra =\lb \tf(\al \cdot \be), \sX\ra ,\\
\lb \star \tf((\star \al)\cdot \be), \sX \ra & = \lb \tf((\star \al) \cdot \be), \star \sX\ra = \lb (\star\al) \cdot \be, \star \sX\ra = -6\lb \star \al, \op{\star\sX}(\be)\ra \\&= -6\lb \star \al ,\star \op{\sX}(\be)\ra = -6\lb  \al, \op{\sX}(\be)\ra =
\lb \al\cdot \be, \sX\ra = \lb \tf(\al\cdot \be), \sX\ra,
\end{split}
\end{align}
which show \eqref{tfsdasd}. By \eqref{tfsy} and \eqref{rictralcdotbe}, 
\begin{align}\label{tfalbe4d}
\begin{split}
\tf(\al \cdot \be) 
& = \al \cdot \be + 3(\al \jmult \be)\kwedge h + \tfrac{1}{2}\lb \al, \be \ra h \kwedge h.
\end{split}
\end{align}
By \eqref{tfalbe4d} and Lemma \ref{albecontractlemma}, $\al^{\pm} \jmult \be^{\pm} = -\tfrac{1}{4}\lb \al^{\pm}, \be^{\pm}\ra h$, and in \eqref{tfalbe4d} this yields \eqref{tfsdasd}.
By \eqref{tfsdasd}, $\star \tf(\al \cdot \be) = \tf((\star \al)\cdot \be) = \tf(\al^{+}\cdot \be^{+}) - \tf(\al^{-}\cdot \be^{-})$, and it follows that $\tf(\al^{\pm}\cdot \be^{\pm}) \in \mcurvweylpm(\std)$ and $\tf(\al^{+}\cdot \be^{-}) =0$. Substituting the last identity in \eqref{tfalbe4d} yields \eqref{mixedcdot}. 
\end{proof}

\begin{example}\label{ssselfdualexample}
Given an oriented $4$-dimensional Euclidean vector space  $(\ste, h, \ep)$, let $\om \in \ext^{2}\std$ satisfy $\om \circ \om = -h$, so that $\om_{i}\,^{j}$ is a compatible almost complex structure and $(\ste, h, \om)$ is a Kähler structure. Then $\om \in \pmdf\std$ as $\tfrac{1}{2}\om \dwedge \om = \pm \ep$, and, by Lemma \ref{albeselfduallemma}, the idempotent $\sS(\om) = \tfrac{1}{6}(h \kwedge h - \om \cdot \om) = -\tf(\om \cdot \om) \in \mcurvweyl(\std)$ defined in \eqref{ssrdefined} satisfies $\sS(\om) \in \mcurvweyl^{\pm}(\std)$. 
\end{example}

Lemma \ref{cmultreductionlemma} gives another expression for $\sX \opmult \sY$ that is used in the proof of Lemma \ref{lovelocklemma}.
\begin{lemma}\label{cmultreductionlemma}
Let $(\ste, h)$ be a metric vector space.
For $\sX, \sY \in \mcurv(\std)$,
\begin{align}\label{cmult0}
\begin{split}
(\sX\opmult \sY)_{ijkl} & =  -4B(\sX, \sY)_{[ij]kl} - 2B(\sX, \sY)_{k[ij]l} + 6B(\sX, \sY)_{[ijkl]}  \\
& = -\tfrac{1}{2}\left( \sX_{ij}\,^{pq}\sY_{pqkl} + \sY_{ij}\,^{pq}\sX_{pqkl}\right) + \tfrac{3}{2}\sX_{[ij}\,^{pq}\sY_{kl]pq} - 2B(\sX, \sY)_{k[ij]l}.
\end{split}
\end{align}
\end{lemma}
\begin{proof}
The first equality of \eqref{cmult0} follows from \eqref{opmultadefinedb}, \eqref{opmultsdefinedb}, and Theorem \ref{opalgebratheorem}. The second equality of \eqref{cmult0} follows from \eqref{prebelskew}.
\end{proof}
\begin{lemma}\label{lovelocklemma}
Let $(\ste, h)$ be a $4$-dimensional Euclidean vector space. For $\sX, \sY \in \mcurvweyl(\std)$,
\begin{align}
\label{lovelocktracedtwice}\sX_{(i}\,^{abc}\sY_{j)abc} &
= \tfrac{1}{4}\lb \sX, \sY\ra h_{ij},\\
\label{lovelockpaired}
\begin{split}
2B(\sX, \sY)_{[ij]kl}& = \tfrac{1}{4}(\sX_{ij}\,^{pq}\sY_{pqkl} + \sY_{ij}\,^{pq}\sX_{pqkl})
 = 2B(\sX, \sY)_{k[ij]l}  + \tfrac{1}{8}\lb \sX, \sY\ra(h\kwedge h)_{ijkl}.
\end{split}\\
\label{cmult4d}
\begin{split}
(\sX\opmult \sY)_{ijkl} &= -\tfrac{3}{4}\left( \sX_{ij}\,^{pq}\sY_{pqkl} + \sY_{ij}\,^{pq}\sX_{pqkl}\right) + \tfrac{1}{8}\lb \star \sX, \sY \ra \ep_{ijkl} +\tfrac{1}{8}\lb \sX, \sY\ra (h\kwedge h)_{ijkl}.
\end{split}
\end{align}
\end{lemma}
\begin{proof}
By \eqref{lovelockid}, $0 = \sX_{bcl}\,^{p}\sX_{[ij}\,^{[kb}\delta_{p]}\,^{c]}$. Lowering the index $k$ and simplifying this expression yields 
\begin{align}\label{lovelockpaired1}
\sX_{ij}\,^{pq}\sX_{pqkl}  =  - 2\sX_{iplq}\sX_{k}\,^{p}\,_{j}\,^{q} + 2\sX_{jplq}\sX_{k}\,^{p}\,_{i}\,^{q} + \sX_{labc}\sX_{j}\,^{abc}h_{ik} - \sX_{labc}\sX_{i}\,^{abc}h_{jk}.
\end{align} 
Tracing \eqref{lovelockpaired1} in $jl$ and relabeling the result yields $\sX_{iabc}\sX_{j}\,^{abc} = \tfrac{1}{4}|\sX|^{2}h_{ij}$. Polarizing this yields \eqref{lovelocktracedtwice}.
Substituting \eqref{lovelocktracedtwice} into \eqref{lovelockpaired1} and using \eqref{prebelskew} yields
\begin{align}\label{lovelockpaired2}
\begin{split}
4B(\sX, \sX)_{[ij]kl} &= \sX_{ij}\,^{pq}\sX_{pqkl}  =  - 2\sX_{iplq}\sX_{k}\,^{p}\,_{j}\,^{q} + 2\sX_{jplq}\sX_{k}\,^{p}\,_{i}\,^{q} + \tfrac{1}{2}|\sX|^{2} h_{k[i}h_{j]l}\\
& = 4B(\sX, \sX)_{k[ij]l} + \tfrac{1}{4}|\sX|^{2}(h \kwedge h)_{ijkl}.
\end{split}
\end{align} 
Polarizing \eqref{lovelockpaired2} yields \eqref{lovelockpaired}.
Because $\dim \ext^{4}\std = 1$, $\sX_{[ij}\,^{pq}\sY_{kl]pq} = c \ep_{ijkl}$ for some $c \in \rea$. Contracting this equality with $\ep^{ijkl}$ yields $24c = 2\lb \star \sX, \sY \ra$, so that $12\sX_{[ij}\,^{pq}\sY_{kl]pq} = \lb \star \sX, \sY \ra \ep_{ijkl}$. Substituting this and \eqref{lovelockpaired} into \eqref{cmult0} of Lemma \ref{cmultreductionlemma} yields \eqref{cmult4d}.
\end{proof}

\begin{lemma}\label{selfdualcurvmultlemma}
Let $(\ste, h, \ep)$ be a $4$-dimensional oriented Euclidean vector space. There hold
\begin{align}\label{cmultpreisopm}
\begin{split}
\tfrac{1}{3}\op{\sX \opmult \sY}_{\ext^{2}\std}& = \op{\sX}_{\ext^{2}\std}\jmult \op{\sY}_{\ext^{2}\std}   - \tfrac{1}{3}\tr( \op{\sX} _{\ext^{2}\std}\circ \op{\sY}_{\ext^{2}\std}) \oproj_{\pmdf\std} \quad \text{if} \quad \sX, \sY \in \mcurvweyl^{\pm}(\std),\\
\op{\sX \opmult \sY}_{\ext^{2}\std}& = 0 \quad \text{if} \quad \sX \in \mcurvweyl^{+}(\std), \sY \in \mcurvweyl^{-}(\std).
\end{split}
\end{align}
where $\oproj_{\pmdf\std} \in \eno(\ext^{2}\std)$ are the orthogonal projections onto $\pmdf\std$. 
Consequently,
\begin{align}\label{mcurvmultrelations}
& \mcurvweylpm(\std) \opmult \mcurvweylpm(\std) = \mcurvweylpm(\std), &
&\mcurvweylp(\std) \opmult \mcurvweylm(\std) = \{0\}.
\end{align}
\end{lemma}

\begin{proof}
For $\sX, \sY \in \mcurvweyl(\std)$, rewriting \eqref{cmult4d} in terms of $\,\op{\,\,}_{\ext^{2}\std}$ and using Corollary \ref{traceextcorollary} yields 
\begin{align}\label{cmultpreiso}
\begin{split}
\tfrac{1}{3}\op{\sX \opmult \sY}_{\ext^{2}\std}& 
 = \op{\sX}_{\ext^{2}\std}\jmult \op{\sY}_{\ext^{2}\std}  - \tfrac{1}{6}\tr(\op{\star \sX}_{\ext^{2}\std} \circ \op{ \sY}_{\ext^{2}\std}) \star - \tfrac{1}{6}\tr( \op{\sX}_{\ext^{2}\std} \circ \op{\sY}_{\ext^{2}\std}	) \Id_{\ext^{2}\std}.
\end{split}
\end{align}
Because, by Lemma \ref{starmcurvweylemma}, $\star$ commutes with $\op{\sX}$ and $\op{\sY}$, \eqref{cmultpreiso} implies the equality $\star \sX \opmult \star \sY = \sX \opmult \sY$ which shows $\mcurvweylp(\std) \opmult \mcurvweylm(\std) = \{0\}$. Since $\op{\star \sX}_{\ext^{2}\std} = \op{\sX}_{\ext^{2}\std}\circ \star$, if $\sX = \sX^{+} + \sX^{-}$ with $\sX^{\pm}  \in \mcurvweyl^{\pm}(\std)$ then $\op{\sX}_{\ext^{2}\std}(\al) = \op{\sX^{+}}_{\ext^{2}\std}(\al^{+}) + \op{\sX^{-}}_{\ext^{2}\std}(\al^{-})$ where $\al = \al^{+} + \al^{-}$ is the decomposition of $\al \in \ext^{2}\std$ into its self-dual and anti-self-dual parts. In particular $\op{\sX^{\pm}}_{\ext^{2}\std}$ annihilates $\mpdf\std$ and $\op{\sX^{+}}_{\ext^{2}\std}$ and $\op{\sY^{-}}_{\ext^{2}\std}$ anticommute. In \eqref{cmultpreiso} these observations yield \eqref{cmultpreisopm}, which implies the containments in \eqref{mcurvmultrelations}. Because, by Example \ref{ssselfdualexample}, each of $\mcurvweylp(\std)$ contains a nontrivial idempotent, the $SO(4)$-submodules $\mcurvweylpm(\std) \opmult \mcurvweylpm(\std)$ are nontrivial $SO(4)$-submodules of the irreducible $SO(4)$-modules $\mcurvweylp(\std)$, so equality holds in \eqref{mcurvmultrelations}.
\end{proof}

\begin{remark}
Reversing the orientation of $\ste$ interchanges the subspaces $\mcurvweylpm(\std)$, but the orthogonal decomposition remains; all that changes is the labeling as $+$ or $-$. Hence the relations \eqref{mcurvmultrelations} make sense independently of any choice of orientation, in the sense that $\mcurvweyl(\std)$ decomposes as an orthogonal direct sum of two $5$-dimensional $\opmult$-subalgebras whose product is $\{0\}$.
\end{remark}

\begin{lemma}\label{deunitsimplelemma}
The deunitalization $(\sherm_{0}(3, \rea), \jrd)$ of the $6$-dimensional rank $3$ simple real Euclidean Jordan algebra $\sherm(3, \rea)$ is simple and contains no nontrivial square-zero elements.
\end{lemma}

\begin{proof}
Consider the representation of $\sherm_{0}(3, \rea)$ as trace-free $3 \times 3$ symmetric matrices. 
Let $\dalg \subset \sherm_{0}(3, \rea)$ be the $2$-dimensional subalgebra comprising the diagonal matrices. First it is shown that the subalgebra $(\dalg, \jrd)$ is simple. Let $\ga_{1} = E_{11} - E_{33}, \ga_{2} = E_{22} - E_{33}, \ga_{3} = E_{33} - E_{11} \in \dalg$ where $E_{ij}$ is the matrix with $1$ in the $ij$ component and $0$ in all other components. Then $\{\ga_{i}:1 \leq i \leq 3\}$ are idempotents satisfying $\ga_{i}\circ \ga_{j} = -\ga_{i} - \ga_{j}$. Let $\ideal$ be an ideal of $\dalg$ and let $a = a_{1}\ga_{1} + a_{2}\ga_{2} \in \ideal$. Then $a\jrd \ga_{1} + a = (2a_{1} - a_{2})\ga_{1}$ and $a\jrd \ga_{2} + a = (2a_{2} - a_{1})\ga_{2}$. If $2a_{1} = a_{2}$ and $2a_{2} = a_{1}$, then $4a_{1} = a_{1}$, so $a_{1} = 0$ and $a_{2} = 0$, so $a = 0$. Otherwise, if $2a_{1} \neq a_{2}$, then $\ga_{1} \in \ideal$, in which case $\ga_{2} = \ga_{1} + \ga_{1}\jrd \ga_{2} \in \ideal$, so $\ideal = \dalg$, while, if $a_{1} \neq 2a_{2}$, then $\ga_{2} \in \ideal$, so $\ga_{1} = \ga_{2} + \ga_{1}\jrd \ga_{2} \in \ideal$, so $\ideal = \dalg$. This shows $\dalg$ is simple.

Now let $\ideal$ be a nonzero ideal in $\sherm_{0}(3, \rea, \jrd)$. By the principal axis theorem, every element of $\sherm_{0}(3, \rea, \jrd)$ is equivalent via an automorphism of $(\sherm_{0}(3, \rea), \jrd)$ to an element of $\dalg$, so it can be assumed that $\ideal$ contains a nonzero element. Since $E_{ii} + E_{jj} - 2E_{kk} \in \dalg \subset \ideal$, $E_{ij} + E_{ji} = (E_{ij} + E_{ji})\jrd (E_{ii} + E_{jj} - 2E_{kk}) \in \ideal$ for all $i \neq j \in \{1, 2, 3\}$. Since together $\dalg$ and the elements $E_{ij} + E_{ji}$ with $i \neq j$ span $\sherm_{0}(3, \rea)$, this shows $\ideal = \sherm_{0}(3, \rea)$.

By the principal axis theorem, any square-zero element of $(\sherm_{0}(3, \rea), \jrd)$ is equivalent via an automorphism to an element of $\dalg$. If $a = a_{1}\ga_{1} + a_{2}\ga_{2} \in \dalg$ is square-zero then $a_{1}(a_{1} - 2a_{2}) = 0 = a_{2}(a_{2}- 2a_{1})$ and the unique solution is $a_{1}= 0 = a_{2}$. 
\end{proof}

\begin{remark}
Essentially the same argument shows that the deunitalization of a simple real Euclidean Jordan algebra of rank at least $3$ is simple. However, that there are no nontrivial square-zero elements is true if and only if the rank is odd. See \cite{Fox-simplicial} for details.
\end{remark}

\begin{proof}[Proof of Theorem \ref{mcurvweyltheorem}]
Let $(\ste, h)$ be a $4$-dimensional Euclidean vector space.
By Lemma \ref{starmcurvweylemma}, for any $\sX \in \mcurvweyl(\std)$, $\star \op{\sX}_{\ext^{2}\std} = \op{(\star \sX)}_{\ext^{2}\std} = \op{\sX}_{\ext^{2}\std}\star$, so, if $\sX \in \mcurvweyl^{\pm}(\std)$, then $\sX$ preserves $\pmdf\std$ and annihilates $\mpdf\std$, so induces a symmetric endomorphism of $\pmdf\std$. By Example \ref{ssselfdualexample}, if $\om \in \ext^{2}\std$ satisfies $\om \circ \om = -h$, then $\om \in \pmdf\std$ as $\tfrac{1}{2}\om \dwedge \om = \pm \ep$, and $\sS(\om) \in \mcurvweyl^{\pm}(\std)$. By \eqref{albega}, $\op{\om \cdot \om}_{\ext^{2}\std}(\om) = -5\om$, so $6\op{\sS(\om)}_{\ext^{2}\std}(\om) = \op{h\kwedge h}_{\ext^{2}\std}(\om)  - \op{\om \cdot \om}_{\ext^{2}\std}(\om)  = 4\om$. Since this shows that $\op{\sS(\om)}_{\ext^{2}\std}$ acts nontrivially on $\pmdf\std$, it follows from the Schur lemma that the $SO(4)$-equivariant map $\Psi:\mcurvweyl^{\pm}(\std) \to\sherm_{0}(\pmdf \std, h)$ sending $\sX \in \mcurvweyl^{\pm}(\std)$ to $3\op{\sX}_{\pmdf\std}$ is a linear isomorphism. 
It follows from \eqref{cmultpreisopm} of Lemma \ref{selfdualcurvmultlemma} that $\Psi: (\mcurvweyl^{\pm}(\std), \opmult) \to (\sherm_{0}(\pmdf \std, h), \jrd)$ is an algebra isomorphism. By the definition of $G$ and Corollary \ref{traceextcorollary}, for $\sX, \sY \in \mcurvweyl^{\pm}(\std)$, $\Psi^{\ast}(G)(\sX, \sY) = 9 G(\op{\sX}, \op{\sY}) = 3\tr(\op{\sX} \circ \op{\sY}) = \tfrac{3}{4}\lb \sX, \sY\ra$. That the trace-form $\tau_{\opmult}$ is the stated multiple of $h$ follows from the corresponding statement in the algebra $(\sherm_{0}(\pmdf \std, h), \jrd, \tfrac{4}{3}G)$ and the fact that $\Psi$ is an isometric isomorphism. That $(\mcurvweyl^{\pm}(\std), \opmult, \lb \dum, \dum \ra)$ is simple and contains no square-zero elements follows from the preceding in conjunction with Lemma \ref{deunitsimplelemma}. 
\end{proof}

\section{Idempotents in the subalgebra of Weyl curvature tensors}\label{idempotentsection}
In this section there are constructed some idempotents in $(\mcurvweyl(\std), \opmult)$ and some of their products are calculated. This provides more detailed information about the internal structure of $(\mcurvweyl(\std), \opmult)$ and yields an alternative proof of Theorem \ref{mcurvweyltheorem} along lines viable when $\dim \ste > 4$.

Let $(\ste, h)$ be a Euclidean vector space of dimension $n > 2$. A subspace $\stw \subset \ste$ determines $g \in \idem(S^{2}\std, \jmult)$ such that  $\stw$ equals the image of the endomorphism $g_{i}\,^{j}$, in which case $\tr g = \dim \stw$. The space of orthogonal almost complex structures on $\stw$ is identified with 
\begin{align}
\oc(\stw, h) = \{\al \in \ext^{2}\std: \al \circ \al = -g, \al \circ g = \al = g \circ \al \}
\end{align}
By construction, $\al \in \oc(\stw, h)$ satisfies $\dim \stw = \tr g = |g|^{2} = |\al|^{2}$. It will be said that $\al \in \oc(\stw, h)$ determines an orthogonal almost complex structure on $\stw$.

For even $r$ satisfying $2 \leq r \leq n$, there is $r$-dimensional $\stw \subset \ste$ such that $\oc(\stw, h)$ is nonempty. Let $\{\ep(1), \dots, \ep(n)\}$ be an $h$-orthonormal basis of $\std$ such that $\{\ep(1), \dots, \ep(r)\}$ spans $\stwd$ and $\{\ep(r+1), \dots, \ep(n)\}$ spans the $h$-orthogonal complement $\stw^{\ast\,\perp}$. The endomorphism associated with $g = \sum_{i = 1}^{r}\ep(i)\tensor \ep(i) \in S^{2}\std$ is the orthogonal projection on $\stw$ and $\al = \sum_{i = 1}^{r/2}\ep(2i-1)\dwedge \ep(2i) \in \oc(\stw, h)$.

\begin{lemma}\label{weylidempotentslemma}
Let $(\ste, h)$ be a Euclidean vector space of dimension $n > 2$. Let $r$ be even and satisfy $2 \leq r \leq n$, let $\stw \subset \ste$ have dimension $r$, let $\al \in \oc(\stw, h)$, and let $g = -\al \circ \al$. The elements of $\mcurv(\std)$ defined by
\begin{align}\label{ssrdefined}
\begin{aligned}
&\sH(g)   = \tfrac{1}{1-r}g \kwedge g, \qquad
\sK^{r}(\al)  =  -\tfrac{1}{r + 2}\left(\al \cdot \al + g\kwedge g \right),\\
&\sS^{r}(\al)   = \tfrac{1}{r+2}\left(\tfrac{3}{r-1}g\kwedge g - \al \cdot \al \right) = \sK^{r}(\al) - \sH(g),
\end{aligned}
\end{align}
are idempotents in $(\mcurv(\std), \opmult)$ that are nontrivial and linearly independent when $r \geq 4$, while, when $r = 2$, $\sS^{2}(\al)$ is trivial and $\sK^{2}(\al) = \sH(g)$.
They satisfy the relations
\begin{align}\label{ssrrelations}
&\sK^{r}(\al) \opmult \sH(g) = \sH(g), && \sK^{r}(\al) \opmult \sS^{r}(\al) = \sS^{r}(\al), && \sS^{r}(\al)\opmult \sH(g) = 0,
\end{align} 
\begin{align}
\label{ssrinorm}
\begin{aligned}
&\lb \sK^{r}(\al), \sS^{r}(\al)\ra = |\sS^{r}(\al)|^{2}_{h}  = \tfrac{6r(r-2)}{(r+2)(r-1)} ,&&
|\sK^{r}(\al)|^{2}_{h} =  \tfrac{8r}{r+2},\\
&\lb \sK^{r}(\al), \sH(g)\ra =|\sH(g)|^{2}_{h} =  \tfrac{2r}{r-1},&&
 \lb\sS^{r}(\al), \sH(g) \ra = 0.
\end{aligned}
\end{align}
Moreover:
\begin{enumerate}
\item $\sS^{r}(\al) \in \mcurvweyl(\std)$, while $\sK^{r}(\al)\notin\mcurvweyl(\std)$ because $\rictr(\sK^{r}(\al)) = g$. 
\item If $r \in \{n-1, n\}$, $\sS^{r}(\al) = -\tfrac{1}{r+2}\tf(\al \cdot \al)$.
\end{enumerate}
\end{lemma}

\begin{proof}
By \eqref{alalsquared}, \eqref{cdw}, and Lemma \ref{projectionkwedgelemma},
\begin{align}\label{ssr1}
\begin{aligned}
(\al \cdot \al)& \opmult (\al \cdot \al)  = -(r+2)\al \cdot \al +3g \kwedge g,\qquad
(\al \cdot \al) \opmult (g \kwedge g)  = -3g \kwedge g,\\
(g \kwedge g)&\opmult(g \kwedge g) = (1-r)g \kwedge g.
\end{aligned}
\end{align}
Combining the identities \eqref{ssr1} yields
\begin{align}\label{ssrcalc}
\begin{split}
(A\al \cdot \al &+ Bg \kwedge g) \opmult (A\al \cdot \al + Bg \kwedge g)- (A\al \cdot \al + Bg \kwedge g)\\
& = -A((r+2)A + 1)\al \cdot \al + ((1-r)B^{2} - (6A+1)B + 3A^{2})g \kwedge g.
\end{split}
\end{align}
That \eqref{ssrcalc} vanish yields the equations $A^{2}(r+2) = -A$ and $(r-1)B^{2} + (6A + 1)B - 3A^{2} = 0$. These equations have the three nontrivial solutions $(0, \tfrac{1}{1-r})$, $(-\tfrac{1}{r+2}, -\tfrac{1}{r+2})$ and $(-\tfrac{1}{r+2}, \tfrac{3}{(r-1)(r+2)})$ for $(A, B)$ that yield, respectively, the idempotents $\sH(g)$, $\sK^{r}(\om)$, and $\sS^{r}(\om)$. The relations \eqref{ssrrelations} follow from \eqref{ssr1} by computations similar to those showing \eqref{ssrcalc}.

By \eqref{rictralcdotbe}, $\rictr(\al \cdot \al) = 3\al \circ \al = -3g$, and by \eqref{rictralkwedgebe}, $\rictr(g \kwedge g) = g \circ g - \tr(g)g = -(r-1)g$, from which there follow $\rictr(\sS^{r}(\al)) = 0$ and $\rictr(\sK^{r}(\al)) = g$, so that $\sS^{r}(\al) \in \mcurvweyl(\std)$ but $\sK^{r}(\al) \notin \mcurvweyl(\std)$.

From \eqref{alcdotbegacdotsinorm}, \eqref{alcdotbegawedgesinorm}, and \eqref{alwedgebegawedgesinorm} there follow $|\al \cdot \al|^{2}_{h} =6r(r+1)$, $\lb \al \cdot \al, g \kwedge g \ra = 6r$, and $|g \kwedge g|^{2} = 2r(r-1)$, which yield \eqref{ssrinorm}. 
By \eqref{ssrinorm}, $\sH(g)$, $\sK^{r}(\al)$, and $\sS^{r}(\al)$ are nontrivial and linearly independent when $r \geq 4$, while, when $r = 2$, $\sS^{2}(\al)$ is trivial and $\sK^{2}(\al) = \sH(g)$.

A straightforward computation using \eqref{tfalcdotbe} shows
\begin{align}\label{ssr6}
\begin{split}
-\tf(\al \cdot \al)  &= -\al \cdot \al + \tfrac{6}{n-2}g \kwedge h - \tfrac{3r}{(n-1)(n-2)}h \kwedge h
 =  (r+2)\sS^{r}(\al) +  \tfrac{3(n-1-r)}{n-2}\sH(g), 
\end{split}
\end{align}
where $\sH(g)$ is as in \eqref{sbdefined}. If $r = n$, then $h = g$ and $\sH(g) = 0$, so if $r \in \{n-1, n\}$, by \eqref{ssr6}, $(r+2)\sS^{r}(\al) = -\tf(\al \cdot \al)$.
\end{proof}

\begin{remark}
From Lemma \ref{mcsubspacelemma} it follows that, if $\stwd$ is a subspace of $\std$, the inclusion of an idempotent of $\mcurvweyl(\stwd)$ is an idempotent in $\mcurvweyl(\std)$. 
In particular, the $r < m$ cases of Lemma \ref{weylidempotentslemma} follow from the $r = m$ case of Lemma \ref{weylidempotentslemma} in conjunction with Lemma \ref{mcsubspacelemma}. 
\end{remark}

Lemma \ref{idempotentparameterizationlemma} shows that when $\dim \std = 2n \geq 4$, certain of the idempotents produced by Lemma \ref{weylidempotentslemma} constitute an orbit of $O(2n)$ acting in $\mcurvweyl(\std)$ that can be identified with the space of orthogonal complex structures on $\ste$ inducing a given orientation on $\ste$. The proof uses Lemma \ref{ssrspectrumlemma}.

\begin{lemma}\label{ssrspectrumlemma}
Let $(\ste, h)$ be a Euclidean vector space of dimension $n$. Let $r$ be even and satisfy $4 \leq r \leq n$. Let $\stw \subset \ste$ have dimension $r$, let $\al \in \oc(\stw, h)$, and let $g = -\al \circ \al$. 
\begin{enumerate}
\item\label{ssreigenvalues} The eigenvalues of $\op{\sS^{r}(\al)}_{\ext^{2}\std}$ are $\tfrac{r-2}{r-1}$, with one-dimensional eigenspace spanned by $\al$; $\tfrac{1}{1-r}$, with $\tfrac{r(r-2)}{4}$-dimensional eigenspace contained in $\ext^{2}\stwd$; $\tfrac{r-4}{(r+2)(r-1)}$, with $\tfrac{r^{2} - 4}{4}$-dimensional eigenspace contained in $\ext^{2}\stwd$; and $0$, with eigenspace equal to $\ext^{2}\stw^{\perp\,\ast} \oplus \stw^{\perp\,\ast}\dwedge \stwd$.
\item\label{skreigenvalues} The nonzero eigenvalues of $\op{\sK^{r}(\al)}_{\ext^{2}\std}$ are $1$, with one-dimensional eigenspace spanned by $\al$; and $-\tfrac{2}{r+2}$, with $\tfrac{r^{2} - 4}{4}$-dimensional eigenspace contained in $\ext^{2}\stwd$.
\end{enumerate}
\end{lemma}

\begin{proof}
Write $\A = \qp_{\jmult}(\al) \in \eno(\ext^{2}\std)$.  By \eqref{qpid1}, $\A^{2} = \qp_{\jmult}(\al)\jmult \qp_{\jmult}(\al) = \qp_{\jmult}(g)$ and $\A^{3} = \qp_{\jmult}(\al)\jmult \qp_{\jmult}(g) = \qp_{\jmult}(\al) = \A$. By \eqref{albega} and \eqref{gasiopal},
\begin{align}\label{ssrop}
\begin{aligned}
\op{\sS^{r}(\al)}_{\ext^{2}\std} &= \tfrac{1}{r+2}\left(\al \tensor \al^{\sharp} - \A - \tfrac{3}{r-1}\A^{2}\right).
\end{aligned}
\end{align}
Since $\A(\al) = -\al$, by \eqref{ssrop}, $\op{\sS^{r}(\al)}_{\ext^{2}\std}(\al) = \tfrac{r-2}{r-1}\al$. Because $\op{\sS^{r}(\al)}_{\ext^{2}\std}$ is self-adjoint, it preserves the orthogonal complement $\lb \al \ra^{\perp}\subset \ext^{2}\std$. Because $\lb \A(\ga), \al\ra = -\lb \al, \ga\ra$, $\A$ preserves $\lb \al \ra^{\perp}$ as well. A straightforward computation using $\A^{3} =\A$ shows that the minimal polynomial of the restriction $\op{\sS^{r}(\al)}_{\lb \al \ra^{\perp}}$ is $x(x + \tfrac{1}{r-1})(x - \tfrac{r-4}{(r+2)(r-1)})$. It follows that $\tfrac{r-2}{r-1}$ has multiplicity one as an eigenvalue of $\op{\sS^{r}(\al)}_{\ext^{2}\std}$.
It is convenient to identify $\stwd$ and $\stw^{\perp\,\ast}$ with orthogonal subspaces of $\std$.
If $\nu \in \stw^{\perp\,\ast}$, then $\nu^{p}g _{pi} =0$, so $\nu^{p}\al_{pi} = \nu^{p}g _{p}\,^{q}\al_{q}\,^{i} = 0$, and it follows that, for all $\mu \in \std$, $\A(\mu \dwedge \nu) = 0$, so $\ext^{2}\stw^{\perp\,\ast} \oplus (\stw^{\perp\,\ast}\dwedge \stwd) \subset \ker \op{\sS^{r}(\al)}_{\ext^{2}\std}$. 
Because it commutes with $g _{i}\,^{j}$, the endomorphism $\al_{i}\,^{j}$ preserves $\stw$ and its restriction to $\stw$ is an almost complex structure compatible with the restriction of $h$ to $\stw$. 
Consequently, both $\A$ and $\op{\sS^{r}(\al)}_{\ext^{2}\std}$ preserve $\ext^{2}\stwd$ and their eigenvalues on $\ext^{2}\stwd$ are nonzero. In addition to forcing the equality $\ext^{2}\stw^{\perp\,\ast} \oplus \stw^{\perp\,\ast}\dwedge \stwd = \ker \op{\sS^{r}(\al)}_{\ext^{2}\std}$, this observation implies that the nonzero eigenspaces of $\op{\sS^{r}(\al)}_{\ext^{2}\std}$ on $\ext^{2}\stwd \cap \lb \al \ra^{\perp}$ are the $\pm 1$-eigenspaces of $\A$ on $\ext^{2}\stwd \cap \lb \al \ra^{\perp}$. The dimensions of these eigenspaces can be computed using the observation that these subspaces comprise the real parts of forms of type $(2,0)$ and $(1,1)$ with respect to the almost complex structure given by $\al_{i}\,^{j}$ (see \cite{Fox-curvtensorkahler} for details). 
This proves \eqref{ssreigenvalues}. Claim \eqref{skreigenvalues} follows from the preceding and $\op{\sK^{r}(\al)}_{\ext^{2}\std} = \op{\sS^{r}(\al)}_{\ext^{2}\std} + \op{\sH(g)}_{\ext^{2}\std}$. Further details are omitted.
\end{proof}

For a Euclidean vector space $(\ste, h)$ of even dimension $n = 2m$, the space $\oc(\ste, h) = \{\om_{ij} \in \ext^{2}\std: \om_{i}\,^{p}\om_{p}\,^{j} = -\delta_{i}\,^{j}\}$ is identified with the homogeneous space $O(2m)/U(m)$. The action of $O(2m)$ on $\ext^{2}\std$ preserves $\oc(\ste, h)$. Given $\om \in \oc(\ste, h)$ with associated almost complex structure $J_{i}\,^{j} = \om_{i}\,^{j}$ there exists an orthonormal basis of $\ste$ of the form $\{e_{1}, \dots, e_{n}, J(e_{1}), \dots, J(e_{n})\}$ and this suffices to show that $O(2m)$ acts transitively on $\oc(\ste, h)$. The stabilizer of $\om \in \oc(\ste, h)$ is $U(m) = O(2m) \cap Sp(n, \rea)$, where $Sp(n, \rea)$ is the symplectic group fixing $\om$ and $U(m)$ the unitary group preserving $(h, J)$. Thus $\oc(\ste, h)$ is identified with $O(2m)/U(m)$ for any $\om \in \oc(\ste, h)$. 

Since $O(2m)$ has two connected components, comprising orthogonal transformations preserving opposite orientations of $\ste$, and $U(m)$ is connected, the space $\oc(\ste, h)$ has two connected components $\oc^{\pm}(\ste, h)$, each identified with $SO(2m)/U(m)$, and the complex structures of $\oc^{+}(\ste, h)$ induce the orientation opposite that induced by the complex structures of $\oc^{-}(\ste, h)$. If $\om\in \oc(\ste, h)$ is fixed and $F \in O(2m)$ is the orthogonal reflection through the $+1$ eigenspace of the endomorphism $\om_{i}\,^{j}$, then the action of $F$ interchanges $\oc^{\pm}(\ste, h)$. Let $\hat{U}(m) \subset O(2m)$ be the subgroup generated by $U(m)$ and $F$. Then $O(2m)/\hat{U}(m)$ is identified with the quotient space $\oc(\ste, h)/\sim$ where $\om \sim - \om = F\cdot \om$. Since $U(m)$ and $FU(m)$ lie in different connected components of $O(2m)$, this quotient space is identified with $SO(2m)/U(m)$. 

\begin{lemma}\label{idempotentparameterizationlemma}
Let $(\ste, h)$ be a Euclidean vector space of even dimension $n = 2m$. The map $\sS:\oc(\ste, h) \to \idem(\mcurvweyl(\std), \opmult) \subset \mcurvweyl(\std)$ defined by 
\begin{align}
\sS(\om) = -\tfrac{1}{n+2}\tf(\om \cdot \om) = \tfrac{1}{n+2}\left(\tfrac{3}{n-1}h \kwedge h - \om \cdot \om\right),
\end{align}
is an $O(2m)$-equivariant double cover of its image $\sS(\oc(\ste, h))) = \{\sS(\om): \om \in\oc(\ste, h)\}$, injective on either connected component $\oc^{\pm}(\ste, h) \simeq SO(2m)/U(m)$, with image equal to $O(2m)/\hat{U}(m) \simeq SO(2m)/U(m)$, realized as the $O(2m)$ orbit of $\sS(\om)$ for any $\om \in \oc(\ste, h)$, and spanning $\mcurvweyl(\std)$.
\end{lemma}
\begin{proof}
Because the map $\tf$ and the product $\cdot$ are $O(2m)$-equivariant, so is the map $\sS$. By Lemma \ref{weylidempotentslemma}, for $g \in O(2m)$, $\sS(g\cdot \om)= g \cdot \sS(\om)$ is an idempotent in $(\mcurvweyl(\std), \opmult)$, and so $\sS$ is a map from the homogeneous space $O(2m)/U(m) \simeq \oc(\ste, h)$ to the $O(2m)$ orbit of $\sS(\om)$ in $\mcurvweyl(\std)$ whose image comprises idempotents. By definition, $\sS(-\om) = \sS(\om)$. 
If $\sS(g \cdot \om) = \sS(\om)$ for $g \in O(2n)$, then, by Lemmas \ref{curvoplemma} and \ref{ssrspectrumlemma} and the $O(2n)$-equivariance of $\sS$,
\begin{align}\label{ssinjective}
\op{\sS(\om)}_{\ext^{2}\std}(g\cdot \om) &= \op{\sS(g\cdot \om)}_{\ext^{2}\std}(g\cdot \om)=\op{g\cdot\sS(\om)}_{\ext^{2}\std}(g\cdot \om) =  g\cdot \op{\sS(\om)}_{\ext^{2}\std}(\om) = \tfrac{n-2}{n-1} g\cdot \om.
\end{align}
By Lemma \ref{ssrspectrumlemma}, $\om$ spans the $\tfrac{n-2}{n-1}$-eigenspace of $\op{\sS(\om)}_{\ext^{2}\std}$, so \eqref{ssinjective} shows $g\cdot \om$ is a multiple of $\om$. Since $(g\cdot \om)\circ (g\cdot \om) = -h$, this forces $g\cdot \om = \pm \om$, and hence $g \in \hat{U}(m)$. It follows that $\sS$ is two-to-one, with image equal to the $O(2m)$ orbit of $\sS(\om)$ for any $\om \in \oc(\ste, h)$, and that this orbit is identified with $O(2m)/\hat{U}(m)$ in such a way that $\sS$ maps either connected component $\oc^{\pm}(\om, h)$ onto it bijectively. 
By the $O(2m)$-irreducibility of $\mcurvweyl(\std)$, because the $O(2m)$-invariant subspace $\spn\{\sS(\om): \om \in\oc(\ste, h)\} \subset\mcurvweyl(\std)$ is nonempty, it equals $\mcurvweyl(\std)$.
\end{proof}

\begin{lemma}\label{jendobwlemma}
For $m \times m$ anti-Hermitian complex matrices $A$ and $B$ such that $A^{2} = -I = B^{2}$, 
\begin{align}\label{trcomrelation}
 -2m \leq 2m - 4 \tr \overline{A\jmult B}^{t}A\jmult B= -2\tr(ABAB) = -2m + \tr \overline{[A, B]}^{t}[A, B] \leq  2m, 
\end{align}
with equality in the upper bound if and only if $A$ and $B$ anticommute and equality in the lower bound if and only if $A$ and $B$ commute. In the latter case there is $q \in \integer$ such that $0 \leq q \leq m$ such that $\tr AB = 2q- m$ with $\tr AB = \pm m$ if and only if $B = \pm A$; moreover, if $A$ and $B$ are real matrices then $q$ must be even.
\end{lemma}

\begin{proof}
The equalities in \eqref{trcomrelation} are always valid.
If $A\jmult B = 0$, then $\tr \overline{[A, B]}^{t}[A, B] = -\tr [A, B]^{2} = -4\tr(ABAB) = 4\tr(A^{2}B^{2}) = 4m$, so equality holds in the upper bound in \eqref{trcomrelation}. It is immediate that equality holds in the lower bound in \eqref{trcomrelation} if and only if $A$ and $B$ commute. In this case $iA$ and $iB$ are commuting Hermitian matrices so are simultaneously unitarily diagonalizable. If $m-q$ is the dimension of their joint $1$ eigenspace, then $-\tr AB \tr (\j A)(\j B)= (m-q) - q = m - 2q$, and $\tr AB = \pm m$ if and only if $q \in \{0, m\}$, in which case $B = \pm A$. If $A$ and $B$ are moreover real then their eigenvalues $\pm \j$ have the same multiplicities, so $q$ must be even.
Now suppose $A$ and $B$ are anti-Hermitian and $A^{2} = -I = B^{2}$. There is an $m \times m$ unitary matrix $U$ such that $C = UA\bar{U}^{t}$ is diagonal with diagonal entries $c_{1}, \dots, c_{m} \in \{\pm \j\}$. Define $D = UB\bar{U}^{t}$. The components of $[C, D]$ with respect to a basis satisfy $[C, D]_{ij} = (c_{i} - c_{j})D_{ij}$ and $\overline{c_{i} - c_{j}} = -(c_{i} - c_{j})$, so
\begin{align}\label{trcom}
\begin{aligned}
0 & \leq \tr \overline{[A, B]}^{t}[A, B] = \tr \overline{[C, D]}^{t}[C, D] = -\sum_{ i \neq j} (c_{i} - c_{j})^{2}|D_{ij}|^{2} &\\
&\leq 2\sum_{i \neq j}(|c_{i}|^{2} + |c_{j}|^{2})|D_{ij}|^{2} \leq 4\sum_{i \neq j}|D_{ij}|^{2} \leq -4\tr D^{2} = -4\tr B^{2} = 4m.
\end{aligned}
\end{align}
By \eqref{trcom}, $\tr \overline{[A, B]}^{t}[A, B] = 4m$ if and only if $(c_{i} + c_{j})|D_{ij}|^{2}= 0$ for all $1\leq i \neq j \leq m$ and $D_{ii} = 0$ for all $1 \leq i \leq m$. This means $c_{i} + c_{j} = 0$ or $D_{ij} = 0$ for all $1\leq i \neq j \leq m$, so $\tr \overline{[A, B]}^{t}[A, B] = 4m$ if and only if $(c_{i} + c_{j})D_{ij} = 0$ for all $1 \leq i,j \leq m$. Equivalently, $\tr \overline{[A, B]}^{t}[A, B] = 4m$ if and only if $CD + DC =0$, or, what is the same, if and only if $A$ and $B$ anticommute. 
\end{proof}

\begin{lemma}\label{ssranglelemma}
Let $(\ste, h)$ be a Euclidean vector space of dimension $n$. Let $\stw \subset \ste$ have dimension $r \in 2\integer$ satisfying $4 \leq r \leq n$, let $\al, \be\in \oc(\stw, h)$, and let $g = -\al \circ \al = -\be\circ \be$. There holds
\begin{align}\label{ssangle}
\lb \sS^{r}(\al), \sS^{r}(\be)\ra = \tfrac{6}{(r+2)^{2}}\left(\lb \al, \be\ra^{2} + \tr (\al\circ \be \circ \al\circ \be) - \tfrac{3r}{r-1}\right) \geq - \tfrac{6r}{(r+2)(r-1)},
\end{align}
so that
\begin{align}\label{ssangle2}
\tfrac{\lb \sS^{r}(\al), \sS^{r}(\be)\ra}{|\sS^{r}(\al)||\sS^{r}(\be)|} 
= \tfrac{r-1}{r(r-2)(r+2)}\left(\lb \al, \be\ra^{2} + \tr(\al\circ \be \circ \al\circ \be)  - r(r+1)\right) + 1 \geq -\tfrac{1}{r-2}.
\end{align}
\begin{enumerate}
\item There holds equality in the lower bounds in \eqref{ssangle} and \eqref{ssangle2} if and only if $\al\circ \be + \be \circ \al= 0$.
\item\label{commutingcomplexangle} If $[\al, \be] = 0$, then there is $k \in \{4p - r: 0 \leq p \leq r/2\}$ such that $\lb \al, \be\ra = -k$ and $\lb \sS^{r}(\al), \sS^{r}(\be)\ra =  \tfrac{6r(r-4)}{(r-1)(r+2)^{2}} + \tfrac{6k^{2}}{(r+2)^{2}}$. Moreover, $k = \pm r$ if and only if $\be = \pm \al$. In particular, $\lb \sS^{4}(\al), \sS^{4}(\be)\ra =0$ if and only if $r = 4$ and $\be \neq \pm \al$.
\end{enumerate}
\end{lemma}
\begin{proof}
Let $\be \in \oc(\ste, h)$ and note that this means $\be\circ \be = -g$. By \eqref{alcdotbegacdotsinorm}, $|\al \cdot \al|^{2}_{h} =6r(r+1)$ and $\lb \al \cdot \al, \be \cdot \be\ra = 6\lb \al, \be\ra^{2} + 6\tr (\al \circ \be \circ \al \circ \be)$. By \eqref{alwedgebegawedgesinorm}, $|g  \kwedge g |^{2}_{h} = 2r(r-1)$ and
\begin{align}
\lb \be \cdot \be, g  \kwedge g  \ra = -6\tr \be\circ g \circ \be \circ g  = -6 \tr \be \circ g \circ \be \circ g  = -6\tr \be \circ \be = 6\tr g  = 6r,
\end{align}
and, similarly, $\lb\al \cdot \al, g  \kwedge g \ra = 6r$. There result the equality in \eqref{ssangle} and \eqref{ssangle2}. Equality holds in the lower bounds of \eqref{ssangle} and \eqref{ssangle2} when $\al \circ \be = -\be \circ \al$ because in this case $\lb \al, \be\ra = 0$, and $\tr \al \circ \be \circ \al \circ \be = -r$. 
By Lemma \ref{jendobwlemma} applied to the matrices of $\al_{i}\,^{j}$ and $\be_{i}\,^{j}$ in some basis, $2r - 2\tr(\al \circ \be \circ \al \circ \be) = -\tr([\al, \be]\circ [\al, \be])  \leq 4r$, with equality if and only if $\al \jmult \be = 0$, so that $\tr(\al \circ \be \circ \al \circ \be)  \geq -r$, with equality if and only if $\al \jmult \be = 0$. There follow the lower bounds in \eqref{ssangle} and \eqref{ssangle2}, with equality if and only if $\al \jmult \be = 0$. 
If $[\al, \be] = 0$, then $\al \circ \be$ is self-adjoint and $\al \circ \be\circ \al \circ \be = g$, so $-\lb \al, \be\ra = \tr(\al \circ \be)$ is an integer $k$ such that $-r \leq k \leq r$. By Lemma \ref{jendobwlemma}, $k = 4p - r$ for some $0 \leq p \leq r/2$. 
In \eqref{ssangle} this yields \eqref{commutingcomplexangle}.
\end{proof}

A \emph{hypercomplex structure} on a real vector space $\ste$ is a pair of anticommuting almost complex structures $I, J \in \eno(\ste)$ such that $K = I\circ J$ is an almost complex structure.
Because it is a module over the quaternions, a hypercomplex vector space has dimension divisible by $4$. 
A \emph{hyper-Kähler structure} on a Euclidean vector space $(\ste, h)$ equipped with a hypercomplex structure $\{I, J, K\}$ and a metric $h$ such that each of $I$, $J$, and $K$ is compatible with $h$. By definition this means that $\al_{ij} = I_{i}\,^{p}h_{pj}$, $\be_{ij} = J_{i}\,^{p}h_{pj}$, and $\ga_{ij} = K_{i}\,^{p}h_{pj}$ are symplectic forms. Given a subspace $\stw \subset \ste$, it will be said that an ordered pair $(\al, \be)\in \oc(\stw, h)^{2}$ determines a hyper-Kähler structure on $\stw$ if $\al$ and $\be$ anticommute, meaning $\al \jmult \be = 0$. In this case $\ga = \al \circ \be \in \oc(\stw, h)$, $\stw$ is the image of the endomorphism $g_{i}\,^{j}$ where $-g = \al \circ \al = \be \circ \be = \ga \circ \ga$, $\al \circ g= \al = g \circ g$, $\be \circ g = \be = g \circ \be$, and $\ga \circ g = \ga = g \circ \ga$, and $\tr g = |g|^{2}= |\al|^{2} = |\be|^{2} = \dim \stw$.

\begin{lemma}\label{hyperkahlerlemma}
Let $(\ste, h)$ be a Euclidean vector space. Let $\stw \subset \ste$ be a subspace of dimension $r$ divisible by $4$, suppose $(\al, \be) \in \oc(\stw, h)^{2}$ determines a hyper-Kähler structure on $\stw$, and let $\sS^{r}(\al), \sS^{r}(\be), \sS^{r}(\ga) \in \mcurvweyl(\std)$ be the idempotents defined as in \eqref{ssrdefined} (where $\ga = \al \circ \be$).
\begin{enumerate}
\item There hold the relations
\begin{align}\label{ssrproduct}
\begin{aligned}
&\sS^{r}(\al) \opmult \sS^{r}(\be) = -\tfrac{1}{r+2}\left(\sS^{r}(\al) +\sS^{r}(\be) -5 \sS^{r}(\ga)\right),&\\
&\sS^{r}(\al)\opmult(\al \cdot \be)  = \tfrac{1}{2}\al \cdot \be,& 
&\sS^{r}(\al)\opmult(\be \cdot \ga)  = -\tfrac{6}{r+2}\be \cdot \ga,&
\end{aligned}
\end{align}
\begin{align}\label{ssrproduct2}
\begin{aligned}
&(\al \cdot \be) \opmult (\al \cdot \be)=(r+2)\left(\tfrac{(r+4)}{4}(\sS^{r}(\al) + \sS^{r}(\be)) - \tfrac{5}{2}\sS^{r}(\ga)\right),&\\
&(\al \cdot \be) \opmult (\al \cdot \ga) = -\tfrac{r+14}{4}\be \cdot \ga.&
\end{aligned}
\end{align}
and those obtained from them by permuting $\al$, $\be$, and $\ga$.
\item \label{ssalbeangle} $\lb \sS^{r}(\al), \sS^{r}(\be)\ra = \lb \sS^{r}(\be), \sS^{r}(\ga)\ra = \lb \sS^{r}(\ga), \sS^{r}(\al)\ra =-\tfrac{6r}{(r+2)(r-1)}$,
so the cosine of the angle between any two of $\sS^{r}(\al)$, $\sS^{r}(\be)$, and $\sS^{r}(\ga)$ is $1/(2-r)$.
\item\label{albegaproducts} The elements $\al\cdot \be$, $\be \cdot \ga$, and $\ga\cdot \al$ are pairwise orthogonal of norm $\sqrt{3r(r+2)}$ and each is orthogonal to $\spn\{\sS^{r}(\al), \sS^{r}(\be), \sS^{r}(\ga)\}$. 
\item\label{hcsubalgebra} Let $\balg = \spn\{\sS^{4}(\al),\sS^{4}(\be), \sS^{4}(\ga), \al\cdot \be, \be \cdot \ga, \ga \cdot \al\}$.
\begin{enumerate}
\item\label{ssalbegasum} If $r = 4$, $\sS^{4}(\al) + \sS^{4}(\be) + \sS^{4}(\ga) = 0$, and $\balg \subset (\mcurvweyl(\std), \opmult)$ is a $5$-dimensional subalgebra.
\item If $r \geq 8$, $\balg$ is a $6$-dimensional subalgebra of $(\mcurvweyl(\std), \opmult)$.
\end{enumerate}
\item\label{opalbepre} For $\sX \in \spn\{\sS^{r}(\al), \sS^{r}(\be), \sS^{r}(\ga), \be \cdot \ga, \ga \cdot \al, \al \cdot \be\}$, $\op{\sX}_{\ext^{2}\std}$ preserves $\spn \{\al, \be, \ga \} \subset \ext^{2}\std$. \item\label{opalbe} For $\sX = -x_{1}\sS^{r}(\al) - x_{2}\sS^{r}(\be) - x_{3}\sS^{r}(\ga) - \tfrac{2}{r+2}\left(w_{1}\be \cdot \ga + w_{2}\ga \cdot \al + w_{3}\al \cdot \be \right)$,
the matrix of $\op{\sX}_{\spn\{\al, \be, \ga\}}$ with respect to the equal-norm orthogonal basis $\{\al, \be, \ga\}$ is
\begin{align}\label{hkmatrix}
\begin{pmatrix}
\tfrac{(2-r)x_{1} + x_{2} + x_{3}}{r-1} & w_{3} & w_{2} \\
w_{3} &\tfrac{x_{1}+ (2-r)x_{2} + x_{3}}{r-1}   & w_{1}\\
w_{2} & w_{1} &\tfrac{x_{1} + x_{2}  + (2-r)x_{3}}{r-1} 
\end{pmatrix}.
\end{align}
\end{enumerate}
\end{lemma}

\begin{proof}
The two-forms $\al, \be, \ga$ are pairwise orthogonal. Specializing \eqref{albecmultgasi} yields \eqref{ssrproduct2} and the identities obtained from them by permuting $\al$, $\be$, and $\ga$. By \eqref{symalbecmgasi}, \eqref{cdw}, and \eqref{albecmultgasi} there hold 
\begin{align}\label{alcmbeya}
\begin{aligned}
&(\al \cdot \al)\opmult (\be \cdot \be)  =  \al \cdot \al + \be \cdot \be - 5\ga \cdot \ga,&&\\
&(\al \cdot \al)\opmult (\be \cdot \ga)  = 6\be \cdot \ga, &&
(\al \cdot \al)\opmult (\al \cdot \be)  = -\tfrac{r+2}{2}\be \cdot \ga,\\
&(\al \cdot \al)\opmult (g\kwedge g)  = -3g\kwedge g, &&
(\al \cdot \be)\opmult (g\kwedge g)  = 0,
\end{aligned}
\end{align}
and the identities obtained from them by permuting $\al$, $\be$, and $\ga$. The identities \eqref{ssrproduct} follow from \eqref{alcmbeya}. By \eqref{ssrproduct} and \eqref{ssrinorm},
\begin{align}
\begin{split}
 \lb \sS^{r}(\al),\sS^{r}(\be)\ra & =  \lb \sS^{r}(\al)\opmult \sS^{r}(\al),\sS^{r}(\be)\ra =   \lb \sS^{r}(\al), \sS^{r}(\al)\opmult \sS^{r}(\be)\ra \\
&= -\tfrac{1}{r+2}|\sS^{r}(\al)|^{2} - \tfrac{1}{r+2} \lb \sS^{r}(\al),\sS^{r}(\be)\ra + \tfrac{5}{r+2}\lb \sS^{r}(\al),\sS^{r}(\ga)\ra \\
&= -\tfrac{6r(r-2)}{(r+2)^{2}(r-1)} + \tfrac{4}{r+2}\lb \sS^{r}(\al),\sS^{r}(\be)\ra,
\end{split}
\end{align}
so that $\lb \sS^{r}(\al),\sS^{r}(\be)\ra = -\tfrac{6r}{(r+2)(r-1)}$. Combined with \eqref{ssrinorm} this shows that the cosine of the angle between $\sS^{r}(\al)$ and $\sS^{r}(\be)$ is $1/(2-r)$. The preceding claims remain true when $\al$, $\be$, and $\ga$ are permuted cyclically. The identities \eqref{alcmbeya} show $\balg$ as in \eqref{hcsubalgebra} is a subalgebra. By \eqref{ssrinorm} and \eqref{ssalbeangle},
\begin{align}
\begin{split}
\left|c_{1}\sS^{r}(\al) + c_{2}\sS^{r}(\be) + c_{3}\sS^{r}(\ga)\right|^{2} = \tfrac{6r\left( (r-4)(c_{1}^{2} + c_{2}^{2} + c_{3}^{2}) + (c_{1} - c_{2})^{2} + (c_{2} - c_{3})^{2} + (c_{3} - c_{1})^{2}\right)}{(r+2)(r-1)}.
\end{split}
\end{align}
It follows that $c_{1}\sS^{r}(\al) + c_{2}\sS^{r}(\be) + c_{3}\sS^{r}(\ga) = 0$ for $c_{i}$ not all zero if and only if $r = 4$ and $c_{1} = c_{2} = c_{3}$. Together with \eqref{ssrinorm}, \eqref{ssalbeangle}, and \eqref{albegaproducts} this implies both claims of \eqref{hcsubalgebra} straightforwardly.

By \eqref{albega} and \eqref{gasiopal}, there hold $\op{\al \cdot \al}_{\ext^{2}\std}(\al) = -(r+1)\al$, $\op{g\kwedge g}_{\ext^{2}\std}(\al) = -\al$, $\op{(\al \cdot \be)}_{\ext^{2}\std}(\al) = -\tfrac{r+2}{2}\be$, $\op{(\al \cdot \be)}_{\ext^{2}\std}(\ga) = 0$, and the identities obtained from these by permuting $\al$, $\be$, and $\ga$. Combining these identities with the definition of $\sS(\al)$ yields $\op{\sS(\al)}_{\ext^{2}\std}(\al)  =  \tfrac{r-2}{r-1}\al$ and $\op{\sS(\be)}_{\ext^{2}\std}(\al) = -\tfrac{1}{r-1}\al$. Claims \eqref{opalbepre} and \eqref{opalbe} follow from these identities.
\end{proof}

Lemma \ref{4dsubalgebralemma} identifies the $5$-dimensional subalgebra of \eqref{ssalbegasum} of Lemma \ref{hyperkahlerlemma}. 

\begin{lemma}\label{4dsubalgebralemma}
Let $(\ste, h)$ be a Euclidean vector space of dimension at least $4$. Let $\stw \subset \ste$ be a $4$-dimensional subspace and suppose $(\al, \be) \in \oc(\stw, h)^{2}$ determines a hyper-Kähler structure on $\stw$. Define $\sS(\al) = \sS^{4}(\al), \sS(\be) = \sS^{4}(\be), \sS(\ga) = \sS^{4}(\ga) \in \mcurvweyl(\std)$ as in \eqref{ssrdefined} (where $\ga = \al \circ \be$).
\begin{enumerate}
\item For $\sX$ contained in the $5$-dimensional subalgebra $\balg = \spn\{\sS(\al), \sS(\be), \sS(\ga), \al\cdot \be, \be \cdot \ga, \ga\cdot \al\}$, $\op{\sX}$ preserves $\stu = \spn\{\al, \be, \ga\} \subset \ext^{2}\std$.
\item\label{4dsubiso} The map $\Psi:\balg \to \eno(\stu)$ defined by $\Psi(\sX) = 3\op{\sX}_{\stu}$ is an isometric algebra isomorphism from $(\balg, \opmult, h)$ to the deunitalization $(\sherm_{0}(\stu, h), \jrd, \tfrac{4}{3}G)$ of the $6$-dimensional rank $3$ simple real Euclidean Jordan algebra $(\sherm(\stu, h), \jmult)$, in its realization as the trace-free symmetric endomorphisms of $\stu$ equipped with the product $\jrd$ equal to the traceless part of the usual Jordan product $\jmult$ of endomorphisms and the metric $G(A, B) = \tfrac{1}{3}\tr A \circ B$. 
\item The Killing form $\tau_{\opmult, \balg}(\sX, \sY) = \tr L_{\opmult, \balg}(\sX)L_{\opmult, \balg}(\sY)$ on $(\balg, \opmult)$ satisfies $\tau_{\opmult, \balg} = \tfrac{21}{16}\lb \dum, \dum \ra$.
\item The subalgebra $(\balg, \opmult, h)$ is simple.
\end{enumerate}
\end{lemma}
\begin{proof}
By \eqref{hcsubalgebra} of Lemma \ref{hyperkahlerlemma}, $\sS(\ga) = -\sS(\al) -\sS(\be)$. In \eqref{ssrproduct} and \eqref{ssrproduct2} this yields the relations
\begin{align}\label{ssalbe4d}
\begin{aligned}
&\sS(\al) \opmult \sS(\be) = -\sS(\al) -\sS(\be),& &\sS(\al)\opmult (\sS(\be) - \sS(\ga))  = -(\sS(\be) - \sS(\ga)),&\\
&(\al \cdot \ga) \opmult (\be \cdot \ga) = -\tfrac{9}{2}\al \cdot \be, &
&(\al \cdot \be) \opmult (\al \cdot \be) = 27(\sS(\al) + \sS(\be)),&\\
&\sS(\al)\opmult(\al \cdot \be)  = \tfrac{1}{2}\al \cdot \be,&
&\sS(\al)\opmult(\be \cdot \ga)  = -\be \cdot \ga,&
\end{aligned}
\end{align}
and those obtained from them by permuting $\al$, $\be$, and $\ga$. By \eqref{opalbe} of Lemma \ref{hyperkahlerlemma}, for
\begin{align}\label{sx4d}
\begin{split}
\sX &=  (2x_{1} + x_{2})\sS(\al) + (x_{1} + 2x_{2})\sS(\be) - \tfrac{1}{3}\left(z_{1}\be \cdot \ga + z_{2}\al \cdot \ga + z_{3}\al \cdot \be\right),
\end{split}
\end{align}
the matrix of $\op{\sX}_{\stu}$ with respect to the equal-norm orthogonal ordered basis $\{\al, \be, \ga\}$ of $\stu$ is 
\begin{align}\label{sxopstw}
\X = \begin{pmatrix*}[c]
x_{1} & z_{3} & z_{2} \\
z_{3} & x_{2}& z_{1}\\
z_{2} & z_{1} & - x_{1} - x_{2}
\end{pmatrix*}\in \sherm_{0}(\stu, h).
\end{align}
From \eqref{sxopstw} it is apparent that $\Psi: \balg \to \sherm_{0}(\stu, h)$ defined by $\Psi(\sX) = 3\op{\sX}_{\stu}$ is a linear isomorphism. Because $\opmult$ and $\jrd$ are commutative, by polarization, to check that $\Psi$ is an algebra homomorphism it suffices to check that $\Psi(\sX \opmult \sX) = \Psi(\sX)\jrd \Psi(\sX)$. By \eqref{ssalbe4d},
\begin{align}\label{sxcmsx4d}
\begin{split}
\sX \opmult \sX & = -3(2x_{1}x_{2} + x_{2}^{2} + z_{1}^{2} - z_{3}^{2})\sS(\al)  -3(2x_{1}x_{2} + x_{1}^{2} + z_{2}^{2} - z_{3}^{2})\sS(\be)\\
&\quad - (z_{1}z_{2} + (x_{1} + x_{2})z_{3})\al \cdot \be + (z_{2}z_{3} - x_{1}z_{1})\be \cdot \ga + (z_{1}z_{3} - x_{2}z_{2})\ga \cdot \al.
\end{split}
\end{align}
Comparing \eqref{sxcmsx4d} with
\begin{align}\label{sxopjrd}
\begin{split}
\X \jrd \X= \begin{pmatrix*}[c]
\tfrac{x_{1}^{2} - 2x_{2}^{2} - 2x_{1}x_{2} - 2z_{1}^{2} + z_{2}^{2} + z_{3}^{2}}{3} & z_{1}z_{2} + (x_{1} + x_{2})z_{3} &z_{1}z_{3} - x_{2}z_{2}\\
z_{1}z_{2} + (x_{1} + x_{2})z_{3} &\tfrac{-2x_{1}^{2} +x_{2}^{2} - 2x_{1}x_{2} +z_{1}^{2} -2 z_{2}^{2} + z_{3}^{2} }{3}& z_{2}z_{3} - x_{1}z_{1}\\
z_{1}z_{3} - x_{2}z_{2} &z_{2}z_{3} - x_{1}z_{1} &\tfrac{x_{1}^{2} + x_{2}^{2} + 4x_{1}x_{2}  + z_{1}^{2} + z_{2}^{2} - 2z_{3}^{2}}{3}
\end{pmatrix*},
\end{split}
\end{align}
shows that $\Psi(\sX \opmult \sX) = \Psi(\sX)\jrd\Psi(\sX)$. 

By claims \eqref{ssalbeangle} and \eqref{albegaproducts} of Lemma \ref{hyperkahlerlemma} with $r = 4$, 
\begin{align}\label{obasis4d}
\left\{\tfrac{\sqrt{3}}{2\sqrt{2}}(\sS(\al) + \sS(\be)), -\tfrac{1}{2\sqrt{2}}(\sS(\al) - \sS(\be)), \tfrac{1}{6\sqrt{2}}\be \cdot \ga, \tfrac{1}{6\sqrt{2}}\ga \cdot \al, \tfrac{1}{6\sqrt{2}}\al \cdot \be\right\}
\end{align}
is an orthonormal basis of $\balg$. 
By definition of $G$, \eqref{sxopstw}, and the orthonormality of \eqref{obasis4d} (used to compute the norm of \eqref{sx4d}),
\begin{align}\label{psiGisometry}
\begin{split}
\Psi^{\ast}(G)(\sX, \sX) &= \tfrac{1}{3}\tr\Psi(\sX)^2 = 3\tr \X^{2} = 6\left(x_{1}^{2} + x_{1}x_{2} + x_{2}^{2} + z_{1}^{2} + z_{2}^{2} + z_{3}^{2}\right) = \tfrac{3}{4}|\sX|^{2}.
\end{split}
\end{align}
Slightly tedious calculations using \eqref{ssalbe4d} show that the matrix of the restriction to $\balg$ of $L_{\opmult, \balg}(\sX)$ with respect to the ordered orthonormal basis \eqref{obasis4d} is
\begin{align}\label{Mmatrix}
M= 
\begin{pmatrix*}[c]
-\tfrac{3}{2}(x_{1} + x_{2}) & -\tfrac{\sqrt{3}}{2}(x_{1} - x_{2}) & \tfrac{\sqrt{3}}{2}z_{1} & \tfrac{\sqrt{3}}{2}z_{2} & -\sqrt{3}z_{3}\\
-\tfrac{\sqrt{3}}{2}(x_{1} - x_{2}) & \tfrac{3}{2}(x_{1} + x_{1}) & -\tfrac{3}{2}z_{1} & \tfrac{3}{2}z_{2} & 0\\
\tfrac{\sqrt{3}}{2}z_{1} & -\tfrac{3}{2}z_{1} & -\tfrac{3}{2}x_{1} & \tfrac{3}{2}z_{3} & \tfrac{3}{2}z_{2}\\
 \tfrac{\sqrt{3}}{2}z_{2} & \tfrac{3}{2}z_{2} & \tfrac{3}{2}z_{3} & -\tfrac{3}{2}x_{2} & \tfrac{3}{2}z_{1}\\
-\sqrt{3}z_{3} & 0 & \tfrac{3}{2}z_{2} & \tfrac{3}{2}z_{1} & \tfrac{3}{2}(x_{1} + x_{2})
\end{pmatrix*}
\end{align}
Comparing \eqref{Mmatrix} with \eqref{psiGisometry} shows
\begin{align}
\begin{split}
\tau_{\balg, \opmult} &= \tr L_{\opmult, \balg}(\sX)L_{\opmult, \balg}(\sX) = \tr M^{2} = \tfrac{21}{2}\left(x_{1}^{2} + x_{1}x_{2} + x_{2}^{2} + z_{1}^{2} + z_{2}^{2} + z_{3}^{2} \right) = \tfrac{21}{16}|\sX|^{2}.
\end{split}
\end{align}
That $(\balg, \opmult, h)$ is simple follows from \eqref{4dsubiso} and Lemma \ref{deunitsimplelemma}.
\end{proof}

\section{Subalgebra of Kähler-Weyl tensors and the \texorpdfstring{$4$}{}-dimensional case revisited}\label{kahlersection}
A $2n$-dimensional \emph{Kähler vector space} $(\ste, h, J, \om)$ is a Euclidean vector space $(\ste, h)$ of dimension $m = 2n$ equipped with a compatible complex structure $J_{i}\,^{j}$, meaning that $J_{i}\,^{p}J_{j}\,^{q}h_{pq} = h_{ij}$ and $\om_{ij} = J_{i}\,^{p}h_{pj}$ is a symplectic form. When a Kähler vector space is fixed, the abstract unitary group $U(n)$ is identified with the unitary group $U(h, J)$ of linear automorphisms of $\ste$ preserving $h$ and $J$.

\begin{lemma}\label{kahlersubalgebralemma}
Let $(\ste, h, J, \om)$ be a $2n$-dimensional Kähler vector space. The $U(n)$-submodules
\begin{align}\label{mcurvkahkerdefined}
\begin{split}
\mcurvkahler(\std) &= \{\sX \in \mcurv(\std): J_{i}\,^{a}J_{j}\,^{b}\sX_{abkl} = \sX_{ijkl}\}
= \{\sX \in \mcurv(\std):J_{[i}\,^{p}\sX_{j]pkl} = 0\},
\end{split}
\end{align}
of curvature tensors of \emph{Kähler type}, and $\mcurvkahlerweyl(\std) = \mcurvkahler(\std) \cap \mcurvweyl(\std)$, of \emph{Kähler-Weyl} curvature tensors, are subalgebras of $(\mcurv(\std), \opmult)$ on which $U(n)$ acts by automorphisms.
\end{lemma}

\begin{proof}
Let $\sX, \sY \in \mcurvkahler(\std)$. Write $B_{ijkl} = B(\sX, \sY)_{ijkl}$. 
By \eqref{prebelskew},
\begin{align}\label{mck1}
\begin{split}
8J_{i}\,^{a}J_{j}\,^{b}B_{ab[kl]} = J_{i}\,^{a}J_{j}\,^{b}(\sX_{ab}\,^{pq}\sY_{pqkl} + \sY_{ab}\,^{pq}\sX_{pqkl}) = \sX_{ij}\,^{pq}\sY_{pqkl} + \sY_{ij}\,^{pq}\sX_{pqkl} = 8B_{ij[kl]}.
\end{split}
\end{align}
Similarly,
\begin{align}\label{mck2}
\begin{split}
2&J_{j}\,^{a}J_{l}\,^{b}B_{iakb} = J_{j}\,^{a}J_{l}\,^{b}\sX_{ipaq}\sY_{k}\,^{p}\,_{b}\,^{q} + J_{j}\,^{a}J_{l}\,^{b}\sY_{ipaq}\sX_{k}\,^{p}\,_{b}\,^{q}\\
& = J_{q}\,^{a}J_{bl}\sX_{ipja}\sY_{k}\,^{pbq} +  J_{q}\,^{a}J_{bl}\sY_{ipja}\sX_{k}\,^{pbq} 
 = - J_{q}\,^{a}J_{b}\,^{q}\sX_{ipja}\sY_{k}\,^{p}\,_{l}\,^{b}- J_{q}\,^{a}J_{b}\,^{q}\sY_{ipja}\sX_{k}\,^{p}\,_{l}\,^{b}\\
& = \sX_{ipjb}\sY_{k}\,^{p}\,_{l}\,^{b} + \sY_{ipjb}\sX_{k}\,^{p}\,_{l}\,^{b} = 2B_{ijkl}.
\end{split}
\end{align}
Together \eqref{curvmultdefined}, \eqref{mck1}, and \eqref{mck2} yield
\begin{align}
\begin{split}
J_{i}\,^{a}J_{j}\,^{b}(\sX \opmult\sY)_{abkl} &= J_{i}\,^{a}J_{j}\,^{b}\left(-2B_{ab[kl]} - B_{akbl} + B_{alkb}\right)
 = -2B_{ij[kl]} - 2B_{i[k|j|l]}   = (\sX \opmult \sY)_{ijkl}.
\end{split}
\end{align}
This shows $\mcurvkahler(\std)$ is a subalgebra of $(\mcurv(\std), \opmult)$. 
By Lemmas \ref{opweylsubalgebralemma} and \ref{kahlersubalgebralemma}, $\mcurvkahlerweyl(\std) = \mcurvkahler(\std) \cap \mcurvweyl(\std)$ is a subalgebra of $(\mcurv(\std), \opmult)$.
That $U(n)$ acts by automorphisms on these subalgebras follows from the containment $U(n) \subset O(2n)$.
\end{proof}

A Kähler vector space is canonically oriented by the Euclidean volume form $\ep = \tfrac{1}{n!}\om^{n}$. 
Lemma \ref{kahlerweylasdlemma} shows that, for a $4$-dimensional Kähler vector space, the space of Weyl curvature tensors anti-self-dual with respect to the orientation determined by the complex structure coincides with the space of Kähler-Weyl curvature tensors. As is explained subsequently, this has the consequence of showing that the subalgebra $(\mcurvkahlerweyl(\std), \opmult)$ is nontrivial whenever $\dim \std \geq 4$.

\begin{lemma}\label{kahlerweylasdlemma}
Let $(\ste, h, J, \om)$ be a $4$-dimensional Kähler vector space oriented by $\ep = \tfrac{1}{2}\om \dwedge \om$.
\begin{enumerate}
\item\label{skalbedefined} For $\al, \be \in \ext^{2}_{-}\std$, $\sK(\al, \be) = \al \cdot \be + (\al \circ \om)\kwedge (\be \circ \om)$ is contained in $\mcurvkahler(\std)$.
\item\label{4dtfsk} If $\al, \be \in \ext^{2}_{-}\std$, then $\tf(\al \cdot \be) = \tfrac{3}{4}\tf \sK(\al, \be) \in \mcurvkahlerweyl(\std)$.
\item\label{4dskwal} If $\al \in \ext^{2}_{-}(\std)$ satisfies $|\al|^{2}_{h} = 4$, then $\sS(\al)  = -\tfrac{1}{6}\tf(\al \cdot \al) = -\tfrac{1}{8}\tf \sK(\al, \al)$, where $\sS(\al)$ is as defined in \eqref{ssrdefined}, is a nontrivial idempotent in $(\mcurvkahlerweyl(\std), \opmult)$ satisfying $|\sS(\al)|^{2} = \tfrac{8}{3}$.
\item\label{kahlerweylintersection} $\mcurvkahlerweyl(\std) = \mcurvweyl^{-}(\std)$.
\end{enumerate}
\end{lemma}
\begin{proof}
The operator $\jstar \in \eno(\ext^{2}\std)$ can be expressed $\jstar = \tfrac{1}{2}\ep^{\sharp} = \tfrac{1}{2}(\om \tensor \om)^{\sharp} + \qp_{\jmult}(\om)$, (recall  from \eqref{sharp4defined} the isomorphism $\sharp$ and its inverse $\flat$).
It follows that $\al \in \ext^{2}\std$ is contained in $\ext^{2}_{-}\std$ if and only if $\lb \al, \om\ra =0$ and $\qp_{\jmult}(\om)(\al) = -\al$. The latter condition is equivalent to $\al\circ \om = \om \circ \al \in S^{2}\std$, and this shows that the element $\sK(\al, \be)$ of \eqref{skalbedefined} is correctly defined.

Suppose $\al, \be \in \ext^{2}\std$. Substituting $\tfrac{1}{2}\lb \al, \star \be \ra \ep  = \al \dwedge \be$ into \eqref{cdotdefined} yields the alternative expressions
\begin{align}\label{4dcdot}
\begin{split}
\al \cdot \be &= \tfrac{3}{2}(\al \tensor \be + \be \tensor \al) - \tfrac{1}{4}\lb \al, \star \be \ra \ep= \tfrac{3}{2}(\al \tensor \be + \be \tensor \al) - \tfrac{1}{8}\lb \al, \star \be \ra \om \dwedge \om.
\end{split}
\end{align}
Taking $\al = \be = \om$ in \eqref{4dcdot} yields
\begin{align}\label{4domcdot}
\om \cdot \om = 3\om\tensor \om - \tfrac{1}{2} \om \dwedge \om = 3\om\tensor \om -\ep.
\end{align}
By Example \ref{hhidexample}, $2\Id_{\ext^{2}\std}^{\flat} = -2\op{h \kwedge h}_{\ext^{2}\std}^{\flat} = h\kwedge h$, so
\begin{align}\label{kahlerep}
(\qp_{\jmult}(\om)\circ \ep^{\sharp})^{\flat} =- \om \tensor \om + 2\Id_{\ext^{2}\std}^{\flat}=  - \om \tensor \om +h \kwedge h.
\end{align}
Combining $\al \dwedge \be = \tfrac{1}{2}\lb \al, \star \be \ra \ep$ with \eqref{kahlerep} and noting $(\al \dwedge \be)^{\sharp} = (\al \tensor \be + \be \tensor \al)^{\sharp} + 4\qp_{\jmult}(\al ,\be)$ yields
\begin{align}\label{jjalwedgebe}
\begin{split}
\tfrac{1}{2}\lb \al , \star \be \ra(\om\tensor \om - h\kwedge h) & = -\tfrac{1}{2}\lb \al, \jstar \be \ra\left(\qp_{\jmult}(\om)\circ \ep^{\sharp}\right)^{\flat} = - \left(\qp_{\jmult}(\om)\circ (\al \dwedge \be)^{\sharp}\right)^{\flat}\\
& = -\left(\qp_{\jmult}(\om)(\al)\tensor \be + \al \tensor\qp_{\jmult}(\om)(\be)\right)- 4\left(\qp_{\jmult}(\om)\circ\qp_{\jmult}(\al ,\be)\right)^{\flat}.
\end{split}
\end{align}
Now suppose $\al, \be \in \asdf\std$. In this case, by \eqref{gasiopal}, 
\begin{align}
-\qp_{\jmult}(\om)\circ\qp_{\jmult}(\al ,\be)= -(\qp_{\jmult}(\al \circ \om, \be \circ \om) = \op{(\al \circ \om) \kwedge (\be \circ \om)}_{\ext^{2}\std},
\end{align}
and, because $-2\op{(\al \circ \om) \kwedge (\be \circ \om)}_{\ext^{2}\std}^{\flat} = (\al \circ \om) \kwedge (\be \circ \om)$, substituting this into \eqref{jjalwedgebe} yields
\begin{align}\label{jjalwedgebeasd}
\begin{split}
(\al \circ \om) \kwedge (\be \circ \om) & = \tfrac{1}{2}(\al \tensor \be + \be \tensor \al) + \tfrac{1}{4}\lb \al, \be \ra(\om\tensor \om - h\kwedge h).
\end{split}
\end{align}
Substituting \eqref{4dcdot} and \eqref{4domcdot} into \eqref{jjalwedgebeasd} and simplifying the result using \eqref{kahlerep} yields
\begin{align}\label{skalbe4d}
\begin{split}
\sK(\al, \be) &= \al \cdot \be + (\al \circ \om) \kwedge (\be \circ \om)  = \tfrac{4}{3} \al \cdot \be + \tfrac{1}{12}\lb \al,  \be \ra  \om\cdot \om   -  \tfrac{1}{4}\lb \al, \be \ra  h\kwedge h\\
& = 2(\al \tensor \be + \be \tensor \al) + \tfrac{1}{4}\lb \al, \be\ra\left(\ep + \om \tensor \om  - h \kwedge h\right).	
\end{split}
\end{align}
A straightforward calculation using the last equality of \eqref{skalbe4d} shows $\qp_{\jmult}(\om)\circ \sK(\al, \be)^{\sharp} = - \sK(\al, \be)^{\sharp}$, so $\sK(\al, \be) \in \mcurvkahler(\std)$. 
By \eqref{skalbe4d}, \eqref{tfalcdotbe}, and \eqref{4dso1} of Lemma \ref{albecontractlemma},
\begin{align}\label{tfalbecdot4d}
\tfrac{3}{4}\tf \sK(\al, \be) &= \tf(\al \cdot \be) = \al \cdot \be + 3(\al \jmult \be) \kwedge h + \tfrac{1}{2}\lb \al, \be \ra h \kwedge h = \al \cdot \be - \tfrac{1}{4}\lb \al, \be\ra h \kwedge h.
\end{align}
This shows \eqref{4dtfsk}.
If $\al \in \ext^{2}_{-}(\std)$, then $\al \dwedge \al = -\al \dwedge \star \al = -\tfrac{1}{4}|\al|^{2}\om \dwedge \om$, so if $|\al|^{2}_{h} = 4$, then $\al \dwedge \al= - \om \dwedge \om$. Hence $0 = (\al \dwedge \al + \om \dwedge \om)^{\sharp}(\om) = 4(\qp_{\jmult}(\al)(\om) + \om)$, so $\al \circ \om \circ \al = -\om$. This implies $\al \circ \al = - \al \circ \om \circ \om \circ \al = -\al \circ \om \circ \al \circ \om  = \om \circ \om = -h$, showing $\al$ satisfies the hypotheses of Lemma \ref{weylidempotentslemma}. Taking $\be = \al$ in \eqref{tfalbecdot4d} yields \eqref{4dskwal}. 

For $\sX \in\mcurvkahler(\std)$, pairing $J_{i}\,^{p}\sX_{pjkl} =  J_{j}\,^{p}\sX_{ipkl}$ with $h^{il}$ yields $\om^{pq}\sX_{pqjk} = -2\om^{pq}\sX_{pjkq} = J_{j}\,^{p}\rictr(\sX)_{pk}$, so, because $\sX \in\mcurvkahlerweyl(\std)$, $(\star \sX)_{ijkl} = (\tfrac{1}{2}\om_{ij}\om^{ab} - J_{i}\,^{a}J_{j}\,^{b})\sX_{abkl} = -\sX_{ijkl}$. This shows that $\mcurvkahlerweyl(\std) \subset \mcurvweyl^{-}(\std)$. Since both $\mcurvkahlerweyl(\std)$ and $\mcurvweyl^{-}(\std)$ are $SU(2)$-modules and $\mcurvweyl^{-}(\std)$ is an irreducible $SU(2)$-module, to prove equality it suffices to show that $\mcurvkahlerweyl(\std)$ has dimension at least $1$. By the preceding paragraph there exists $\al\in \asdf\std$ such that $\tf(\al \cdot \al)$ is nontrivial and is contained in $\mcurvkahlerweyl(\std)$. This proves claim \eqref{kahlerweylintersection}. 
\end{proof}

\begin{theorem}\label{einsteinkahlertheorem}
Let $(\ste, h, J, \om)$ be a Kähler vector space of dimension $m = 2n \geq 4$. The subspace $\mcurvkahlerweyl(\std)$ is a simple subalgebra of $(\mcurvweyl(\std), \opmult)$ that is exact and Killing metrized, with Killing form equal to a positive multiple of the metric $\lb \dum, \dum \ra$, and on which $U(n) = U(\ste, J, h)$ acts irreducibly by automorphisms. 
\end{theorem}

\begin{proof}
The group $U(n)$ acts irreducibly on $\mcurvkahlerweyl(\std)$ \cite[Theorem $6.4$]{Tricerri-Vanhecke}. 
It is straightforward to check that when, in the setting of Lemma \ref{mcsubspacelemma}, $\ste$ is equipped with a Kähler structure and $\stw \subset \ste$ is a Kähler subspace, the map $\imt$ of Lemma \ref{mcsubspacelemma} is an injective algebra homomorphism from $(\mcurvkahler(\stwd), \opmult)$ and $(\mcurvkahlerweyl(\stwd), \opmult)$ to $(\mcurvkahler(\std), \opmult)$ and $(\mcurvkahlerweyl(\std), \opmult)$, respectively. Consequently, by Lemmas \ref{mcsubspacelemma} and \ref{kahlerweylasdlemma}, $(\mcurvkahlerweyl(\std), \opmult)$ contains a nontrivial idempotent, so its multiplication is nontrivial. 
The group $U(n)$ acts on $(\mcurvkahlerweyl(\std), \opmult)$ by isometric automorphisms because it is a subgroup of $O(2n)$, which acts on $(\mcurv(\std), \opmult)$ by algebra automorphisms.
By Theorem \ref{preeinsteintheorem}, $(\mcurvkahlerweyl(\std), \opmult)$ is exact and Killing metrized, with Killing form equal to a positive multiple of the metric $\lb \dum, \dum \ra$. Because every finite-dimensional irreducible representation of $U(n)$ restricts to an irreducible representation of $SU(n)$, $\mcurvkahlerweyl(\std)$ is irreducible as an $SU(n)$-module. Since the action by automorphisms of $SU(n)$ on $\mcurvkahlerweyl(\std)$ is irreducible, the simplicity of $(\mcurvkahlerweyl(\std), \opmult)$ follows from Theorem \ref{simpletheorem}.
\end{proof}

\section{Examples related with idempotents}\label{computationsection}
This final section shows the viability of making explicit computations in $(\mcurvweyl(\std), \opmult)$ and describes some relations among the idempotents constructed earlier that are suggestive both with respect to the structure of $(\mcurvweyl(\std), \opmult)$ and in the context of developing a structure theory for commutative algebras metrized by a trace-form. It is indicated how this yields an alternative proof of Theorem \ref{mcurvweyltheorem}. Theorem \ref{splitcomplextheorem} shows that $(\mcurvweyl(\std), \opmult)$ contains many two-dimensional unital associative subalgebras isomorphic to the paracomplex numbers and exhibits zeros of its cubic polynomial. The final Example \ref{matsuoexample} shows that, when $\dim \ste \geq 6$, $(\mcurvweyl(\std), \opmult)$ contains subalgebras isomorphic to Matsuo algebras.

\begin{lemma}
Let $(\ste, h)$ be a Euclidean vector space. For $x, y, z, w \in \std$, 
\begin{align}\label{xykwedgezwcycle}
\begin{split}
0  =(x\sprod y)\kwedge (z\sprod w) &+ (y\sprod z)\kwedge (x\sprod w) + (z\sprod x)\kwedge (y\sprod w),\\
0 = (x\dwedge y) \cdot (z \dwedge w) &+ (y\dwedge z) \cdot (x \dwedge w) + (z\dwedge x) \cdot (y \dwedge w),
\end{split}\\
\label{mixedidentities}
12 (x\sprod y)\kwedge (z\sprod w) &
= (x \dwedge z) \cdot (y \dwedge w)+ (y\dwedge z) \cdot (x \dwedge w),\\
\label{mixedidentitiesb}
 (x\dwedge y)\cdot (z \dwedge w)&= 4(x\sprod z)\kwedge (y\sprod w) - 4(y\sprod z)\kwedge (x \sprod w).
\end{align}
\end{lemma}

\begin{proof}
The identities \eqref{xykwedgezwcycle} follow from the definitions \eqref{kwedgedefined} and \eqref{cdotdefined}; precisely, summing $(x\sprod y)\sprod (z \sprod w)$ and $(x \dwedge y)\sprod (z \dwedge w)$ cyclically over $x$, $y$, and $z$ yields elements in $\ker \stwoprojdual$ and $\ker\stwoproj$, respectively.
By the definitions, \eqref{kwedgedefined}, and \eqref{cdotdefined},
\begin{align}\label{xxkwedgeyy}
\begin{split}
 - 12&(x\sprod y)\kwedge (x\sprod y)  \\
&= 3((x + y)\sprod(x+y) - x\sprod x - y \sprod y)\kwedge ((x + y)\sprod(x+y) - x\sprod x - y \sprod y) \\
&=  3(x \tensor x + y \tensor y)\kwedge (x \tensor x + y \tensor y) = 6(x\tensor x)\kwedge (y \tensor y) \\
&= 3(x\dwedge y)\tensor (x\dwedge y) = (x\dwedge y)\cdot (x\dwedge y) .
\end{split}
\end{align}
Polarizing \eqref{xxkwedgeyy} first in $x$ then in $y$ and using \eqref{xykwedgezwcycle} yields \eqref{mixedidentities}. Using \eqref{mixedidentities} and \eqref{xykwedgezwcycle} to evaluate the right-hand side of \eqref{mixedidentitiesb} yields the left-hand side of \eqref{mixedidentitiesb}.
\end{proof}

\begin{example}\label{tfbebeexample}
Let $(\ste, h)$ be a Euclidean vector space. 
Let $x, y \in \std$ be orthogonal unit norm vectors. 
Let $\si = x \dwedge y$, so that $\si \circ \si = -\al$ where $\al = x \tensor x + y \tensor y$. 
Then $\al$ satisfies $\al \circ \al = \al$, $\tr \al = 2$, and, by \eqref{xxkwedgeyy}, $\si \cdot \si = 3\al \kwedge \al$, 
so that $4\sS^{2}(x\dwedge y) = 3\al \kwedge \al - \si \cdot \si = 0$. On the other hand, by \eqref{alalsquared}, Lemma \ref{weylidempotentslemma}, and \eqref{xxkwedgeyy}, $\sH(\al) = -\al \kwedge \al =  - (x\dwedge y) \tensor (x \dwedge y) = \sK^{2}(\si)$ is an idempotent in $(\mcurv(\std), \opmult)$ satisfying $\rictr(\al \kwedge \al) = \al$. 
Similarly, if $\{x, y, z, w\}\subset  \std $ is an orthonormal set, 
\begin{align}
\begin{split}
\left((x\dwedge y)\tensor (x \dwedge y)\right)\opmult\left((y\dwedge z)\tensor (y \dwedge z)\right) &=- \tfrac{1}{2}(x\dwedge z)\tensor (x\dwedge z),\\
\left((x\dwedge y)\tensor (x \dwedge y)\right)\opmult\left((z\dwedge w)\tensor (z \dwedge w)\right) &= 0.
\end{split}
\end{align}
These are most easily computed using \eqref{symalbecmgasi} or \eqref{albecmultgasi} in conjunction with \eqref{xxkwedgeyy}.  
By Corollary \ref{kwedgeidempotentcorollary}, as the idempotents $\al = x\tensor x + y \tensor y, \be = z\tensor z + w \tensor w \in \idem(S^{2}, \jmult)$ are orthogonal,
\begin{align}\label{sbxyzw}
\begin{split}
\sB&(x, y, z, w)  = \sB(\al, \be) =  -\tfrac{2}{3}\left((x\dwedge y)\tensor (x \dwedge y) + (z\dwedge w)\tensor(z \dwedge w)\right) \\
&\quad + \tfrac{1}{3}\left((x\dwedge z)\tensor (x \dwedge z) + (x\dwedge w)\tensor (x \dwedge w)+ (y\dwedge z)\tensor (y \dwedge z)+ (y\dwedge w)\tensor (y \dwedge w) \right)
\end{split}
\end{align}
is an idempotent in $(\mcurvweyl(\std), \opmult)$ satisfying $|\sB(x, y, z, w)|^{2} = 16/3$. The symmetries $\sB(y, x, z, w) =\sB(x, y, z, w) = \sB(x, y, w, z) = \sB(z, w, x, y)$, evident from \eqref{sbxyzw}, show that \eqref{sbxyzw} yields only three distinct idempotents, namely $\sB(x, y, z, w)$, $\sB(y, z, x, w)$, and $\sB(z, x, y, w)$. From \eqref{sbxyzw} it is apparent that $\sB(x, y, z, w) + \sB(y, z, x, w) + \sB(z, x, y, w) =0$. This shows $-\sB(x, y, z, w) - \sB(y, z, x, w)$ is idempotent, which implies
\begin{align}
\sB(x, y, z, w) \opmult \sB(y, z, x, w) = \sB(z, x, y, w) = - \sB(x, y, z, w) - \sB(y, z, x, w).
\end{align}
This shows $\sB(x, y, z, w)$, $\sB(y, z, x, w)$, and $\sB(z, x, y, w)$ span a $2$-dimensional subalgebra isomorphic to the algebra $\ealg^{2}(\rea)$ called the simplicial algebra in \cite{Fox-simplicial}; it is the unique $2$-dimensional metrized commutative algebra with automorphism group $S_{3}$. 
\end{example}

\begin{theorem}\label{splitcomplextheorem}
Let $(\ste, h)$ be a Euclidean vector space of dimension $n \geq 4$. For an $h$-orthonormal set $\{x, y, z, w\} \subset \std$,
\begin{align}\label{sedefined}
\sE = \sE(x, y, z, w) = \tfrac{2}{3}(x \dwedge y)\cdot (z \dwedge w) \in \mcurvweyl(\std)
\end{align}
is a zero of the cubic polynomial $P_{\mcurvweyl(\std), \opmult}$ that satisfies:
\begin{enumerate}
\item $\sE\opmult(\sE \opmult \sE) = \sE$ and $\sE \opmult \sE = \sB(x, y, z, w)$ is idempotent.
\item The subspace $\spn\{\sE, \sE\opmult\sE\} \subset \mcurvweyl(\std)$ is an associative subalgebra isomorphic to the algebra $\rea[t]/(t^{2} - 1)$ of paracomplex numbers via the linear map $a\sE\opmult \sE + b \sE \to a + bt$.
\end{enumerate}
\end{theorem}

\begin{proof}
The two-forms $\al^{\pm}(x, y, z, w) = \al^{\pm} = x\dwedge y \pm z \dwedge w$ satisfy $\al^{\pm} \circ \al^{\pm} = -g$ for $g = x\tensor x + y\tensor y + z \tensor z + w\tensor w$ and $[\al^{+}, \al^{-}] = 0$. Because $\lb \al^{+}, \al^{-}\ra = 0$ and $\al^{+}\circ \al^{-}\circ \al^{+}\circ \al^{-} = g$, by \eqref{ssangle}, $\lb \sS^{4}(\al^{+}), \sS^{4}(\al^{-})\ra = 0$. Define $\sS^{\pm}(x, y, z, w) = \sS^{4}(\al^{\pm})$. 
By \eqref{xxkwedgeyy},
\begin{align}\label{orthonormalmixedidentities}
\begin{aligned}
(&x\dwedge y) \cdot (x \dwedge y) +(z\dwedge w) \cdot (z \dwedge w)  - (x\dwedge z)\cdot (x\dwedge z) - (y\dwedge w) \cdot (y \dwedge w),\\
&= -12\left((x\sprod y) \kwedge (x \sprod y) +(z\sprod w) \kwedge (z \sprod w)  
- (x\sprod z)\kwedge (x\sprod z) - (y\sprod w) \kwedge (y \sprod w) \right),\\
&= 6\left(x\tensor x- w \tensor w\right)\kwedge\left(y\tensor y - z\tensor z\right). 
\end{aligned}
\end{align}
By \eqref{mixedidentities}, \eqref{mixedidentitiesb}, and \eqref{orthonormalmixedidentities},
\begin{align}\label{ssxyzw}
\begin{split}
6&\sS^{\pm}(x, y, z, w) = \mp 2(x \dwedge y)\cdot (z \dwedge w)  
- \tfrac{2}{3}\left((x\dwedge y) \cdot (x \dwedge y) + (z\dwedge w) \cdot (z \dwedge w)\right)\\
&+\tfrac{1}{3}\left( (x\dwedge z) \cdot (x \dwedge z)+  (y\dwedge w) \cdot (y \dwedge w) +(y\dwedge z) \cdot (y \dwedge z)+ (x\dwedge w) \cdot (x \dwedge w) \right).
\end{split}
\end{align}
Comparing \eqref{sbxyzw} and \eqref{sedefined} with \eqref{ssxyzw} shows that
\begin{align}\label{sspmsb}
\begin{split}
\sS^{+}(x, y, z, w) + \sS^{-}(x, y, z, w) &= \sB(x, y, z, w),\\
 \sS^{-}(x, y, z, w) - \sS^{+}(x, y, z, w) &= \sE(x, y, z, w). 
\end{split}
\end{align}
Write $\sS^{\pm} = \sS^{\pm}(x, y, z, w)$ and $\sB = \sB(x, y, z, w)$.
Because $\sS^{+}$, $\sS^{-}$, and $\sB(x, y, z, w)$ are idempotents in $(\mcurvweyl(\std), \opmult)$, \eqref{sspmsb} implies $\sS^{+}\opmult \sS^{-} = 0$ and $\lb \sS^{+}, \sS^{-}\ra = 0$. 
By \eqref{sspmsb}, $\sE\opmult \sE = \sS^{-} + \sS^{+} = \sB$. This shows $\sE\opmult \sE$ is idempotent and $\lb \sE, \sE\opmult \sE\ra = \lb \sS^{-} - \sS^{+}, \sS^{-} + \sS^{+}\ra = |\sS^{-}|^{2} - |\sS^{+}|^{2} = 0$, the last equality because $|\sS^{\pm}(x, y, z, w)|^{2} = 8/3$ by Lemma \ref{ssranglelemma} or direct computation using \eqref{ssxyzw}. Finally, $\sE \opmult (\sE \opmult \sE) = \sE \opmult \sB = (\sS^{-} - \sS^{+})\opmult (\sS^{-} +\sS^{+}) = \sS^{-} - \sS^{+} = \sE$. The final claim follows.
\end{proof}

For an $h$-orthonormal set $\{x, y, z, w\} \subset \std$, by \eqref{ssxyzw} and \eqref{xykwedgezwcycle}, 
\begin{align}\label{sscyclic}
\sS^{\pm}(x, y, z, w) + \sS^{\pm}(y, z, x, w) + \sS^{\pm}(z, x, y, w) = 0.
\end{align} 
The symmetries $\sS^{\pm}(x, y, z, w) = \sS^{\pm}(z, w, x, y) = \sS^{\pm}(y, x, w, z) = \sS^{\pm}(w, z, y, x)$ and $\sS^{+}(y, x, z, w) = \sS^{-}(x, y, z, w)$, evident from \eqref{ssxyzw}, show that, under the action of $S_{4}$ permuting $\{x, y, z, w\}$, \eqref{ssxyzw} yields only six distinct idempotents, namely $\sS^{\pm}(x, y, z, w)$, $\sS^{\pm}(y, z, x, w)$, and $\sS^{\pm}(z, x, y, w)$. By \eqref{sscyclic}, $-\sS^{\pm}(x, y, z, w) - \sS^{\pm}(y, z, x, w)$ is idempotent, which implies
\begin{align}\label{sspmxyzw}
\sS^{\pm}(x, y, z, w) \opmult \sS^{\pm}(y, z, x, w)  = -\sS^{\pm}(x, y, z, w) - \sS^{\pm}(y, z, x, w)= \sS^{\pm}(z, x, y, w).
\end{align}
(Alternatively, \eqref{sspmxyzw} follows from Lemma \ref{hyperkahlerlemma}.) This shows $\sS^{\pm}(x, y, z, w)$, $\sS^{\pm}(y, z, x, w)$, and $\sS^{\pm}(z, x, y, w)$ span a $2$-dimensional subalgebra isomorphic to $\ealg^{2}(\rea)$.
Computations using \eqref{alcdotbegacdotsinorm}-\eqref{alcdotbegawedgesinorm}, as in the proof of Lemma \ref{weylidempotentslemma}, or computations using $\al^{+}(x, y, z, w)\jmult \al^{+}(y, z, x, w) = 0$ and $[\al^{+}(x, y, z, w), \al^{-}(y, z, x, w)] = 0$ and those relations obtained from them via the action of $S_{4}$ on $\{x, y, z, w\}$ together with \eqref{ssangle} of Lemma \ref{ssranglelemma}, show that $ \sS^{\pm}(x, y, z, w)$ is orthogonal to the span of $\sS^{\mp}(x, y, z, w)$ and $\sS^{\mp}(y, z, x, w)$ while $|\sS^{\pm}(x, y, z, w)|^{2} = 8/3$ and $\lb \sS^{\pm}(x, y, z, w), \sS^{\pm}(y,z, x, w)\ra = -4/3$. The last identity gives an example where equality holds in \eqref{ssangle}. 

Define $\sY(x, y, z, w) = (x\dwedge z)\cdot (z \dwedge z) - (y\dwedge w)\cdot (y \dwedge w)$.
By Lemma \ref{4dsubalgebralemma} the element
\begin{align}
\begin{split}
\sC^{\pm}(x, y, z, w) &= \al^{\pm}(x, y, z, w) \cdot \al^{\pm}(y, z, x, w)  \\&= \sY(x, z, w, y) \pm \sY(y, w, x, z)
= \al^{\mp}(z, y, x, w) \cdot \al^{\mp}(y, x, z, w).
\end{split}
\end{align}
and its cyclic permutations in $x$, $y$, and $z$ are orthogonal eigenvectors of $L_{\opmult}(\sS^{\pm}(x, y, z, w))$. By definition $\sS^{\pm}(x, y, z, w) = \sS^{4}(\al^{\pm})$, $\sS^{\pm}(y, z, x, w) = \sS^{4}(\be^{\pm})$, and $\sS^{\pm}(z, x, y, w) = \sS^{4}(\ga^{\pm})$, where $\al^{\pm} = \al^{\pm}(x, y, z, w)$, 
$\be^{\pm} = \al^{\pm}(y, z, x, w)$, and $\ga^{\pm} = \al^{\pm}(z, x, y, w)$. There hold $\al^{\pm} \circ \al^{\pm} = \be^{\pm} \circ \be^{\pm} = \ga^{\pm} \circ \ga^{\pm} = -g$, 
and $\al^{\pm} \circ \be^{\pm} = \ga^{\pm} = -\be^{\pm} \circ \al^{\pm}$ and its cyclic permutations. 
Computations using $\al^{+}\circ \al^{-} = \al^{-}\circ \al^{+} = -x\tensor x - y\tensor y + z\tensor z + w \tensor w$; $\al^{\pm}\circ \be^{\mp} = \be^{\mp}\circ \al^{\pm} = 2x\sprod z \pm 2y \sprod w$; $\al^{+}\circ \ga^{-} = \ga^{-}\circ \al^{+} = 2y\sprod z - 2x \sprod w$; that, by \eqref{cdwpre}, $(\al^{\pm}\cdot \al^{\pm})\opmult (g\kwedge g) = -3g\kwedge g$ and $(\al^{\pm}\cdot \be^{\pm})\opmult (g \kwedge g) = 0$; \eqref{alalcmultbebe}; \eqref{xxkwedgeyy}; and \eqref{mixedidentities} yield
$(\al^{+}\cdot \al^{+})\opmult(\al^{-}\cdot \al^{-}) = -3g\kwedge g$, $(\al^{+}\cdot \al^{+})\opmult(\be^{-}\cdot \be^{-}) = -3g\kwedge g$,
$(\al^{+}\cdot \al^{+})\opmult(\al^{-}\cdot \be^{-})=0$, $(\al^{+}\cdot \al^{+})\opmult(\be^{-}\cdot \ga^{-}) =0$, $(\al^{+}\cdot \be^{+})\opmult(\al^{-}\cdot \be^{-}) =0$, and $(\al^{+}\cdot \be^{+})\opmult(\al^{-}\cdot \ga^{-}) =0$.
Together these imply $36\sS^{4}(\al^{+})\opmult \sS^{4}(\al^{-})= (g\kwedge g -\al^{+}\cdot \al^{+})\opmult(g\kwedge g -\al^{-}\cdot \al^{-}) = 0$, $36\sS^{4}(\al^{+})\opmult \sS^{4}(\be^{-})= (g\kwedge g - \al^{+}\cdot \al^{+} )\opmult(g\kwedge g -\be^{-}\cdot \be^{-}) = 0$, $6\sS^{4}(\al^{+})\opmult \sC^{-}(x, y, z, w) = - (\al^{+}\cdot \al^{+})\opmult(\al^{-}\cdot \be^{-}) = 0$, and $6\sS^{4}(\al^{+})\opmult \sC^{-}( y, z, x, w) = - (\al^{+}\cdot \al^{+})\opmult(\be^{-}\cdot \ga^{-}) = 0$. 
In particular, $\sS^{+}(x, y, z, w) \opmult \sS^{-}(y, z, x, w) =0$.
It follows that 
\begin{align}
\balg^{\pm} = \spn\{\sS^{\pm}(x, y, z, w), \sS^{\pm}(y, z, x, w), \sC^{\pm}(x, y, z, w), \sC^{\pm}(y, z, x, w), \sC^{\pm}(z, x, y, w)\}
\end{align}
are orthogonal $5$-dimensional subalgebras of $(\mcurvweyl(\std), \opmult, h)$ that satisfy $\balg^{+}\opmult \balg^{-} = \{0\}$ and each of which is as in Lemma \ref{4dsubalgebralemma}. The two nonisomorphic structures of a $\cliff_{2}$-module on $\ste$ are represented by the hyper-Kähler structures determined by the pairs $(\al^{\pm}, \be^{\pm}) \in \oc(\ste, h)^{2}$.
 
\begin{example}\label{maxdegenexample}
Let $(\ste, h)$ be a Euclidean vector space of dimension $n > 3$. 
Let $x, y \in \std$ be orthogonal unit norm vectors. Since $\al = x\tensor x + y \tensor y$ satisfies $\al \circ \al = \al$ and $\tr \al = 2$, by Lemma \ref{kwedgeidempotentcorollary}, $\sB(\al)$ defined by \eqref{sbdefined} satisfies $|\sB(\al)|^{2} = \tfrac{4(n-2)^{2}}{(n-1)(n-3)}$. Let $\hat{\al} = h - \al$ and write $\hat{\al} = \sum_{\al = 1}^{n-2}z^{(\al)}\tensor z^{(\al)}$ where $\{z^{(1)}, \dots, z^{(n-2)}\}$ is an orthonormal basis of the orthocomplement of $\spn\{x, y\}$. (Note that $n > 3$ is assumed because $\hat{\al} \kwedge \hat{\al} = 0$ if $n = 3$.)
By \eqref{tfalkwedgebe},
\begin{align}\label{maxdegen}
\begin{split}
\sB(\al) & = \tfrac{n-2}{n-3}\tf(\al \kwedge \al)  = \tfrac{n-2}{n-1}\al \kwedge \al - \tfrac{2}{n-1}\al \kwedge \hat{\al} + \tfrac{2}{(n-1)(n-3)}\hat{\al}\kwedge \hat{\al} \\
& = \tfrac{n-2}{n-1}(x\dwedge y)\tensor (x\dwedge y) + \tfrac{2}{(n-1)(n-3)}\sum_{1 \leq \al < \be \leq n-2}(z^{(\al)}\dwedge z^{(\be)})\tensor (z^{(\al)}\dwedge z^{(\be)})\\
&\quad  - \tfrac{1}{n-1}\sum_{\al = 1}^{n-2}\left( (x \dwedge z^{(\al)})\tensor (x \dwedge z^{(\al)}) + (y \dwedge z^{(\al)})\tensor (y \dwedge z^{(\al)}) \right)\\&
  = \tfrac{3}{(n-1)(n-3)}\sum_{1 \leq \al < \be \leq n-2}\sB(x, y, z^{(\al)}, z^{(\be)}).
\end{split}
\end{align}
By Corollary \ref{kwedgeidempotentcorollary}, if $n > 3$, $\sB(\al)$ is a nontrivial idempotent in $(\mcurvweyl(\std), \opmult)$. 

Let $(\ste, h)$ be a metric vector space of dimension $n \geq 4$. In accord with \cite[Definition $1.2$]{Doubrov-The}, a nonzero $\sX \in \mcurvweyl(\std)$ is \emph{maximally degenerate} if the stabilizer $O(h)_{[\sX]}$ in $O(h)$  of $[\sX] \in \proj(\mcurvweyl(\std))$ has maximal dimension among subgroups of $O(h)$ stabilizing a point in $\proj(\mcurvweyl(\std))$. Equivalently, the $O(h)$-orbit of $\sX$ has the minimal dimension possible among nontrivial $O(h)$-orbits in $\mcurvweyl(\std)$. By \cite[Equation $(3.2)$ and Appendix B]{Doubrov-The}, when $h$ is Euclidean, if $n = 5$ or $n \geq 7$ the tensor $\sB(\al)$ of \eqref{maxdegen} is maximally degenerate with $\dim O(n)_{[\sB(\al)]} = \binom{n-2}{2} + 1$ and $O(n)$-orbit of dimension $2n-4$. Moreover, by \cite[Theorem $1.3$]{Doubrov-The}, if $n \geq 7$, any maximally degenerate element of $\mcurvweyl(\std)$ is in the $O(n)$-orbit of a nonzero multiple of $\sB(\al)$ and any element of $\proj(\mcurvweyl(\std))$ stabilized by $\dim O(n)_{[\sB(\al)]}$ equals $[\sB(\al)]$ (this gives an alternative proof that some multiple of $\sB(\al)$ is idempotent, for $[\sB(\al)\opmult \sB(\al)]$ is stabilized by $\dim O(n)_{[\sB(\al)]}$, so must equal $[\sB(\al)]$). The cases $n \in \{4, 5, 6\}$ require special treatment; if $n = 5$ there are maximally degenerate tensors not orthogonally equivalent to a multiple of $\sB(\al)$, while if $n \in \{4, 6\}$, $\sB(\al)$ is not maximally degenerate. In these cases the maximally degenerate tensors can be built from Kähler-Weyl tensors; see \cite{Doubrov-The}.

If both inertial indices of $h$ are least $2$, there are $x$ and $y$ spanning a $2$-dimensional $h$-null subspace of $\std$, and it follows from \eqref{alalsquared} that $(x\dwedge y)\tensor (x \dwedge y)$ is a nontrivial square-zero element in $(\mcurvweyl(\std), \opmult)$. By \cite{Doubrov-The}, any maximally degenerate Weyl tensor is in the $O(h)$ orbit of a nonzero multiple of this tensor. Its stabilizer in $O(h)$ has dimension $\binom{n-2}{2} + 4$.

If $h$ has negative inertial index one, for $\sZ^{(\al)}$ as in the proof of Lemma \ref{indefiniteszlemma}, $\sZ^{A} = \sum_{\al \in A}\sZ^{(\al)}$ is square-zero for any nonempty subset $A \subset \{1, \dots, n-3\}$. The elements $\sZ^{A}$ and $\sZ^{B}$ are orthogonally equivalent if and only if $A$ and $B$ have the same cardinality, and by \cite[Theorem $4.2$]{Doubrov-The}, $\sZ^{\{1, \dots, n-3\}}$ is maximally degenerate having stabilizer of dimension $\binom{n-2}{2} + 2$.
\end{example}

\begin{example}\label{matsuoexample}
This example shows that, when $\dim \ste \geq 6$, $(\mcurvweyl(\std), \opmult)$ contains subalgebras isomorphic to $3$-dimensional Matsuo algebras. 

Straightforward computations using \eqref{alwedgebegawedgesinorm}, \eqref{symalbecmgasi}, and \eqref{sabdefined} show that if $\al(1), \al(2), \al(3) \in \idem(S^{2}\std, \jmult)$ are pairwise orthogonal idempotents of ranks $a_{1}$, $a_{2}$, and $a_{3}$, each at least $2$, and $\sB_{ij} = \sB(\al(i), \al(j))$ for $1 \leq i < j \leq 3$, there hold
\begin{align}\label{ranktworelation}
\begin{split}
2&(a_{1} + a_{2}-1)(a + a_{3} - 1)\sB_{12}\opmult \sB_{13} \\
&= (a_{3}(a_{1}+a_{2}-1)\sB_{12} + a_{2}(a_{1} + a_{3} - 1)\sB_{13} - a_{1}(a_{2}+a_{3}-1)\sB_{23},\\
\lb&\sB_{12}, \sB_{13}\ra = \tfrac{2a_{1}a_{2}a_{3}}{(a_{1} + a_{2} - 1)(a_{1} + a_{3} - 1)(a_{1} - 1)}.
\end{split}
\end{align}
Note that there must hold $n \geq a + a_{2} + a_{3} \geq 6$. 

Suppose that $a_{1} = a_{2} = a_{3} = 2$. Then \eqref{ranktworelation} becomes $\sB_{12}\opmult \sB_{13} = \tfrac{1}{3}(\sB_{12} + \sB_{13} - \sB_{23})$. Using this relation it is straightforward to check that the elements $e_{0} = \tfrac{3}{5}(\sB_{12} + \sB_{13} + \sB_{23})$, $e_{1} = \tfrac{3}{5}(\sB_{12} + \sB_{13} -\tfrac{2}{3} \sB_{23})$, $e_{2} = \tfrac{3}{5}(\sB_{12} -\tfrac{2}{3} \sB_{13} + \sB_{23})$, and $e_{3} = \tfrac{3}{5}(-\tfrac{2}{3}\sB_{12} + \sB_{13} + \sB_{23})$ are idempotents satisfying $3(e_{1} + e_{2} + e_{2}) = 4e_{0}$ and the relations
\begin{align}\label{matsuorelations}
\begin{aligned}
&e_{0}\opmult e_{i} = e_{i},& &e_{i}\opmult e_{j} = \tfrac{1}{6}(e_{i} + e_{j}- e_{i\join j}), & &1\leq i \neq j \leq 3,&
\end{aligned}
\end{align}
where $i \join j$ is the unique element of $\{1, 2, 3\}$ distinct from $i$ and $j$. From \eqref{ranktworelation} it follows that $\lb\sB_{ij}, \sB_{ik}\ra = 16/9$ and so, for $1 \leq i \neq j \leq 3$, $|e_{0}|^{2} = 48/5$, $|e_{i}|^{2} = 64/15 = \lb e_{0}, e_{i}\ra$, and $\lb e_{i}, e_{j}\ra = 32/45$. These norm calculations together with the relations \eqref{matsuorelations} show that the subalgebra $\spn\{e_{1}, e_{2}, e_{3}\} \subset (\mcurvweyl(\std), \opmult)$ is the $3$-dimensional algebra based on the Fischer space with one line $\{1, 2, 3\}$ and having parameters $\ga = 512/15$ and $\delta = 2/3$ defined by Matsuo in \cite[Sections $3.2$ and $3.3$]{Matsuo-3transpositionarxiv};\footnote{The arXiv version \cite{Matsuo-3transpositionarxiv} differs substantially from the published version \cite{Matsuo-3transposition}; see also \cite[Sections $2$ and $5$]{Matsuo}).} this algebra is denoted $M(\{1, 2, 3\}, \tfrac{1}{6}, \rea)$ in the notations of \cite{Hall-Rehren-Shpectorov-primitive}.
\end{example}

\section*{Acknowledgements}
I thank Vladimir Tkachev for his interest in this work, for helpful comments on the first version, and for giving me access to related work in preparation. I thank Mamuka Jibladze for comments (unrelated to this paper) that brought to my attention \cite{Elashvili} and \cite{Popov}.

\bibliographystyle{amsplain}

\begin{thebibliography}{10}

\bibitem{Besse}
A.~L. Besse, \emph{Einstein manifolds}, Ergebnisse der Mathematik und ihrer
  Grenzgebiete (3), vol.~10, Springer-Verlag, Berlin, 1987.

\bibitem{Bohm-Wilking}
C.~B\"ohm and B.~Wilking, \emph{Manifolds with positive curvature operators are
  space forms}, Ann. of Math. (2) \textbf{167} (2008), no.~3, 1079--1097.

\bibitem{Bordemann}
M.~Bordemann, \emph{Nondegenerate invariant bilinear forms on nonassociative
  algebras}, Acta Math. Univ. Comenian. (N.S.) \textbf{66} (1997), no.~2,
  151--201.

\bibitem{Brendle-book}
S.~Brendle, \emph{Ricci flow and the sphere theorem}, Graduate Studies in
  Mathematics, vol. 111, American Mathematical Society, Providence, RI, 2010.

\bibitem{Brendle-surgeryisotropic}
\bysame, \emph{Ricci flow with surgery on manifolds with positive isotropic
  curvature}, Ann. of Math. (2) \textbf{190} (2019), no.~2, 465--559.

\bibitem{Dixmier-algebres}
J.~Dixmier, \emph{Certaines alg\`ebres non associatives simples d\'efinies par
  la transvection des formes binaires}, J. Reine Angew. Math. \textbf{346}
  (1984), 110--128.

\bibitem{Doubrov-The}
B.~Doubrov and D.~The, \emph{Maximally degenerate {W}eyl tensors in
  {R}iemannian and {L}orentzian signatures}, Differential Geom. Appl.
  \textbf{34} (2014), 25--44.

\bibitem{Elashvili}
A.~G. \`Elashvili, \emph{Invariant algebras}, Lie groups, their discrete
  subgroups, and invariant theory, Adv. Soviet Math., vol.~8, Amer. Math. Soc.,
  Providence, RI, 1992, pp.~57--64.

\bibitem{Faraut-Koranyi}
J.~Faraut and A.~Kor{\'a}nyi, \emph{Analysis on symmetric cones}, Oxford
  Mathematical Monographs, The Clarendon Press Oxford University Press, New
  York, 1994, Oxford Science Publications.

\bibitem{Fox-simplicial}
D.~J.~F. Fox, \emph{Commutative algebras with nondegenerate invariant trace
  form and trace-free multiplication endomorphisms},
  \href{https://arxiv.org/abs/2004.12343}{arXiv:2004.12343}.

\bibitem{Fox-curvtensorkahler}
\bysame, \emph{The commutative nonassociative algebra of {K}ähler curvature
  tensors}, In preparation.

\bibitem{Griess-Monster}
R.~L. Griess, Jr., \emph{The {M}onster and its nonassociative algebra}, Finite
  groups---coming of age ({M}ontreal, {Q}ue., 1982), Contemp. Math., vol.~45,
  Amer. Math. Soc., Providence, RI, 1985, pp.~121--157.

\bibitem{Griess-gnavoai}
\bysame, \emph{G{NAVOA}. {I}. {S}tudies in groups, nonassociative algebras and
  vertex operator algebras}, Vertex operator algebras in mathematics and
  physics ({T}oronto, {ON}, 2000), Fields Inst. Commun., vol.~39, Amer. Math.
  Soc., Providence, RI, 2003, pp.~71--88.

\bibitem{Hall-transpositionalgebras}
J.~I. Hall, \emph{Transposition algebras}, Bull. Inst. Math. Acad. Sin. (N.S.)
  \textbf{14} (2019), no.~2, 155--187.

\bibitem{Hall-Rehren-Shpectorov-primitive}
J.~I. Hall, F.~Rehren, and S.~Shpectorov, \emph{Primitive axial algebras of
  {J}ordan type}, J. Algebra \textbf{437} (2015), 79--115.

\bibitem{Hall-Rehren-Shpectorov}
\bysame, \emph{Universal axial algebras and a theorem of {S}akuma}, J. Algebra
  \textbf{421} (2015), 394--424.

\bibitem{Hamilton}
R.~S. Hamilton, \emph{Three-manifolds with positive {R}icci curvature}, J.
  Differential Geom. \textbf{17} (1982), no.~2, 255--306.

\bibitem{Hamilton-four}
\bysame, \emph{Four-manifolds with positive curvature operator}, J.
  Differential Geom. \textbf{24} (1986), no.~2, 153--179.

\bibitem{Hamilton-formation}
\bysame, \emph{The formation of singularities in the {R}icci flow}, Surveys in
  differential geometry, {V}ol.\ {II} ({C}ambridge, {MA}, 1993), Int. Press,
  Cambridge, MA, 1995, pp.~7--136.

\bibitem{Andrews-Hopper}
C.~Hopper and B.~Andrews, \emph{The {R}icci flow in {R}iemannian geometry},
  Lecture Notes in Mathematics, vol. 2011, Springer, Heidelberg, 2011, A
  complete proof of the differentiable 1/4-pinching sphere theorem.

\bibitem{Huisken}
G.~Huisken, \emph{Ricci deformation of the metric on a {R}iemannian manifold},
  J. Differential Geom. \textbf{21} (1985), no.~1, 47--62.

\bibitem{Ivanov}
A.~A. Ivanov, \emph{The {M}onster group and {M}ajorana involutions}, Cambridge
  Tracts in Mathematics, vol. 176, Cambridge University Press, Cambridge, 2009.

\bibitem{Jacobson-jordan}
N.~Jacobson, \emph{Structure and representations of {J}ordan algebras},
  American Mathematical Society Colloquium Publications, Vol. XXXIX, American
  Mathematical Society, Providence, R.I., 1968.

\bibitem{Koecher}
M.~Koecher, \emph{The {M}innesota notes on {J}ordan algebras and their
  applications}, Lecture Notes in Mathematics, vol. 1710, Springer-Verlag,
  Berlin, 1999.

\bibitem{Lovelock}
D.~Lovelock, \emph{Dimensionally dependent identities}, Proc. Cambridge Philos.
  Soc. \textbf{68} (1970), 345--350.

\bibitem{Matsuo-3transpositionarxiv}
A.~Matsuo, \emph{3-transposition groups of symplectic type and vertex operator
  algebras}, \href{https://arxiv.org/abs/math/0311400v1}{arXiv:math/0311400v1}.

\bibitem{Matsuo}
\bysame, \emph{Norton's trace formulae for the {G}riess algebra of a vertex
  operator algebra with larger symmetry}, Comm. Math. Phys. \textbf{224}
  (2001), no.~3, 565--591.

\bibitem{Matsuo-3transposition}
\bysame, \emph{3-transposition groups of symplectic type and vertex operator
  algebras}, J. Math. Soc. Japan \textbf{57} (2005), no.~3, 639--649.

\bibitem{Nadirashvili-Tkachev-Vladuts}
N.~Nadirashvili, V.~Tkachev, and S.~Vl\u{a}du\c{t}, \emph{Nonlinear elliptic
  equations and nonassociative algebras}, Mathematical Surveys and Monographs,
  vol. 200, American Mathematical Society, Providence, RI, 2014.

\bibitem{Penrose-Rindler}
R.~Penrose and W.~Rindler, \emph{Spinors and space-time: {V}ol. 1, {T}wo-spinor
  calculus and relativistic fields}, Cambridge Monographs on Mathematical
  Physics, Cambridge University Press, Cambridge, 1984.

\bibitem{Popov}
V.~L. Popov, \emph{An analogue of {M}. {A}rtin's conjecture on invariants for
  nonassociative algebras}, Lie groups and {L}ie algebras: {E}. {B}. {D}ynkin's
  {S}eminar, Amer. Math. Soc. Transl. Ser. 2, vol. 169, Amer. Math. Soc.,
  Providence, RI, 1995, pp.~121--143.

\bibitem{Richard-curvaturecones}
T.~Richard, \emph{{Curvature cones and the Ricci flow}}, {Actes de S\'eminaire
  de Th\'eorie Spectrale et G\'eom\'etrie. Ann\'ee 2012--2014}, St. Martin
  d'H\`eres: Universit\'e de Grenoble I, Institut Fourier, 2014, pp.~197--220.

\bibitem{Richard-lowerbounds}
\bysame, \emph{Lower bounds on {R}icci flow invariant curvatures and geometric
  applications}, J. Reine Angew. Math. \textbf{703} (2015), 27--41.

\bibitem{Richard-Seshadri}
T.~Richard and H.~Seshadri, \emph{Non-coercive {R}icci flow invariant curvature
  cones}, Proc. Amer. Math. Soc. \textbf{143} (2015), no.~6, 2661--2674.

\bibitem{Ryba}
A.~J.~E. Ryba, \emph{A natural invariant algebra for the {H}arada-{N}orton
  group}, Math. Proc. Cambridge Philos. Soc. \textbf{119} (1996), no.~4,
  597--614.

\bibitem{Tkachev-laplace}
V.~G. Tkachev, \emph{On the non-vanishing property for real analytic solutions
  of the {$p$}-{L}aplace equation}, Proc. Amer. Math. Soc. \textbf{144} (2016),
  no.~6, 2375--2382.

\bibitem{Tkachev-correction}
\bysame, \emph{A correction of the decomposability result in a paper by
  {M}eyer-{N}eutsch}, J. Algebra \textbf{504} (2018), 432--439.

\bibitem{Tkachev-universality}
\bysame, \emph{The universality of one half in commutative nonassociative
  algebras with identities}, J. Algebra \textbf{569} (2021), 466--510.

\bibitem{Tricerri-Vanhecke}
F.~Tricerri and L.~Vanhecke, \emph{Curvature tensors on almost {H}ermitian
  manifolds}, Trans. Amer. Math. Soc. \textbf{267} (1981), no.~2, 365--397.

\bibitem{Weyl}
H.~Weyl, \emph{The classical groups. {T}heir invariants and representations},
  Princeton University Press, Princeton, N.J., 1939.

\bibitem{Wilking}
B.~Wilking, \emph{A {L}ie algebraic approach to {R}icci flow invariant
  curvature conditions and {H}arnack inequalities}, J. Reine Angew. Math.
  \textbf{679} (2013), 223--247.

\end{thebibliography}
\def\polhk#1{\setbox0=\hbox{#1}{\ooalign{\hidewidth
  \lower1.5ex\hbox{`}\hidewidth\crcr\unhbox0}}} \def\cprime{$'$}
  \def\cprime{$'$} \def\cprime{$'$}
  \def\polhk#1{\setbox0=\hbox{#1}{\ooalign{\hidewidth
  \lower1.5ex\hbox{`}\hidewidth\crcr\unhbox0}}} \def\cprime{$'$}
  \def\cprime{$'$} \def\cprime{$'$} \def\cprime{$'$} \def\cprime{$'$}
  \def\polhk#1{\setbox0=\hbox{#1}{\ooalign{\hidewidth
  \lower1.5ex\hbox{`}\hidewidth\crcr\unhbox0}}} \def\cprime{$'$}
  \def\Dbar{\leavevmode\lower.6ex\hbox to 0pt{\hskip-.23ex \accent"16\hss}D}
  \def\cprime{$'$} \def\cprime{$'$} \def\cprime{$'$} \def\cprime{$'$}
  \def\cprime{$'$} \def\cprime{$'$} \def\cprime{$'$} \def\cprime{$'$}
  \def\cprime{$'$} \def\cprime{$'$} \def\cprime{$'$} \def\dbar{\leavevmode\hbox
  to 0pt{\hskip.2ex \accent"16\hss}d} \def\cprime{$'$} \def\cprime{$'$}
  \def\cprime{$'$} \def\cprime{$'$} \def\cprime{$'$} \def\cprime{$'$}
  \def\cprime{$'$} \def\cprime{$'$} \def\cprime{$'$} \def\cprime{$'$}
  \def\cprime{$'$} \def\cprime{$'$} \def\cprime{$'$} \def\cprime{$'$}
  \def\cprime{$'$} \def\cprime{$'$} \def\cprime{$'$} \def\cprime{$'$}
  \def\cprime{$'$} \def\cprime{$'$} \def\cprime{$'$} \def\cprime{$'$}
  \def\cprime{$'$} \def\cprime{$'$} \def\cprime{$'$} \def\cprime{$'$}
  \def\cprime{$'$} \def\cprime{$'$} \def\cprime{$'$} \def\cprime{$'$}
  \def\cprime{$'$} \def\cprime{$'$} \def\cprime{$'$}
\providecommand{\bysame}{\leavevmode\hbox to3em{\hrulefill}\thinspace}
\providecommand{\MR}{\relax\ifhmode\unskip\space\fi MR }
\providecommand{\MRhref}[2]{%
  \href{http://www.ams.org/mathscinet-getitem?mr=#1}{#2}
}
\providecommand{\href}[2]{#2}

\end{document}